\documentclass[11pt]{article}
\usepackage{etex}
\usepackage{tikz}
\tikzset{node distance=2.5cm, auto}
\usepackage{amssymb, amsmath}
\usepackage{latexsym}
\usepackage{amsthm}
\usepackage{color}
\usepackage{mathrsfs}
\usepackage[dvips]{epsfig}
\usepackage[all]{xy}
\usepackage{hyperref}

\usepackage{graphics,graphicx}

\usepackage[usenames,dvipsnames]{pstricks}
 \usepackage{pst-grad} 
\usepackage{pst-plot}
\usepackage[latin1]{inputenc}
\usepackage{psfrag}
\usepackage{amssymb}
\usepackage{latexsym}
\usepackage[dvips]{epsfig}
\newcommand{\casikcuatro}{
\scalebox{0.7} 
{
\begin{pspicture}(0,-0.25)(0.5399996,0.25)
\psdots[dotsize=0.12](0.06,0.17)
\psdots[dotsize=0.12](0.4599996,0.17)
\psdots[dotsize=0.12](0.06,-0.17)
\psdots[dotsize=0.12](0.4599996,-0.17)
\psline[linewidth=0.04cm](0.06,0.15)(0.4399996,-0.15)
\psline[linewidth=0.04cm](0.4599996,0.13)(0.4599996,-0.13)
\psline[linewidth=0.04cm](0.08,0.17)(0.4199997,0.17)
\psline[linewidth=0.04cm](0.06,0.13)(0.06,-0.15)
\psline[linewidth=0.04cm](0.08,-0.17)(0.4199997,-0.17)
\end{pspicture}}}
\newcommand{\arbolitoizq}{\scalebox{0.7} 
{
\begin{pspicture}(0,-0.32)(0.44,0.32)
\psdots[dotsize=0.12](0.06,0.24)
\psdots[dotsize=0.12](0.06,0.0)
\psdots[dotsize=0.12](0.22,-0.24)
\psdots[dotsize=0.12](0.36,0.0)
\psline[linewidth=0.04cm](0.06,0.2)(0.06,0.02)
\psline[linewidth=0.04cm](0.08,-0.02)(0.2,-0.24)
\psline[linewidth=0.04cm](0.22,-0.26)(0.36,0.0)
\end{pspicture}
} }
\newcommand{\arbolitoder}{\scalebox{0.7}{\begin{pspicture}(0,-0.32)(0.44,0.32)
\psdots[dotsize=0.12](0.36,0.24)
\psdots[dotsize=0.12](0.36,0.0)
\psdots[dotsize=0.12](0.2,-0.24)
\psdots[dotsize=0.12](0.06,0.0)
\psline[linewidth=0.04cm](0.36,0.2)(0.36,0.02)
\psline[linewidth=0.04cm](0.34,-0.02)(0.22,-0.24)
\psline[linewidth=0.04cm](0.2,-0.26)(0.06,0.0)
\end{pspicture} 
}}
\newcommand{\ldos}{\scalebox{0.7} 
{
\begin{pspicture}(0,-0.2)(0.14,0.2)
\psdots[dotsize=0.12](0.06,0.12)
\psdots[dotsize=0.12](0.06,-0.12)
\psline[linewidth=0.04cm](0.06,0.08)(0.06,-0.1)
\end{pspicture}
} }
\newcommand{\ltres}{\scalebox{0.7} 
{
\begin{pspicture}(0,-0.32)(0.16,0.3)
\psdots[dotsize=0.12](0.08,0.0)
\psdots[dotsize=0.12](0.08,-0.24)
\psline[linewidth=0.04cm](0.08,-0.04)(0.08,-0.22)
\psdots[dotsize=0.12](0.08,0.22)
\psline[linewidth=0.04cm](0.08,0.22)(0.08,0.04)
\end{pspicture}
} }
\newcommand{\cdos}{\scalebox{0.7} 
{
\begin{pspicture}(0,-0.21)(0.46,0.19)
\psdots[dotsize=0.12](0.38,0.11)
\psdots[dotsize=0.12](0.08,0.11)
\psline[linewidth=0.04cm](0.36,0.09)(0.24,-0.13)
\psline[linewidth=0.04cm](0.22,-0.15)(0.08,0.11)
\psdots[dotsize=0.12](0.22,-0.13)
\end{pspicture}
} }

\newcommand{\Bc}{\mathscr{B}_c}

\newcommand{\PP}{\mathbb P}

\newcommand{\Eb}{E^{\bullet}}
\theoremstyle{plain} \theoremstyle{remark}
\newtheorem{theo}{\textbf{Theorem}}
\newtheorem{remark}{ \textbf{Remark}}
\newtheorem{defi}{ \textbf{Definition}}

\newtheorem{coro}{\textbf{Corollary}}
\newtheorem{prop}{ \textbf{Proposition}}

\newtheorem{ex}{\textbf{Example}}

\newcommand{\Iv}{\mathrm{Iv}}
\newcommand{\Ep}{E^{\bullet}}
\newcommand{\lacito}{\scalebox{0.7} 
{\begin{pspicture}(0,-0.2641111)(0.76822263,0.2641111)
\psline[linewidth=0.04cm](0.064580314,0.1884375)(0.34458032,0.0284375)
\psline[linewidth=0.04cm](0.064580314,-0.1915625)(0.34458032,-0.0315625)
\psline[linewidth=0.04cm](0.40458032,0.0484375)(0.6845803,0.1884375)
\psline[linewidth=0.04cm](0.4245803,-0.0315625)(0.6845803,-0.2115625)
\psdots[dotsize=0.12](0.6941115,0.19)
\psdots[dotsize=0.12](0.07411111,-0.19)
\psdots[dotsize=0.12](0.074111514,0.19)
\psdots[dotsize=0.12](0.6941111,-0.19)
\psline[linewidth=0.04cm](0.064580314,0.1484375)(0.064580314,-0.1515625)
\psline[linewidth=0.04cm](0.6941115,0.15)(0.6845803,-0.2115625)
\psdots[dotsize=0.22093812,linecolor=blue](0.37411126,-0.01)
\end{pspicture}}}
\newcommand{\triangulito}{
\scalebox{0.7} 
{\begin{pspicture}(0,-0.2641111)(0.4845803,0.2641111)
\psline[linewidth=0.04cm](0.0504692,0.1884375)(0.3304692,0.0284375)
\psline[linewidth=0.04cm](0.0504692,-0.1915625)(0.3304692,-0.0315625)
\psdots[dotsize=0.12](0.06,-0.19)
\psdots[dotsize=0.12](0.0600004,0.19)
\psline[linewidth=0.04cm](0.0504692,0.1484375)(0.0504692,-0.1515625)
\psdots[dotsize=0.22093812,linecolor=blue](0.36000013,-0.01)
\end{pspicture}}}
\newcommand{\HCK}{\mathcal{H}_{{\mathrm CK}}}
\newcommand{\Gr}{\mathscr{G}_c}
\newcommand{\Grp}{\mathscr{G}^{\bullet}_c}
\newcommand{\Arb}{\mathscr{A}}

\newcommand{\BB}{\mathbb{B}}
\newcommand{\FF}{\mathbb{F}}
\newcommand{\KK}{\mathbb{K}}

\newcommand{\Tt}{\mathbb{T}}
\newcommand{\T}{\mathcal{T}}
\newcommand{\TT}{\mathscr{T}}
\newcommand{\admc}{\mathrm{Ac}}

\newcommand{\be}{\begin{equation}}
\newcommand{\eeq}{\end{equation}}

\newcommand{\CN}{\mathcal{N}}
\newcommand{\CH}{\mathcal{H}}
\newcommand{\CF}{\mathcal{F}}

\newcommand{\NA}{\mathbb{K}[t_{\alpha}:\alpha\in\mathbb{T}(M_{2^+})]}
\newcommand{\aut}{\mathrm{aut}}

\begin{document}
\title{An antipode formula for the natural Hopf algebra of a set operad.}

\author{Miguel A. M\'endez 
\\Departamento de Matem\'atica\\
\small \emph{Instituto Venezolano de Investigaciones
Cient\'\i ficas}\\
 \small \emph{A.P. 21827, Caracas 1020--A, Venezuela}\\
\small {email:mmendezenator@gmail.com}\\\small {phone:582125041412.}\\ and\\
\\Jean Carlos Liendo 
  \small \emph{Departamento de Matemática}\\ \emph{Facultad de Ciencias}\\ \emph{Universidad Central de Venezuela} \\
  \small \emph{Caracas 1020-Venezuela.}} \maketitle

\begin{abstract} A set-operad is a monoid in the category of combinatorial species with
respect to the operation of substitution. From a set-operad, we give here a simple construction of a Hopf algebra that we call {\em the natural Hopf algebra} of the operad.  We obtain a combinatorial formula for its antipode in terms of Shr\"oder trees,
generalizing the Hayman-Schmitt formula for the Faá di Bruno Hopf algebra. From there we derive more readable formulas for specific operads. The classical Lagrange inversion formula is obtained in this way from the set-operad of pointed sets. 
We also derive antipodes formulas for the natural Hopf algebra corresponding to the operads of connected graphs, the NAP operad, and for its generalization, the set-operad of trees enriched with a monoid. When the set operad  is left
cancellative, we can construct a family of posets. The natural Hopf algebra is then obtained as an incidence reduced Hopf algebra,
by taking a suitable equivalence relation over the intervals of that family of posets.
We also present a simple combinatorial construction of an epimorphism from the natural Hopf algebra corresponding to the NAP operad, to the Connes and Kreimer Hopf
algebra. \end{abstract}

Keywords: Hopf algebras, posets,  operads, species.
\section{Introduction}
Assembling and disassembling combinatorial structures was stated by Joni and Rota \cite{Joni-Rota} to be the paradigm for defining products and coproducts in combinatorial bialgebras. Set operads are families of combinatorial structures where this paradigm can be completely fulfilled. 

The operation of substitution of species formalizes the notion, present in many combinatorial constructions, of placing combinatorial  objects  `inside' other combinatorial objects.
Hence, the intuitive and informal description of a combinatorial species as a family of combinatorial structures that is closed by relabelling could be extended to describe a set operad. For this end it is appropriate to think of a set operad $M$ as a such family together with a self reproducing recipe $\eta:M(M)\rightarrow M$ that assembles a set of structures (pieces)  using an external structure (pattern) in order to obtain a bigger one. This recipe satisfies axioms of associativity and existence of identity. 
 
When the operad also satisfies the left cancellation law, the relation obtained by splitting sets of structures (assemblies) into smaller ones by the use of $\eta$ is a partial order. A set operad satisfying a left cancellation law is called a cancellative operad ($c$-operad). 

The construction of families of partially ordered sets from $c$-operads and the {\em natural incidence Hopf algebras}  associated to those families were introduced in \cite{PhdMendez} (see also \cite{MendezandYang, Bruno1, MendezKoszulduality, Bruno2}, for the construction of posets from $c$-operads, and \cite{Chapoton-Livernet} for a similar (independent) construction of incidence Hopf algebras from operads).
The intervals of these posets form an hereditary family in the sense of Schmitt \cite{Schmitt2}, and any Hopf relation on them  gives rise to a reduced incidence Hopf algebra. The natural relation $\sim$ gives rise to $\CN_M$, the natural reduced incidence Hopf algebra, and the coarser isomorphism relation 
$\equiv$ to the standard reduced incidence Hopf algebra $\CH_M$ \cite{Chapoton-Livernet}. The present approach gives at once an epimorhism from $\CN_M$ to  $\CH_M$. Since in \cite{Chapoton-Livernet}  the standard reduced incidence Hopf algebra associated to the NAP operad was proved to be isomorphic to the Connes and Kreimer Hopf algebra, we obtain an epimorphism $\CN_{\Arb}\rightarrow \HCK$. We give here a direct explicit combinatorial construction of it.

However, the requirement of the operad to be cancellative is not essential for the definition of the natural Hopf algebra. A similar general definition of Hopf algebras from algebraic (non-symmetric) cooperads was given in \cite{Vanderlaan}. See also \cite{Chapoton-Livernet} for an equivalent one as the Hopf algebra of functions of a group of formal series associated to the operad.
In this article we give a simple construction of the natural Hopf algebra from an arbitrary set operad, and a general combinatorial antipode formula based on signed enriched Schr\"oder trees \cite{Schroder}. Although not stated precisely in these terms, our definition of the natural Hopf algebra actually rests in the cooperad structure of the dual $M^*$. This can be done since, as (vector) species with distinguished basis, $M^{*}$ and $M$ are isomorphic. For reasons of simplicity, in this article we state the definition of coproduct in the natural coalgebra without the use of the notion of cooperad. In a forthcoming paper, we shall extend the results presented here to more general cooperads.   

It is worth mentioning that the device of Schr\"oder trees in our antipode formula is the same as in the bar and cobar constructions for the definition of Koszulness for quadratic operads (see \cite{Ginzburg-Kapranov, Getzler, Bruno2}). For sure not a coincidence, it remains an intriguing fact as yet to be fully understood.

Using the technique of coloring the internal nodes of the Shr\"oder trees with its depth and other bijections, we obtain for each of the examples studied here, particular and more readable forms of our general antipode formula.  All these  examples of operads and Hopf algebras were already introduced in \cite{PhdMendez} and the present article closes questions unanswered for more than twenty years.

\section{Set operads and set monoids}
Recall that a (combinatorial) species is a functor from the category of finite sets and
bijections $\BB$, to the category $\FF$ of finite sets and arbitrary functions \cite{Joyal1, B-L-L}. The
combinatorial species with the natural transformation as morphisms, form a category. We
will say that $M=N$ if $M$ and $N$ are isomorphic species.

An element $m_U\in M[U]$ is called an $M$-structure over the set of labels $U$. Two
$M$-structures $m_U$ and $m'_V$ are said to be isomorphic if there exists a bijection
$f:U\rightarrow V$ such that $M[f]m_U=m'_V$. The bijection $f$ is called an isomorphism
$m_U\stackrel{f}{\rightarrow}m'_V$. The isomorphism types of $M$-structures are usually
also called {\em unlabelled} $M$-structures. We denote by $\Tt(M)$ the set of isomorphism
types of $M$-structures. For an $M$-structure $m$, $\tau_M(m)=\tau(m)\in \Tt(M)$ denotes the isomorphism type of $m$.

A species sending the empty set to the empty set is called {\em positive}. A positive
species sending singleton sets to singleton sets is called a {\em delta} species. For a species $M$, $M_+$ will denote the positive species from $M$
$$M_+[U]=\begin{cases}M[U]&\mbox{ if $U\neq\emptyset$}\\\emptyset &\mbox{otherwise}.\end{cases}$$
More generally, for a positive integer $k$, $M_{k^+}$ denotes the species 
$$M_{k^+}[U]=\begin{cases}M[U]&\mbox{ if $|U|\geq k$}\\\emptyset &\mbox{otherwise}.\end{cases}$$
 Given
two species $M$ and $N$, the operations of product and substitution are
defined respectively by 
\be M.N[U]=\coprod_{U_1+U_2=U}M[U_1]\times N[U_2]\eeq
\noindent and
\be M(N)[U]=\coprod_{\pi\in \Pi[U]} \prod_{B\in\pi}N[B]\times M[\pi],\eeq
\noindent where $\Pi[U]$ is the set of set partitions of the finite set $U$, and $N$ is assumed to be positive in the definition of substitution. The elements
of $M(N)[U]$ are pairs of the form $(a,m'_{\pi})$, where $a=\{n_B|B\in \pi\}$, an assembly of
$N$-structures, is an element of $\prod_{B\in\pi}N[B]$, and $m'_{\pi}$ is an element of
$M[\pi]$. The set partition $\pi$ is called the partition subjacent to the assembly $a$.
The species\be 1[U]= \begin{cases}\{U\}&\text{if $U=\emptyset$}\\\emptyset &
\text{otherwise},\end{cases} \eeq is the identity with respect to the product $1.M=M.1=M$.
The singleton species: \be X[U]=\begin{cases}U&\text{if $|U|=1$}\\\emptyset &
\text{otherwise},\end{cases} \eeq

\noindent is the identity for the operation of substitution, $M(X)=X(M)=M$, $M$ being a
positive species. The class of species with the product is a monoidal category. A monoid in this category will be called a {\em set monoid}, or shorter: a monoid. A set monoid is then a species $N$ together with two morphisms: the identity $\iota:1\rightarrow N$ and an associative product $\nu:N.N\rightarrow N$. We assume here that the $M[\emptyset]$ is a unitary set, and hence the morphism $\iota$ is trivially defined. The species $E$ of sets has a natural monoidal structure $\eta:E.E\rightarrow E$ sending a pair $(U_1,U_2)$ to the union $U_1+U_2$. The species $L$ of linear orders has a natural monoidal structure $\eta: L.L\rightarrow L$ being the concatenation of linear orders. See \cite{LibroMarcelo} for theoretical developments and interesting examples of the more general structure of Hopf monoids in  species.    

The class of positive species  with the operation of substitution and
$X$ as identity is a monoidal category. A monoid in this category will be called a {\em
set operad}.

A set operad is then a positive species plus two morphisms $\eta:M(M)\rightarrow M$,
$e:X\rightarrow M$, $\eta$ being and associative `product' and $e$ `choosing' the
identity in $M[U]$ for each unitary set $U$. In this article we only consider {\em
augmented} operads, i.e., those whose subjacent species is delta. For a connected operad
the identity $e$ is trivially defined and we will only specify the product $\eta$.

Recall that the derivative $M'$ and the pointing $M^{\bullet}$ of a species  $M$ are defined respectively by 
$M'[U]=M[U\uplus\{\ast\}]$ and $M^{\bullet}[U]=U\times M[U]$. The relation between both operations is the following: $M^{\bullet}=X.M'.$

There is a very interesting connection between operads and monoids, the derivative is a functorial operation that sends operads into monoids. Let $(M,\eta)$ be an operad. By the chain rule we have
$$\eta':M(M)'=M'(M).M'\rightarrow M'.$$ 
Let $i:M'=M'(X)\hookrightarrow M'(M)$ be the inmersion $i=I_{M'}(e)$. The reader can check that the natural transformation $\nu=\eta'\circ i$ is an associative product $\nu:M'.M'\rightarrow M'$, and hence $(M',\nu)$ is a monoid. For example, the derivative of the operad $E_+$ (Example \ref{Faa}) is the monoid $E$ of above.

\section{The Natural Hopf algebra}
Let $\KK$ be a field of characteristic zero and $\KK[t_{\alpha}|\alpha\in \Tt(M)]$ the
free commutative $\KK$-algebra generated by indeterminates  that are in correspondence
with the isomophism types of the structures of an operad $M$.

The purpose of this section is to construct a coalgebra structure compatible with the
algebra structure of $\KK[t_{\alpha}|\alpha\in \Tt(M)]$, $M$ being a set operad. For that
end, for an assembly of $M$-structures $a=\{m_B\}_{B\in\pi}$, let us denote by
$t_{\tau(a)}$ the monomial $\prod_{B\in\pi}t_{\tau(m_B)}$. For an indeterminate
$t_{\alpha}$, $\alpha \in \Tt(M)$, we define \be\label{Deltanat}
\Delta(t_{\alpha})=\sum_{\eta(a,m'_{\pi})=m}t_{\tau(a)}\otimes t_{\tau(m'_{\pi})},\eeq \noindent
where $m$ is some $M$-structure of type $\alpha$. By equivariance, the definition in
equation (\ref{Deltanat}) is independent of the selection of the representative $m$.

Extending $\Delta$ multiplicatively to $\mathbb{K}[t_{\alpha}|\alpha\in\Tt(M)]$, we
obtain a coproduct \be \Delta:\mathbb{K}[t_{\alpha}|\alpha\in\Tt(M)]\rightarrow
\mathbb{K}[t_{\alpha}|\alpha\in\Tt(M)]\otimes\KK[t_{\alpha}|\alpha\in\Tt(M)],\eeq
\noindent the coassociativity of $\Delta$ follows from the associativity of the operad
product $\eta$.
We denote the type of the singleton element of $M[1]$ by a singleton unlabelled node $\bullet$.
The multiplicative functional \be \epsilon(t_{\alpha})=\begin{cases} 1&\mbox{if $\alpha=\bullet$
}\\0&\mbox{otherwise}\end{cases}\eeq \noindent is a counity because
of the properties of $e$ as unity of the operad. They form a bialgebra that we call the
{\em natural bialgebra}, and denote by $\CN_{M}.$ Identifying $t_{\bullet}$ with $1$,
the unity of the algebra, $\CN_M$ becomes a Hopf algebra. As an algebra $\CN_M$ is equal
to $\NA$. $\CN_M$ is obviously commutative but in general not co-commutative.
\subsection{Examples}
\begin{ex}\label{Faa}
	Let $E$ be the uniform species defined as $E[U]=\{U\}$ for every finite set $U$. $E_+$ is an operad, $\eta_U:E_+(E_+)[U]\rightarrow E_+[U]$ being the morphism sending $(\pi,\{\pi\})$ to $\{U\}$. the linearization of $E_+$ is the $\mathrm{Comm}$ operad.
	The isomorphism types can be identified with the positive integers, and the natural Hopf algebra $\mathcal{N}_{E_+}=\mathbb{K}[t_1,t_2,t_3,\dots]$, $t_1\cong 1$,  is the Fa\'a di Bruno Hopf algebra,
	$$\Delta t_n=\sum_{\pi\in\Pi[n]} \prod_{B\in\pi}t_{|B|}\otimes t_{|\pi|}=\sum_k\left(\sum_{\substack{j_1+2j_2+3j_3+\dots=n\\
		j_1+j_2+\dots=k}}\frac{n!}{1!^{j_1}j_1!2!^{j_2}j_2!3!^{j_3}j_3!\dots}t_2^{j_2}t_3^{j_3}\dots\right
		)\otimes
		t_k.$$
\end{ex}
\begin{figure}
\scalebox{1}
{
\begin{pspicture}(0,-2.36)(15.393,2.373)
\usefont{T1}{ptm}{m}{n}
\rput(1.2528125,1.63){$b$}
\usefont{T1}{ptm}{m}{n}
\rput(0.64281255,1.63){$a$}
\usefont{T1}{ptm}{m}{n}
\rput(0.8028125,0.97){$c$}
\usefont{T1}{ptm}{m}{n}
\rput(2.2728124,0.23){$d$}
\psline[linewidth=0.03cm](5.87,1.5)(6.35,1.5)
\psline[linewidth=0.03cm](5.81,1.0)(5.79,1.46)
\pscircle[linewidth=0.03,dimen=outer](5.81,0.92){0.1}
\psline[linewidth=0.03cm](5.89,0.94)(6.35,0.94)
\psline[linewidth=0.03cm](6.41,1.02)(6.41,1.44)
\psline[linewidth=0.03cm](5.85,1.46)(6.35,0.98)
\usefont{T1}{ptm}{m}{n}
\rput(5.7528124,1.77){$a$}
\usefont{T1}{ptm}{m}{n}
\rput(6.4028125,1.77){$b$}
\usefont{T1}{ptm}{m}{n}
\rput(5.7328124,0.71){$c$}
\usefont{T1}{ptm}{m}{n}
\rput(6.4228125,0.71){$d$}
\pscircle[linewidth=0.03,dimen=outer,fillstyle=solid](0.66,0.78){0.1}
\psline[linewidth=0.03cm](0.72,1.36)(1.2,1.36)
\psline[linewidth=0.03cm](0.66,0.86)(0.64,1.32)
\psellipse[linewidth=0.03,linestyle=dashed,dash=0.16cm 0.16cm,dimen=outer](0.95,1.2)(0.69,0.7)
\psline[linewidth=0.03cm](1.5,0.8)(2.0,0.46)
\usefont{T1}{ptm}{m}{n}
\rput(3.926719,1.125){\Large ${\stackrel{\eta}{\mapsto}}$}
\pscircle[linewidth=0.03,dimen=outer](2.18,0.4){0.1}
\pscircle[linewidth=0.03,linestyle=dashed,dash=0.16cm 0.16cm,dimen=outer](2.27,0.31){0.33}
\usefont{T1}{ptm}{m}{n}
\rput(9.772813,1.04){$b$}
\usefont{T1}{ptm}{m}{n}
\rput(8.202812,1.94){$a$}
\usefont{T1}{ptm}{m}{n}
\rput(9.322812,0.38){$c$}
\usefont{T1}{ptm}{m}{n}
\rput(9.892813,0.4){$d$}
\pscircle[linewidth=0.03,dimen=outer](8.36,1.71){0.1}
\psline[linewidth=0.03cm](14.26,1.71)(14.74,1.71)
\psline[linewidth=0.03cm](14.2,1.21)(14.18,1.67)
\pscircle[linewidth=0.03,dimen=outer](14.2,1.13){0.1}
\psline[linewidth=0.03cm](14.28,1.15)(14.74,1.15)
\psline[linewidth=0.03cm](14.8,1.23)(14.8,1.65)
\psline[linewidth=0.03cm](14.24,1.67)(14.74,1.19)
\usefont{T1}{ptm}{m}{n}
\rput(14.142813,1.98){$a$}
\usefont{T1}{ptm}{m}{n}
\rput(14.792812,1.98){$b$}
\usefont{T1}{ptm}{m}{n}
\rput(14.142813,0.92){$c$}
\usefont{T1}{ptm}{m}{n}
\rput(14.812813,0.92){$d$}
\pscircle[linewidth=0.03,dimen=outer](9.78,0.77){0.1}
\psline[linewidth=0.03cm](9.78,0.29)(9.78,0.67)
\psellipse[linewidth=0.03,linestyle=dashed,dash=0.16cm 0.16cm,dimen=outer](9.47,0.61)(0.69,0.7)
\psline[linewidth=0.03cm](8.44,1.53)(8.98,1.07)
\usefont{T1}{ptm}{m}{n}
\rput(12.556719,1.43){\Large $\stackrel{\eta}{\mapsto}$}
\pscircle[linewidth=0.03,linestyle=dashed,dash=0.16cm 0.16cm,dimen=outer](8.31,1.83){0.33}
\psline[linewidth=0.03cm](9.24,0.22)(9.72,0.22)
\psline[linewidth=0.03cm](5.86,-1.08375)(6.34,-1.08375)
\psline[linewidth=0.03cm](5.88,-1.64375)(6.34,-1.64375)
\psline[linewidth=0.03cm](5.84,-1.12375)(6.34,-1.60375)
\usefont{T1}{ptm}{m}{n}
\rput(5.7428126,-0.81375){$a$}
\usefont{T1}{ptm}{m}{n}
\rput(6.3928123,-0.81375){$b$}
\usefont{T1}{ptm}{m}{n}
\rput(5.7628126,-1.87375){$c$}
\usefont{T1}{ptm}{m}{n}
\rput(6.412812,-1.87375){$d$}
\psline[linewidth=0.03cm](5.8,-1.58375)(5.78,-1.12375)
\pscircle[linewidth=0.03,dimen=outer](5.8,-1.66375){0.1}
\psline[linewidth=0.03cm](6.4,-1.56375)(6.4,-1.14375)
\usefont{T1}{ptm}{m}{n}
\rput(4.096719,-1.34375){\Large $\stackrel{\eta}{\mapsto}$}
\psellipse[linewidth=0.03,linestyle=dashed,dash=0.16cm 0.16cm,dimen=outer](1.07,-1.50375)(0.81,0.82)
\psline[linewidth=0.03cm](0.8,-1.18375)(1.28,-1.18375)
\psline[linewidth=0.03cm](0.74,-1.68375)(0.72,-1.22375)
\pscircle[linewidth=0.03,dimen=outer](0.74,-1.76375){0.1}
\psline[linewidth=0.03cm](0.82,-1.74375)(1.28,-1.74375)
\psline[linewidth=0.03cm](1.34,-1.66375)(1.34,-1.24375)
\psline[linewidth=0.03cm](0.78,-1.22375)(1.28,-1.70375)
\usefont{T1}{ptm}{m}{n}
\rput(0.6828125,-0.91375){$a$}
\usefont{T1}{ptm}{m}{n}
\rput(1.3328124,-0.87375){$b$}
\usefont{T1}{ptm}{m}{n}
\rput(0.70281243,-1.97375){$c$}
\usefont{T1}{ptm}{m}{n}
\rput(1.3328124,-1.99375){$d$}
\pscircle[linewidth=0.03,dimen=outer](9.42,-1.5607812){0.1}
\pscircle[linewidth=0.03,dimen=outer](8.62,-1.0007813){0.1}
\psline[linewidth=0.03cm](8.9,-0.9007813)(9.2,-0.9007813)
\pscircle[linewidth=0.03,dimen=outer](9.48,-0.98078126){0.1}
\psline[linewidth=0.03cm](9.48,-1.3807813)(9.48,-1.1407813)
\pscircle[linewidth=0.03,dimen=outer](8.64,-1.5607812){0.1}
\psline[linewidth=0.03cm](8.96,-1.62)(9.2,-1.62)
\psline[linewidth=0.03cm](8.86,-1.0807812)(9.3,-1.4407812)
\usefont{T1}{ptm}{m}{n}
\rput(8.522813,-0.7107813){$a$}
\usefont{T1}{ptm}{m}{n}
\rput(9.472813,-0.7107813){$b$}
\usefont{T1}{ptm}{m}{n}
\rput(8.602814,-1.7707813){$c$}
\usefont{T1}{ptm}{m}{n}
\rput(9.452812,-1.7707813){$d$}
\usefont{T1}{ptm}{m}{n}
\rput(12.296719,-1.2807811){\Large $\stackrel{\eta}{\mapsto}$}
\psline[linewidth=0.03cm](14.28,-1.0607812)(14.76,-1.0607812)
\pscircle[linewidth=0.03,dimen=outer](14.22,-1.6407813){0.1}
\psline[linewidth=0.03cm](14.3,-1.6207813)(14.76,-1.6207813)
\psline[linewidth=0.03cm](14.82,-1.5407813)(14.82,-1.1207813)
\psline[linewidth=0.03cm](14.26,-1.1007812)(14.76,-1.5807812)
\usefont{T1}{ptm}{m}{n}
\rput(14.162812,-0.79078126){$a$}
\usefont{T1}{ptm}{m}{n}
\rput(14.812813,-0.79078126){$b$}
\usefont{T1}{ptm}{m}{n}
\rput(14.182813,-1.8707813){$c$}
\usefont{T1}{ptm}{m}{n}
\rput(14.812813,-1.8707813){$d$}
\psline[linewidth=0.03cm](14.22,-1.5607812)(14.2,-1.1007812)
\psline[linewidth=0.03cm](8.62,-1.3407812)(8.62,-1.2007812)
\pscircle[linewidth=0.03,linestyle=dashed,dash=0.16cm 0.16cm,dimen=outer](9.49,-0.81){0.33}
\pscircle[linewidth=0.03,linestyle=dashed,dash=0.16cm 0.16cm,dimen=outer](9.51,-1.69){0.33}
\pscircle[linewidth=0.03,linestyle=dashed,dash=0.16cm 0.16cm,dimen=outer](8.67,-1.65){0.33}
\pscircle[linewidth=0.03,linestyle=dashed,dash=0.16cm 0.16cm,dimen=outer](8.59,-0.87){0.33}
\psframe[linewidth=0.025999999,dimen=outer](15.36,2.36)(0.0,-2.36)
\psline[linewidth=0.025999999cm](0.02,-0.16)(15.38,-0.16)
\psline[linewidth=0.025999999cm](7.26,2.36)(7.26,-2.34)
\pscircle[linewidth=0.03,dimen=outer,fillstyle=solid](1.26,1.36){0.1}
\pscircle[linewidth=0.03,dimen=outer,fillstyle=solid](0.66,1.38){0.1}
\pscircle[linewidth=0.03,dimen=outer,fillstyle=solid](0.74,-1.16375){0.1}
\pscircle[linewidth=0.03,dimen=outer,fillstyle=solid](1.34,-1.18375){0.1}
\pscircle[linewidth=0.03,dimen=outer,fillstyle=solid](1.34,-1.76375){0.1}
\pscircle[linewidth=0.03,dimen=outer,fillstyle=solid](5.8,-1.06375){0.1}
\pscircle[linewidth=0.03,dimen=outer,fillstyle=solid](6.4,-1.08375){0.1}
\pscircle[linewidth=0.03,dimen=outer,fillstyle=solid](6.4,-1.66375){0.1}
\pscircle[linewidth=0.03,dimen=outer,fillstyle=solid](6.41,1.5){0.1}
\pscircle[linewidth=0.03,dimen=outer,fillstyle=solid](6.41,0.92){0.1}
\pscircle[linewidth=0.03,dimen=outer,fillstyle=solid](5.81,1.52){0.1}
\pscircle[linewidth=0.03,dimen=outer,fillstyle=solid](9.78,0.22){0.1}
\pscircle[linewidth=0.03,dimen=outer,fillstyle=solid](9.18,0.22){0.1}
\pscircle[linewidth=0.03,dimen=outer,fillstyle=solid](14.2,1.73){0.1}
\pscircle[linewidth=0.03,dimen=outer,fillstyle=solid](14.8,1.71){0.1}
\pscircle[linewidth=0.03,dimen=outer,fillstyle=solid](14.8,1.13){0.1}
\pscircle[linewidth=0.03,dimen=outer,fillstyle=solid](14.22,-1.0407813){0.1}
\pscircle[linewidth=0.03,dimen=outer,fillstyle=solid](14.82,-1.0607812){0.1}
\pscircle[linewidth=0.03,dimen=outer,fillstyle=solid](14.82,-1.6407813){0.1}
\end{pspicture}}
\caption{Products of $\Gr.$}\label{graphproduct}
\end{figure}
\begin{ex}\label{pointed}
Let $\Ep$ be the species of pointed sets. The elements of $\Ep[U]$ are the pairs $(U,v)$, where $v$ is the distinguished element of $U$. The pointed set $(U,v)$ can be though of as a small tree (or corolla). We describe the operad structure of $\Ep$. An element of $\Ep(\Ep)[U]$ is a pair $(a,(\pi, B_0))$ where the assemmbly $a=\{(B,v_B)\mid B\in \pi\}$ is a pointed partition (a partition where each block is a pointed set) and $B_0$ is the distinguished block of the partition $\pi$
We define the product 
$$\eta(a,(\pi,B_0))=(U,v_0),$$
\noindent where $v_0$ is defined to be $v_{B_0}$, the distinguished element of the distinguished block $B_0$. The isomorphism types can be identified also with the positive integers, the subjacent algebra is also
$\mathbb{K}[t_2,t_3,\dots]$. 
Its coproduct $\Delta_{\Eb}$ is obtained as follows
\begin{eqnarray}\label{copoint}
n\Delta_{\Eb}(t_n)&=&\sum_{m\in\Eb[n]}\sum_{\eta(a,m')=m}t_{\tau(a)}\otimes
t_{\tau(m')}\\&=&\sum_{k=1}^n
\sum_{\pi\in\gamma_k(\Eb)[n]}\prod_{(B,b)\in\pi}t_{|B|}\otimes
kt_k\\&=&\sum_{k=1}^{n}B_{n,k}(0,2t_2,3t_3,\dots)\otimes k t_k.
\end{eqnarray}
$\CN_{\Eb}$ is isomorphic to the Fa\'a di Bruno Hopf algebra, making the change $t_n\leftarrow nt_n$.
\end{ex}

\begin{ex}Let $\Gr$ be the species of simple, connected graphs.
$\Gr$ is a set operad.
 For $(\{g_B\}_{B\in\pi},g'_{\pi})\in\mathscr{G}_c(\Gr)[U]$,
$\eta(\{g_B\}_{B\in\pi},g'_{\pi})$ is the graph $g$ constructed with vertices in $U$ as
follows. Keep all the edges of the graphs in the assembly, and for each edge $\{B,B'\}$
of the external graph $g'$, connect all the vertices in $B$ with all the vertices in
$B'$. In other words, $\{x,y\}$ is an edge in $g$ if one of the following two conditions
is satisfied:
\begin{enumerate} \item \{x,y\} is an edge in $g_B$, for some block $B$ of $\pi$, \item
There exist an edge $\{B,B'\}$ of $g'$, $B, B'\in \pi$, such that $x\in B$ and $y\in
B'.$ \end{enumerate} The Hopf algebra $\CN_{\Gr}$ is freely generated by the
unlabelled simple and connected graphs. As an example of the coproduct we have

\begin{equation}\label{cograph}
\Delta(t_{\casikcuatro})=t_{\casikcuatro}\otimes 1
+2t_{\ltres}\otimes t_{\ldos}+
t_{\ldos}\otimes t_{\ltres}+
1\otimes t_{\casikcuatro}.
\end{equation}

See Fig. \ref{graphproduct} for the products used to compute the coproduct in Eq. (\ref{cograph}). By considering the complete graphs, the Fa\'a Di Bruno Hopf algebra is contained in $\CN_{\Gr}$.
\end{ex}

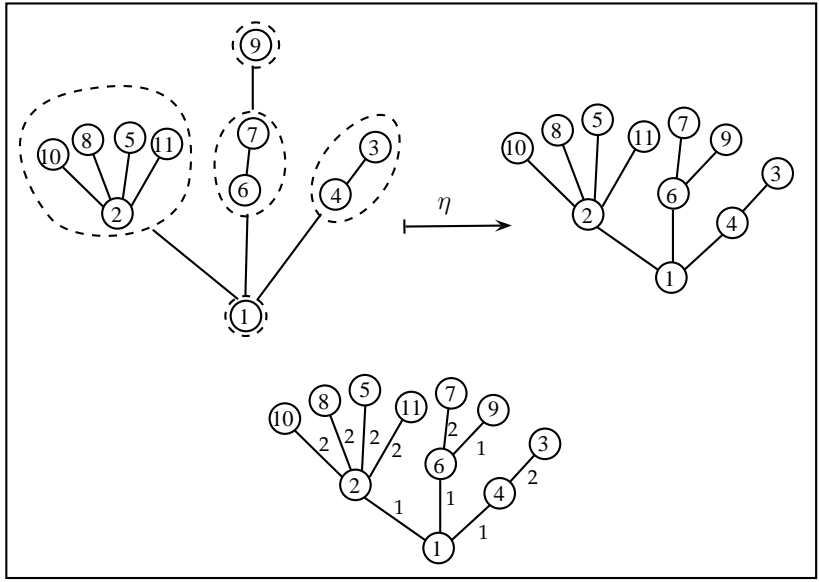
\begin{figure}\begin{center}
\scalebox{0.7} 
{
\begin{pspicture}(0,-5.48)(15.86,5.48)
\psline[linewidth=0.04cm](2.34,1.56)(1.4,2.48)
\psline[linewidth=0.04cm](2.48,1.62)(2.04,2.78)
\psline[linewidth=0.04cm](2.66,1.66)(2.84,2.94)
\psline[linewidth=0.04cm](2.76,1.48)(3.46,2.7)
\psline[linewidth=0.04cm](5.0,2.0)(5.1,2.9)
\psline[linewidth=0.04cm](6.74,1.78)(7.34,2.58)
\psbezier[linewidth=0.04,linestyle=dashed,dash=0.16cm 0.16cm](1.82,3.8)(3.62,4.16)(4.099997,3.4301543)(3.94,2.14)(3.780003,0.8498457)(1.7780365,0.8502638)(1.319457,1.36)(0.8608775,1.8697362)(0.02,3.44)(1.82,3.8)
\psbezier[linewidth=0.04,linestyle=dashed,dash=0.16cm 0.16cm](4.44,2.083335)(4.66,0.8000002)(6.0,1.552058)(5.72,2.746029)(5.44,3.94)(4.22,3.36667)(4.44,2.083335)
\psbezier[linewidth=0.04,linestyle=dashed,dash=0.16cm 0.16cm](6.2662134,1.7661103)(6.4324274,0.7200002)(7.916699,1.6139488)(7.9283495,2.8369744)(7.94,4.06)(6.1,2.8122203)(6.2662134,1.7661103)
\psellipse[linewidth=0.04,linestyle=dashed,dash=0.16cm 0.16cm,dimen=outer](5.2,4.67)(0.46,0.45)
\psellipse[linewidth=0.04,linestyle=dashed,dash=0.16cm 0.16cm,dimen=outer](5.01,-0.48)(0.41,0.4)
\psline[linewidth=0.04cm](3.24,1.16)(4.86,-0.16)
\psline[linewidth=0.04cm](5.0,-0.1)(5.04,1.46)
\psline[linewidth=0.04cm](5.14,3.44)(5.14,4.28)
\psline[linewidth=0.04cm](5.22,-0.18)(6.44,1.46)
\pscircle[linewidth=0.04,dimen=outer,fillstyle=solid](2.57,1.47){0.31}
\pscircle[linewidth=0.04,dimen=outer,fillstyle=solid](1.35,2.59){0.31}
\pscircle[linewidth=0.04,dimen=outer,fillstyle=solid](2.01,2.87){0.31}
\pscircle[linewidth=0.04,dimen=outer,fillstyle=solid](2.81,2.91){0.31}
\pscircle[linewidth=0.04,dimen=outer,fillstyle=solid](3.51,2.79){0.31}
\pscircle[linewidth=0.04,dimen=outer,fillstyle=solid](5.15,2.99){0.31}
\pscircle[linewidth=0.04,dimen=outer,fillstyle=solid](4.99,1.91){0.31}
\pscircle[linewidth=0.04,dimen=outer,fillstyle=solid](5.01,-0.48){0.31}
\pscircle[linewidth=0.04,dimen=outer,fillstyle=solid](5.2,4.67){0.31}
\usefont{T1}{ptm}{m}{n}
\rput(4.9793754,-0.485){\large 1}
\usefont{T1}{ptm}{m}{n}
\rput(2.555,1.455){\large 2}
\usefont{T1}{ptm}{m}{n}
\rput(5.1975,4.655){\large 9}
\usefont{T1}{ptm}{m}{n}
\rput(1.285625,2.555){\large 10}
\usefont{T1}{ptm}{m}{n}
\rput(3.4493752,2.775){\large 11}
\usefont{T1}{ptm}{m}{n}
\rput(1.9862515,2.855){\large 8}
\pscircle[linewidth=0.04,dimen=outer,fillstyle=solid](6.71,1.83){0.31}
\pscircle[linewidth=0.04,dimen=outer,fillstyle=solid](7.47,2.73){0.31}
\usefont{T1}{ptm}{m}{n}
\rput(7.4412503,2.715){\large 3}
\usefont{T1}{ptm}{m}{n}
\rput(6.7256246,1.815){\large 4}
\usefont{T1}{ptm}{m}{n}
\rput(4.9581246,1.915){\large 6}
\usefont{T1}{ptm}{m}{n}
\rput(5.1381264,2.955){\large 7}
\usefont{T1}{ptm}{m}{n}
\rput(2.81,2.875){\large 5}
\psline[linewidth=0.04cm,tbarsize=0.07055555cm 5.0,arrowsize=0.093cm 3.45,arrowlength=1.4,arrowinset=0.4]{|->}(8.0,1.22)(10.06,1.24)
\usefont{T1}{ptm}{m}{n}
\rput(8.781251,1.62){\Large $\eta$}
\psframe[linewidth=0.04,dimen=outer](15.86,5.48)(0.44,-5.48)
\psline[linewidth=0.04cm](8.7,-4.94)(9.66,-4.06)
\psline[linewidth=0.04cm](8.7,-4.94)(7.16,-3.86)
\psline[linewidth=0.04cm](8.7,-4.9)(8.7,-3.6)
\psline[linewidth=0.04cm](6.94,-3.64)(5.82,-2.54)
\psline[linewidth=0.04cm](7.04,-3.5)(6.6,-2.34)
\psline[linewidth=0.04cm](7.2,-3.7)(7.28,-2.14)
\psline[linewidth=0.04cm](7.36,-3.54)(8.06,-2.32)
\psline[linewidth=0.04cm](8.92,-3.14)(9.6,-2.42)
\psline[linewidth=0.04cm](8.76,-3.1)(8.86,-2.2)
\psline[linewidth=0.04cm](10.0,-3.66)(10.66,-2.94)
\pscircle[linewidth=0.04,dimen=outer,fillstyle=solid](7.27,-1.89){0.31}
\pscircle[linewidth=0.04,dimen=outer,fillstyle=solid](5.75,-2.43){0.31}
\pscircle[linewidth=0.04,dimen=outer,fillstyle=solid](6.51,-2.09){0.31}
\pscircle[linewidth=0.04,dimen=outer,fillstyle=solid](7.09,-3.69){0.31}
\pscircle[linewidth=0.04,dimen=outer,fillstyle=solid](8.15,-2.21){0.31}
\pscircle[linewidth=0.04,dimen=outer,fillstyle=solid](8.91,-1.95){0.31}
\pscircle[linewidth=0.04,dimen=outer,fillstyle=solid](9.71,-2.27){0.31}
\pscircle[linewidth=0.04,dimen=outer,fillstyle=solid](8.71,-3.29){0.31}
\pscircle[linewidth=0.04,dimen=outer,fillstyle=solid](9.83,-3.85){0.31}
\pscircle[linewidth=0.04,dimen=outer,fillstyle=solid](10.69,-2.91){0.31}
\usefont{T1}{ptm}{m}{n}
\rput(6.48625,-2.105){\large 8}
\usefont{T1}{ptm}{m}{n}
\rput(8.678123,-3.325){\large 6}
\usefont{T1}{ptm}{m}{n}
\rput(8.898123,-1.965){\large 7}
\usefont{T1}{ptm}{m}{n}
\rput(9.7175,-2.245){\large 9}
\usefont{T1}{ptm}{m}{n}
\rput(8.149376,-2.225){\large 11}
\usefont{T1}{ptm}{m}{n}
\rput(5.705624,-2.425){\large 10}
\usefont{T1}{ptm}{m}{n}
\rput(9.825624,-3.865){\large 4}
\pscircle[linewidth=0.04,dimen=outer,fillstyle=solid](8.67,-4.89){0.31}
\usefont{T1}{ptm}{m}{n}
\rput(8.639375,-4.905){\large 1}
\usefont{T1}{ptm}{m}{n}
\rput(7.075,-3.705){\large 2}
\usefont{T1}{ptm}{m}{n}
\rput(10.661249,-2.905){\large 3}
\usefont{T1}{ptm}{m}{n}
\rput(7.27,-1.885){\large 5}
\usefont{T1}{ppl}{m}{n}
\rput(7.4609375,-2.79){2}
\usefont{T1}{ppl}{m}{n}
\rput(6.9809375,-2.77){2}
\usefont{T1}{ppl}{m}{n}
\rput(6.5009375,-2.91){2}
\usefont{T1}{ppl}{m}{n}
\rput(8.940937,-2.63){2}
\usefont{T1}{ppl}{m}{n}
\rput(7.9239063,-4.11){1}
\usefont{T1}{ppl}{m}{n}
\rput(8.903906,-3.95){1}
\usefont{T1}{ppl}{m}{n}
\rput(9.483906,-2.99){1}
\usefont{T1}{ppl}{m}{n}
\rput(9.523907,-4.59){1}
\usefont{T1}{ppl}{m}{n}
\rput(10.4609375,-3.55){2}
\usefont{T1}{ppl}{m}{n}
\rput(7.8809376,-3.03){2}
\psline[linewidth=0.04cm](13.12,0.2)(14.08,1.08)
\psline[linewidth=0.04cm](13.12,0.2)(11.58,1.28)
\psline[linewidth=0.04cm](13.12,0.24)(13.12,1.54)
\psline[linewidth=0.04cm](11.36,1.5)(10.24,2.6)
\psline[linewidth=0.04cm](11.46,1.64)(11.02,2.8)
\psline[linewidth=0.04cm](11.62,1.44)(11.7,3.0)
\psline[linewidth=0.04cm](11.78,1.6)(12.48,2.82)
\psline[linewidth=0.04cm](13.34,2.0)(14.02,2.72)
\psline[linewidth=0.04cm](13.18,2.04)(13.28,2.94)
\psline[linewidth=0.04cm](14.42,1.48)(15.08,2.2)
\pscircle[linewidth=0.04,dimen=outer,fillstyle=solid](11.69,3.25){0.31}
\pscircle[linewidth=0.04,dimen=outer,fillstyle=solid](10.17,2.71){0.31}
\pscircle[linewidth=0.04,dimen=outer,fillstyle=solid](10.93,3.05){0.31}
\pscircle[linewidth=0.04,dimen=outer,fillstyle=solid](11.51,1.45){0.31}
\pscircle[linewidth=0.04,dimen=outer,fillstyle=solid](12.57,2.93){0.31}
\pscircle[linewidth=0.04,dimen=outer,fillstyle=solid](13.33,3.19){0.31}
\pscircle[linewidth=0.04,dimen=outer,fillstyle=solid](14.13,2.87){0.31}
\pscircle[linewidth=0.04,dimen=outer,fillstyle=solid](13.13,1.85){0.31}
\pscircle[linewidth=0.04,dimen=outer,fillstyle=solid](14.25,1.29){0.31}
\pscircle[linewidth=0.04,dimen=outer,fillstyle=solid](15.11,2.23){0.31}
\usefont{T1}{ptm}{m}{n}
\rput(10.90625,3.035){\large 8}
\usefont{T1}{ptm}{m}{n}
\rput(13.098124,1.815){\large 6}
\usefont{T1}{ptm}{m}{n}
\rput(13.318123,3.175){\large 7}
\usefont{T1}{ptm}{m}{n}
\rput(14.1375,2.895){\large 9}
\usefont{T1}{ptm}{m}{n}
\rput(12.569376,2.915){\large 11}
\usefont{T1}{ptm}{m}{n}
\rput(10.125625,2.715){\large 10}
\usefont{T1}{ptm}{m}{n}
\rput(14.245625,1.275){\large 4}
\pscircle[linewidth=0.04,dimen=outer,fillstyle=solid](13.09,0.25){0.31}
\usefont{T1}{ptm}{m}{n}
\rput(13.059375,0.235){\large 1}
\usefont{T1}{ptm}{m}{n}
\rput(11.495,1.435){\large 2}
\usefont{T1}{ptm}{m}{n}
\rput(15.081249,2.235){\large 3}
\usefont{T1}{ptm}{m}{n}
\rput(11.69,3.255){\large 5}
\end{pspicture}
}\caption{ Product of the operad $\Arb$ and corresponding bicoloring.}\label{NAP}\end{center}
\end{figure}
\begin{ex}\label{rooted}
Let $\Arb$ be the species of rooted trees. The operad structure on $\Arb$ is defined as follows.
For a pair $(a,T'_{\pi})\in \Arb(\Arb)[U]$, $a=\{T_B\}_{B\in\pi}$ a forest, and $T'_{\pi}$ a tree with vertices in the partition $\pi$, the tree  $T=\eta(a,T'_{\pi})$ will have all the edges in $a$ plus a few more defined as follows: for every pair of trees $T_{B}$ and $T_{B'}$ such that $\{B,B'\}$ is an edge of $T'_{\pi}$, insert an edge between the roots of  $T_{B}$ and $T_{B'}$. The root of $T$ will be the root of the tree in $a$ whose associate block is the root of $T'_{\pi}$ (see fig. \ref{NAP}). The linearization of it is the NAP operad (\cite{Livernet}). The factorization of a tree $T$ can be codified as a bicoloring $b$ of the edges of $T$ in the colors $1$ and $2$ (see Fig. \ref{NAP}). The edges in $T$ coming from the external tree are colored $1$, the edges coming from the internal ones are colored $2$. Denote by $a_b(2)$ the assembly of internal trees (colored 2) and by $T_b(1)$ the sub-rooted tree of $T$ induced by the edges of color $1$. Then, the coproduct of $\CN_{\Arb}$ can be written as
\be\label{coproductnap} \Delta(t_{\tau(T)})=\sum_{b}t_{\tau(a_b(2))}\otimes t_{\tau(T_b(1))}
\eeq 
\noindent The sum of above is over all the bicolorations $b$ of edges of $T$ such that edges colored $1$ induce a subrooted tree of $T$ (that may be empty).
For example, we have the coproduct
$$\Delta( t_{\arbolitoder})=1\otimes t_{\arbolitoder}+t_{\ldos}\otimes t_{\cdos}+t_{\ltres}\otimes t_{\ldos}+t_{\ldos}\otimes t_{\ltres}+t_{\ldos}^2\otimes t_{\ldos}+t_{\arbolitoder}\otimes 1,$$
according with the bicolorings
\[
\scalebox{0.8}
{
\begin{pspicture}(0,-0.34)(5.0759373,0.34)
\psdots[dotsize=0.12](0.42,0.26)
\psdots[dotsize=0.12](0.42,0.02)
\psdots[dotsize=0.12](0.26,-0.22)
\psdots[dotsize=0.12](0.12,0.02)
\psline[linewidth=0.04cm](0.42,0.22)(0.42,0.04)
\psline[linewidth=0.04cm](0.4,0.0)(0.28,-0.22)
\psline[linewidth=0.04cm](0.26,-0.24)(0.12,0.02)
\usefont{T1}{ptm}{m}{n}
\rput(0.04421875,-0.2){\small 1}
\usefont{T1}{ptm}{m}{n}
\rput(0.40421876,-0.2){\small 1}
\usefont{T1}{ptm}{m}{n}
\rput(0.52421874,0.14){\small 1}
\psdots[dotsize=0.12](1.2857188,0.26)
\psdots[dotsize=0.12](1.2857188,0.02)
\psdots[dotsize=0.12](1.1257187,-0.22)
\psdots[dotsize=0.12](0.9857187,0.02)
\psline[linewidth=0.04cm](1.2857188,0.22)(1.2857188,0.04)
\psline[linewidth=0.04cm](1.2657187,0.0)(1.1457187,-0.22)
\psline[linewidth=0.04cm](1.1257187,-0.24)(0.9857187,0.02)
\usefont{T1}{ptm}{m}{n}
\rput(0.9299375,-0.2){\small 1}
\usefont{T1}{ptm}{m}{n}
\rput(1.2699375,-0.2){\small 1}
\usefont{T1}{ptm}{m}{n}
\rput(1.4200938,0.14){\small 2}
\psdots[dotsize=0.12](3.0748124,0.26)
\psdots[dotsize=0.12](3.0748124,0.02)
\psdots[dotsize=0.12](2.9148126,-0.22)
\psdots[dotsize=0.12](2.7748125,0.02)
\psline[linewidth=0.04cm](3.0748124,0.22)(3.0748124,0.04)
\psline[linewidth=0.04cm](3.0548124,0.0)(2.9348125,-0.22)
\psline[linewidth=0.04cm](2.9148126,-0.24)(2.7748125,0.02)
\usefont{T1}{ptm}{m}{n}
\rput(2.7291875,-0.2){\small 2}
\usefont{T1}{ptm}{m}{n}
\rput(3.0590312,-0.2){\small 1}
\usefont{T1}{ptm}{m}{n}
\rput(3.1790311,0.14){\small 1}
\psdots[dotsize=0.12](2.1651876,0.26)
\psdots[dotsize=0.12](2.1651876,0.02)
\psdots[dotsize=0.12](2.0051875,-0.22)
\psdots[dotsize=0.12](1.8651875,0.02)
\psline[linewidth=0.04cm](2.1651876,0.22)(2.1651876,0.04)
\psline[linewidth=0.04cm](2.1451876,0.0)(2.0251875,-0.22)
\psline[linewidth=0.04cm](2.0051875,-0.24)(1.8651875,0.02)
\usefont{T1}{ptm}{m}{n}
\rput(1.8094063,-0.2){\small 1}
\usefont{T1}{ptm}{m}{n}
\rput(2.1795626,-0.2){\small 2}
\usefont{T1}{ptm}{m}{n}
\rput(2.2995625,0.14){\small 2}
\psdots[dotsize=0.12](3.9405313,0.26)
\psdots[dotsize=0.12](3.9405313,0.02)
\psdots[dotsize=0.12](3.7805312,-0.22)
\psdots[dotsize=0.12](3.6405313,0.02)
\psline[linewidth=0.04cm](3.9405313,0.22)(3.9405313,0.04)
\psline[linewidth=0.04cm](3.9205313,0.0)(3.8005311,-0.22)
\psline[linewidth=0.04cm](3.7805312,-0.24)(3.6405313,0.02)
\usefont{T1}{ptm}{m}{n}
\rput(3.5949063,-0.2){\small 2}
\usefont{T1}{ptm}{m}{n}
\rput(3.92475,-0.2){\small 1}
\usefont{T1}{ptm}{m}{n}
\rput(4.0749063,0.14){\small 2}
\psdots[dotsize=0.12](4.82,0.26)
\psdots[dotsize=0.12](4.82,0.02)
\psdots[dotsize=0.12](4.66,-0.22)
\psdots[dotsize=0.12](4.52,0.02)
\psline[linewidth=0.04cm](4.82,0.22)(4.82,0.04)
\psline[linewidth=0.04cm](4.8,0.0)(4.68,-0.22)
\psline[linewidth=0.04cm](4.66,-0.24)(4.52,0.02)
\usefont{T1}{ptm}{m}{n}
\rput(4.474375,-0.2){\small 2}
\usefont{T1}{ptm}{m}{n}
\rput(4.854375,-0.2){\small 2}
\usefont{T1}{ptm}{m}{n}
\rput(4.954375,0.14){\small 2}
\end{pspicture}
}.\]
\end{ex}

\begin{ex} We generalize now the operad $\Arb$. Let $N$ be an arbitrary species.
 The species   $\Arb_N$ has as elements $N$-enriched rooted trees. An enriched rooted tree $T_N\in\Arb_N[U]$ is of the form $(T,\{n_u\}_{u\in  U})$, where $T\in\Arb[U]$ is a rooted tree and $n_u$ is an $N$-structure on the fiber (set of descendants) of the vertex $u$. As a species $\Arb_N$ is uniquely defined by the implicit equation \be \Arb_N=X.N(\Arb_N).\eeq

  An element $(a, T_N)$ in $\Arb_N(\Arb_N)[U]$ is of the form
 \be (a,T_N)=(\{(T_B,\{n_u\}_{u\in B})\mid B\in\pi\},(T_{\pi},\{n'_B\}_{B\in\pi}))
 \eeq
 If $(N,\nu)$ is a monoid, we define the product $\eta_{\nu}:\Arb_N(\Arb_N)\rightarrow\Arb_N$, $\eta_{\nu}(a,T_N)=(T_2,\{n''_u\}_{u\in U})$, where $T_2$ is the product $T_2=\eta(\{T_B\}_{B\in\pi}, T_{\pi})$, $\eta$ as defined in the Example \ref{rooted}. The fibers of $T_2$ are enriched by setting
 \begin{equation}
 n''_u=\begin{cases}\nu(n_u,n'_B)&\mbox{if $u$ is the root of some tree $T_B$}\\ n_u&\mbox{otherwise}.\end{cases}
 \end{equation}
\noindent  See Fig. \ref{NAPN}.
\begin{figure}[h]
\begin{center}
\scalebox{0.8}
{\begin{pspicture}(0,-4.91)(14.52,4.91)
\pscircle[linewidth=0.04,dimen=outer](10.76,1.16){0.26}
\pscircle[linewidth=0.04,dimen=outer](11.22,-0.24){0.26}
\rput{-1.6081004}(0.030324182,0.3708421){\pscircle[linewidth=0.04,dimen=outer](13.227206,-0.8949428){0.26}}
\rput{-1.6081004}(0.029001217,0.33943635){\pscircle[linewidth=0.04,dimen=outer](12.107647,-0.8635123){0.26}}
\rput{-1.6081004}(0.06779826,0.35353276){\pscircle[linewidth=0.04,dimen=outer](12.62926,-2.2386918){0.26}}
\psline[linewidth=0.04cm](12.10091,-1.1034178)(12.515482,-2.015411)
\psline[linewidth=0.04cm](12.755387,-2.022146)(13.12051,-1.132042)
\psline[linewidth=0.04cm](10.8,0.94)(11.16,0.0)
\pscircle[linewidth=0.04,dimen=outer](8.9,-1.42){0.26}
\pscircle[linewidth=0.04,dimen=outer](10.76,-3.1){0.26}
\pscircle[linewidth=0.04,dimen=outer](9.82,0.58){0.26}
\pscircle[linewidth=0.04,dimen=outer](9.86,-0.9){0.26}
\pscircle[linewidth=0.04,dimen=outer](11.74,1.32){0.26}
\pscircle[linewidth=0.04,dimen=outer](12.64,2.04){0.26}
\pscircle[linewidth=0.04,dimen=outer](12.66,0.84){0.26}
\psline[linewidth=0.04cm](12.64,1.79)(12.66,1.1)
\psline[linewidth=0.04cm](9.04,-1.58)(10.52,-2.98)
\psline[linewidth=0.04cm](10.7,-2.9)(9.9,-1.11)
\psline[linewidth=0.04cm](9.9,-0.66)(9.88,0.34)
\psline[linewidth=0.04cm](11.22,-0.49)(10.8,-2.84)
\psline[linewidth=0.04cm](11.44,-0.14)(12.5,0.69)
\psline[linewidth=0.04cm](11.42,-0.06)(11.74,1.06)
\psline[linewidth=0.04cm](10.96,-2.96)(12.4,-2.33)
\rput{-1.6081004}(0.057703555,0.353958){\psarc[linewidth=0.04](12.639363,-1.8788337){0.58}{0.0}{180.0}}
\psarc[linewidth=0.04](10.89,-3.15){1.39}{5.630683}{161.09543}
\psarc[linewidth=0.04](11.6,-0.06){0.9}{5.630683}{175.3141}
\usefont{T1}{ptm}{m}{n}
\rput{-1.6081004}(0.06801672,0.35303995){\rput(12.56025,-2.2604713){\large 1}}
\usefont{T1}{ptm}{m}{n}
\rput(10.72875,-3.125){\large 2}
\usefont{T1}{ptm}{m}{n}
\rput(9.788439,-0.885){\large 3}
\usefont{T1}{ptm}{m}{n}
\rput(9.796719,0.555){\large 4}
\usefont{T1}{ptm}{m}{n}
\rput{-1.6081004}(0.028245404,0.3392986){\rput(12.0833,-0.8477474){\large 6}}
\usefont{T1}{ptm}{m}{n}
\rput(8.816094,-1.445){\large 7}
\usefont{T1}{ptm}{m}{n}
\rput(10.732187,1.155){\large 8}
\usefont{T1}{ptm}{m}{n}
\rput(12.635626,0.835){\large 9}
\usefont{T1}{ptm}{m}{n}
\rput(11.208438,-0.245){\large b}
\usefont{T1}{ptm}{m}{n}
\rput(11.709375,1.275){\large d}
\usefont{T1}{ptm}{m}{n}
\rput(12.628124,1.995){\large k}
\psarc[linewidth=0.04](12.72,1.08){0.62}{37.568592}{137.2906}
\psarc[linewidth=0.04](9.92,-0.76){0.62}{37.568592}{137.2906}
\usefont{T1}{ptm}{m}{n}
\rput{-1.6081004}(0.031316563,0.36955005){\rput(13.160732,-0.9389133){\large 5}}
\pscircle[linewidth=0.04,dimen=outer](1.76,1.07){0.26}
\pscircle[linewidth=0.04,dimen=outer](1.74,-0.17){0.26}
\psline[linewidth=0.04cm](1.74,0.83)(1.76,0.05)
\pscircle[linewidth=0.04,dimen=outer](1.87,0.38){1.19}
\psarc[linewidth=0.04](1.78,-0.15){0.62}{56.309933}{123.69007}
\usefont{T1}{ptm}{m}{n}
\rput(1.70625,1.045){\large 8}
\usefont{T1}{ptm}{m}{n}
\rput(1.72125,-0.175){\large b}
\usefont{T1}{ptm}{m}{n}
\rput(2.338125,0.1956226){$n_4$}
\pscircle[linewidth=0.04,dimen=outer](0.89,2.76){0.26}
\pscircle[linewidth=0.04,dimen=outer](0.89,2.76){0.39}
\usefont{T1}{ptm}{m}{n}
\rput(0.8753125,2.7451563){\large d}
\pscircle[linewidth=0.04,dimen=outer](2.84,4.15){0.26}
\pscircle[linewidth=0.04,dimen=outer](2.84,3.01){0.26}
\psline[linewidth=0.04cm](2.84,3.93)(2.84,3.25)
\pscircle[linewidth=0.04,dimen=outer](2.9,3.49){1.06}
\psarc[linewidth=0.04](2.88,3.07){0.62}{56.309933}{118.07249}
\usefont{T1}{ptm}{m}{n}
\rput(2.7975,3.005){\large 9}
\usefont{T1}{ptm}{m}{n}
\rput(2.8275,4.105){\large k}
\usefont{T1}{ptm}{m}{n}
\rput(3.4181252,3.4156225){$n_5$}
\psline[linewidth=0.04cm](1.28,1.37)(0.96,2.39)
\psline[linewidth=0.04cm](2.24,1.47)(2.68,2.47)
\psarc[linewidth=0.04](1.8,1.53){0.9}{23.629377}{155.37643}
\usefont{T1}{ptm}{m}{n}
\rput(2.9381251,1.6956227){$n'_4$}
\pscircle[linewidth=0.04,dimen=outer](5.18,1.23){0.26}
\pscircle[linewidth=0.04,dimen=outer](4.26,1.35){0.26}
\pscircle[linewidth=0.04,dimen=outer](4.64,0.13){0.26}
\psline[linewidth=0.04cm](4.26,1.11)(4.54,0.35)
\psline[linewidth=0.04cm](4.76,0.35)(5.06,1.03)
\pscircle[linewidth=0.04,dimen=outer](2.66,-3.31){0.26}
\pscircle[linewidth=0.04,dimen=outer](3.26,-4.23){0.26}
\pscircle[linewidth=0.04,dimen=outer](3.88,-2.03){0.26}
\pscircle[linewidth=0.04,dimen=outer](3.84,-3.05){0.26}
\psline[linewidth=0.04cm](2.82,-3.45)(3.12,-4.07)
\psline[linewidth=0.04cm](3.4,-4.03)(3.74,-3.27)
\psline[linewidth=0.04cm](3.84,-2.81)(3.86,-2.27)
\pscircle[linewidth=0.04,dimen=outer](4.75,0.76){1.15}
\pscircle[linewidth=0.04,dimen=outer](3.56,-3.19){1.54}
\psline[linewidth=0.04cm](2.42,-0.67)(2.92,-1.81)
\psline[linewidth=0.04cm](4.0,-1.71)(4.48,-0.31)
\psarc[linewidth=0.04](3.28,-4.07){0.58}{26.565052}{157.38014}
\psarc[linewidth=0.04](3.85,-3.02){0.59}{36.253838}{122.90524}
\psarc[linewidth=0.04](4.6,0.23){0.58}{21.801409}{146.30994}
\psarc[linewidth=0.04](3.47,-1.98){1.39}{37.05653}{147.75754}
\usefont{T1}{ptm}{m}{n}
\rput(4.5593753,0.105){\large 1}
\usefont{T1}{ptm}{m}{n}
\rput(3.215,-4.255){\large 2}
\usefont{T1}{ptm}{m}{n}
\rput(3.8412502,-3.055){\large 3}
\usefont{T1}{ptm}{m}{n}
\rput(3.8256247,-2.055){\large 4}
\usefont{T1}{ptm}{m}{n}
\rput(5.15,1.205){\large 5}
\usefont{T1}{ptm}{m}{n}
\rput(4.1981254,1.345){\large 6}
\usefont{T1}{ptm}{m}{n}
\rput(2.638125,-3.295){\large 7}
\psline[linewidth=0.04cm,tbarsize=0.07055555cm 5.0,arrowsize=0.093cm 2.0,arrowlength=1.4,arrowinset=0.4]{|*->}(6.56,-0.51)(8.72,-0.51)
\usefont{T1}{ptm}{m}{n}
\rput(7.332031,-0.155){\large $\eta_{\nu}$}
\usefont{T1}{ptm}{m}{n}
\rput(4.018125,-4.0243773){$n_1$}
\usefont{T1}{ptm}{m}{n}
\rput(12.268126,-3.2243776){$\nu(n_1, n'_1)$}
\usefont{T1}{ptm}{m}{n}
\rput(5.278125,0.2756226){$n_3$}
\usefont{T1}{ptm}{m}{n}
\rput(10.438126,-0.6443775){$n_2$}
\usefont{T1}{ptm}{m}{n}
\rput(4.418125,-2.7843773){$n_2$}
\usefont{T1}{ptm}{m}{n}
\rput(4.8181252,-1.3843774){$n'_1$}
\usefont{T1}{ptm}{m}{n}
\rput(12.788124,-0.2443774){$\nu(n_4, n'_4)$}
\usefont{T1}{ptm}{m}{n}
\rput(13.318127,1.2356226){$n_5$}
\usefont{T1}{ptm}{m}{n}
\rput(13.358126,-2.0043776){$n_3$}
\psframe[linewidth=0.04,dimen=outer](14.52,4.91)(0.0,-4.91)
\end{pspicture}}
\end{center}
\caption{Product of the operad $\Arb_{N}$}\label{NAPN}
\end{figure}
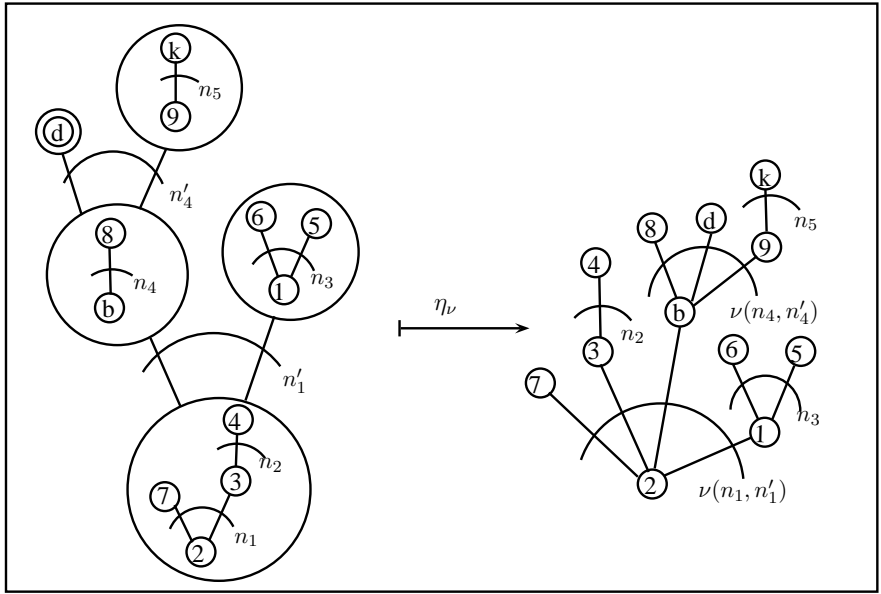

 $\Arb_E=\Arb$, and $\Arb_L$ is the operad of plannar rooted trees. Since for any positive species $M$, $L(M)$ and $E(M)$ are respectively the free and commutatively free monoids generated by $M$, $\CN_{\Arb_{E(N)}}$ and $\CN_{\Arb_{L(N)}}$ give us a large class of Hopf algebras.
For the Hopf algebra $\CN_{\Arb_L}$ of planar trees, we have the coproducts
 \begin{eqnarray*} \Delta( t_{\arbolitoder})&=&t_{\arbolitoder}\otimes 1+t_{\ldos}\otimes t_{\cdos}+t_{\ldos}\otimes t_{\ltres}+t_{\ldos}^2\otimes t_{\ldos}+1\otimes t_{\arbolitoder}\\ 
  \Delta( t_{\arbolitoizq})&=&t_{\arbolitoizq}\otimes 1+t_{\ldos}\otimes t_{\cdos}+t_{\ltres}\otimes t_{\ldos}+1\otimes t_{\arbolitoizq}\end{eqnarray*}

The construction at the end of section 2 is also a source of Hopf algebras associated to enriched trees, since for each operad $M$ we have an operad $\Arb_{M'}$. We can also associate to a monoid $N$ an infinite sequence of Hopf algebras of this form, because we have
 
\begin{equation*}
N\mbox{ monoid}\Rightarrow \Arb_N \mbox{ operad}\Rightarrow(\Arb_M)'\mbox{ monoid}\Rightarrow \Arb_{(\Arb_N)'}\mbox{ operad}
 \Rightarrow (\Arb_{(\Arb_N)'})' \mbox{ monoid}\Rightarrow\dots.
\end{equation*}
\end{ex}
\begin{ex}Let $\Grp$ be the species of simple, connected pointed graphs, i.e., graphs with a distinguished vertex. The elements of $\Grp$ are pairs $(g,v)$, where $g$ is a simple and connected graph $g\in\Gr[B]$, and $v\in B$ is its distinguished vertex. $\Grp$ is a set operad.
For $(\{(g_B,v_B)\}_{B\in\pi},(g'_{\pi},B_0))\in\Grp(\Grp)[U]$,
$\eta(\{(g_B,v_B)\}_{B\in\pi},(g'_{\pi},B_0))$ is the pointed graph $(g,v_0)$ constructed with vertices in $U$ as
follows. Keep all the edges of the graphs in the assembly, and for each edge $\{B,B'\}$
of the external graph $g'_{\pi}$, connect the respective distinguished vertices $v_B$ in $B$ and $v_{B'}$ in 
$B'$. In other words, $\{x,y\}$ is an edge in $g$ if one of the following two conditions
is satisfied:
\begin{enumerate} \item $\{x,y\}$ is an edge in $g_B$, for some block $B$ of $\pi$, \item
There exist an edge $\{B,B'\}$ of $g_{\pi}'$, $B, B'\in \pi$, such that $x=v_B$ and $y=v_B',$ the distinguished vertices of the respective graphs $g_B$ and $g_{B'}$.
\end{enumerate}	Finally, we chose $v_{B_0}$, the distinguished vertex of $g_{B_0}$, to be the distinguished vertex $v_0$ of $g$.
The Hopf algebra $\CN_{\Arb}$ is clearly contained in $\CN_{\Grp}$
As an example we compute the coproduct of the bow tie graph (see Fig. \ref{prodGrp} for corresponding products).
\[\Delta(t_{\lacito})=t_{\lacito}\otimes 1+2t_{\triangulito}\otimes t_{\triangulito}+1\otimes t_{\lacito}.\] 
\end{ex}
 
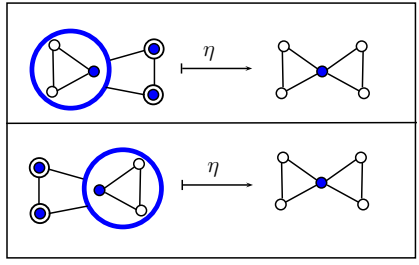
\begin{figure}[h]
\begin{center}
\scalebox{0.5} 
{
\begin{pspicture}(0,-3.41)(10.96,3.41)
\pscircle[linewidth=0.04,dimen=outer,fillstyle=solid,fillcolor=blue](2.35,1.57){0.16}
\pscircle[linewidth=0.04,dimen=outer,fillstyle=solid](1.24,1.03){0.16}
\pscircle[linewidth=0.126,linecolor=blue,dimen=outer](1.73,1.62){1.09}
\pscircle[linewidth=0.04,dimen=outer,fillstyle=solid,fillcolor=blue](3.95,2.17){0.16}
\pscircle[linewidth=0.04,dimen=outer,fillstyle=solid,fillcolor=blue](3.93,0.95){0.16}
\pscircle[linewidth=0.048,dimen=outer](3.95,2.17){0.28}
\pscircle[linewidth=0.05,dimen=outer](3.93,0.95){0.28}
\psline[linewidth=0.04cm](2.78,1.87)(3.7,2.13)
\psline[linewidth=0.04cm](2.66,1.15)(3.7,0.95)
\psline[linewidth=0.04cm](3.96,1.93)(3.96,1.19)
\psline[linewidth=0.04cm](1.26,2.17)(1.28,1.15)
\psline[linewidth=0.04cm](1.36,1.05)(2.2,1.51)
\psline[linewidth=0.04cm](1.42,2.19)(2.24,1.61)
\pscircle[linewidth=0.04,dimen=outer,fillstyle=solid](1.3,2.27){0.16}
\pscircle[linewidth=0.04,dimen=outer,fillstyle=solid,fillcolor=blue](2.5,-1.59){0.16}
\pscircle[linewidth=0.04,dimen=outer,fillstyle=solid](3.61,-2.13){0.16}
\pscircle[linewidth=0.126,linecolor=blue,dimen=outer](3.12,-1.54){1.09}
\pscircle[linewidth=0.04,dimen=outer,fillstyle=solid,fillcolor=blue](0.9,-0.99){0.16}
\pscircle[linewidth=0.04,dimen=outer,fillstyle=solid,fillcolor=blue](0.92,-2.21){0.16}
\pscircle[linewidth=0.048,dimen=outer](0.9,-0.99){0.28}
\pscircle[linewidth=0.05,dimen=outer](0.92,-2.21){0.28}
\psline[linewidth=0.04cm](2.07,-1.29)(1.15,-1.03)
\psline[linewidth=0.04cm](2.19,-2.01)(1.15,-2.21)
\psline[linewidth=0.04cm](0.89,-1.23)(0.89,-1.97)
\psline[linewidth=0.04cm](3.59,-0.99)(3.57,-2.01)
\psline[linewidth=0.04cm](3.49,-2.11)(2.65,-1.65)
\psline[linewidth=0.04cm](3.43,-0.97)(2.61,-1.55)
\pscircle[linewidth=0.04,dimen=outer,fillstyle=solid](3.55,-0.89){0.16}
\psline[linewidth=0.04cm,tbarsize=0.07055555cm 5.0,arrowsize=0.05291667cm 2.0,arrowlength=1.4,arrowinset=0.4]{|->}(4.66,1.65)(6.56,1.65)
\psline[linewidth=0.04cm,tbarsize=0.07055555cm 5.0,arrowsize=0.05291667cm 2.0,arrowlength=1.4,arrowinset=0.4]{|->}(4.68,-1.45)(6.58,-1.45)
\pscircle[linewidth=0.04,dimen=outer,fillstyle=solid](7.34,0.99){0.16}
\psline[linewidth=0.04cm](7.36,2.17)(7.38,1.15)
\psline[linewidth=0.04cm](7.52,2.19)(8.34,1.61)
\pscircle[linewidth=0.04,dimen=outer,fillstyle=solid](7.4,2.24){0.16}
\pscircle[linewidth=0.04,dimen=outer,fillstyle=solid,fillcolor=blue](8.42,1.57){0.16}
\pscircle[linewidth=0.04,dimen=outer,fillstyle=solid](9.53,0.99){0.16}
\psline[linewidth=0.04cm](9.51,2.17)(9.49,1.15)
\psline[linewidth=0.04cm](9.41,1.05)(8.57,1.51)
\psline[linewidth=0.04cm](9.35,2.19)(8.53,1.61)
\pscircle[linewidth=0.04,dimen=outer,fillstyle=solid](9.47,2.24){0.16}
\psline[linewidth=0.04cm](7.46,1.05)(8.3,1.51)
\pscircle[linewidth=0.04,dimen=outer,fillstyle=solid](7.32,-1.97){0.16}
\psline[linewidth=0.04cm](7.34,-0.79)(7.36,-1.81)
\psline[linewidth=0.04cm](7.5,-0.77)(8.32,-1.35)
\pscircle[linewidth=0.04,dimen=outer,fillstyle=solid](7.38,-0.72){0.16}
\pscircle[linewidth=0.04,dimen=outer,fillstyle=solid,fillcolor=blue](8.4,-1.39){0.16}
\pscircle[linewidth=0.04,dimen=outer,fillstyle=solid](9.51,-1.97){0.16}
\psline[linewidth=0.04cm](9.49,-0.79)(9.47,-1.81)
\psline[linewidth=0.04cm](9.39,-1.91)(8.55,-1.45)
\psline[linewidth=0.04cm](9.33,-0.77)(8.51,-1.35)
\pscircle[linewidth=0.04,dimen=outer,fillstyle=solid](9.45,-0.72){0.16}
\psline[linewidth=0.04cm](7.44,-1.91)(8.28,-1.45)
\usefont{T1}{ptm}{m}{n}
\rput(5.4090624,2.05){\LARGE $\eta$}
\usefont{T1}{ptm}{m}{n}
\rput(5.5290623,-1.03){\LARGE $\eta$}
\psframe[linewidth=0.04,dimen=outer](10.92,3.41)(0.02,-3.41)
\psline[linewidth=0.04cm](0.02,0.19)(10.94,0.21)
\end{pspicture}
}
\caption{Example of products in $\Grp$}\label{prodGrp}
\end{center}
\end{figure}

\section{The general  antipode formula}
In this subsection we present an antipode formula for $\CN_M$, which is a natural
generalization of the antipode formula of Haiman and Schmitt \cite{Haimann-Schmitt}.
We begin by introducing some terminology about enriched Schr\"oder trees as presented in
 \cite{MendezKoszulduality}.  We denote by
 $\mathcal{F}$ the species of Schr\"oder trees, or generalized commutative
parenthesizations. It satisfies the implicit equation.
\begin{eqnarray}\label{ec16}
\mathcal{F}&=&X+E_{2^+}(\mathcal{F})
\end{eqnarray}

The structures of $\mathcal{F}[U]$ are trees whose internal vertices are unlabelled and
whose leaves has as labels the elements of $U$.

For a tree
$\T\in \CF[U]$, denote by $\mathrm{Iv}(\T)$ the set of the internal vertices of $\T$. For an
internal vertex $v$, we denote by $\T_v$ the subtree of $\T$ that has root $v$ and vertices
all the descendants of $v$ in $\T$. Let $U_v$ be the set of leaves of $\T_v$. This allows
us to identify the vertex $v$ with $U_v$, we shall use the set $U_v$ as a label for the vertex $v$.
 
Let $\{v_1, v_2,...,v_k\}$ be the set of sons of $v$. Each of them is either an internal vertex, or a leaf. We denote by $\pi_{v}$ the
partition of $U_v$ induced by the branching at $v$:
$$\pi_{v}=\{U_{v_i}\mid i=1,2,\dots,k\}, $$
        
        \noindent each $U_{v_i}$ being the set of leaves of the tree $\T_{v_i}$, if $v_i$ is an internal vertex, and $U_{v_i}=\{v_i\}$ if not.
        For a species $M$  having the form
        $$M=X+M_{2^+},$$ the species of $M$-enriched Schr\"oder trees is the solution to
        the implicit equation  \be\label{ec17} \CF_M=X+M_{2^+}(\CF_M) \eeq Using (\ref{ec17}), we
        obtain the following recursive description $\CF_M[U]$. If $U$ is an unitary set, the only
        tree in  $\CF_M[U]$ consists of only one leaf labelled with the element of $U$ (the
        singleton-leaf tree). If the cardinal of $|U|\geq 2$, an element  $\TT$ of
        $\CF_M[U]$ is a pair $(\{\TT_B\}_{B\in\pi_r},m_r)$, where for each $B$ in
        the partition $\pi_r$, $\TT_B\in \mathcal{F}_M[B]$ and $m_r\in M_{2^+}[\pi_r]$.
        Iterating this recursive description we get an explicit expression  of the set $\CF_M[U]$;
        
        $$\CF_M[U]=X[U]+\sum_{\T\in \CF[U]} \{\T\}\times \prod_{v\in \mathrm{Iv}(\T)}
        M_{2^+}[\pi_v].$$

In other words, a tree $\TT$ in  $\mathcal{F}_M[U]$ is a Schr\"oder tree $\T\in
\CF[U]$ together with a structure $m_v\in M[\pi_v]$ for each internal vertex $v\in
\mathrm{Iv}(\T)=\mathrm{Iv}(\TT)$ (see Fig. \ref{Schro}).

 If $(M,\eta)$ is a set operad, for a tree $\TT$ in $\mathcal{F}_M[U]$, $\widehat{\eta}(\TT)$
is the element of $M[U]$ obtained by applying the operad product recursively on each
level of the tree. In this way, if $\TT$ is the singleton-leaf tree,
 $\widehat{\eta}(\TT)$ is equal to the singleton structure in $M[U]$. Otherwise, $\TT$ is of
the form $(\{\TT_B\}_{B\in\pi_r}, m_r)$, where $m_r$ is the structure attached to the root
and $\{\TT_B\}_{B\in\pi_r}$ is the assembly of smaller trees whose roots are the sons of
$r$. We define $\widehat{\eta}(\TT)$ by

\be\widehat{\eta}(\TT)=\eta(\{\widehat{\eta}(\TT_B))\}_{B\in \pi_r}, m_r).\eeq
The previous recursive procedure can be replace by any other that systematically apply the products on each internal vertex of the tree, finishing with the root. By associativity of the operad product the result will be the same.
\begin{theo}\label{antipode} The antipode $S$ of  $\CN_M$ is as follows:
for $\alpha\in \Tt(M_{2^+})$,
\begin{equation}
\label{ec18} S(t_{\alpha})= \begin{cases}1&\mbox{if $\alpha=\bullet$} \\\sum_{\substack{\TT\in \mathcal{F}_M \\
\widehat{\eta}(\TT)=m}}\,(-1)^{\mid \mathrm{Iv}(\TT)\mid}\,\prod_{v\in
\mathrm{Iv}(\TT)}t_{\tau(m_v)}&\mbox{otherwise}.\end{cases}
\end{equation}
\noindent where $m$ is an $M$-structure of type $\alpha.$
\end{theo}
\begin{proof}

Denote by $\varphi:\NA\rightarrow\NA$  the algebra homomorphism defined by the right hand
side of Eq. (\ref{ec18}). If we denote by $\bullet$ the type of the singleton structure of $M$, by definition $\varphi(t_{\bullet})=\varphi(1)=1$. For $\alpha\neq \bullet$,
by the recursive definition of $\widehat{\eta}$,
\begin{eqnarray}
\varphi(t_{\alpha})&=&\sum_{\substack{\eta(a,m_r)=m\\ \tau(m_r)\neq
\tau(m)}}(-1)t_{\tau(m_r)}\sum_{\{\widehat{\eta}(\mathrm{T}_B)\}_{B\in\pi_r}=a}
\prod_{B\in\pi_r}(-1)^{\mid \mathrm{Iv(\mathrm{T_B})}\mid}\,\prod_{v\in
\mathrm{Iv}(\mathrm{T}_B)}t_{\tau(m_v)} \\
&=&\sum_{\substack{\eta(a,m_r)=m\\ \tau(m_r)\neq \tau(m)}}(-1)t_{\tau(m_r)}
\prod_{B\in\pi_r}\sum_{\widehat{\eta}(\mathrm{T}_B)=m_B}(-1)^{\mid
\mathrm{Iv(\mathrm{T_B})}\mid}\,\prod_{v\in \mathrm{Iv}(\mathrm{T}_B)}t_{\tau(m_v)}\\
&=&\sum_{\substack{\eta(a,m_r)=m\\ \tau(m_r)\neq \tau(m)}}(-1)t_{\tau(m_r)}
\prod_{B\in\pi_r}\varphi(t_{\tau(m_B)})\\&=&-\sum_{\substack{\eta(a,m_r)=m\\ \tau(m_r)\neq
\tau(m)}}t_{\tau(m_r)}\varphi(t_{\tau(a)}).
\end{eqnarray}
Hence, $(\varphi\ast I)(t_{\alpha})=(e\circ\epsilon)(t_{\alpha})$, and then $\varphi=S$.
\end{proof}

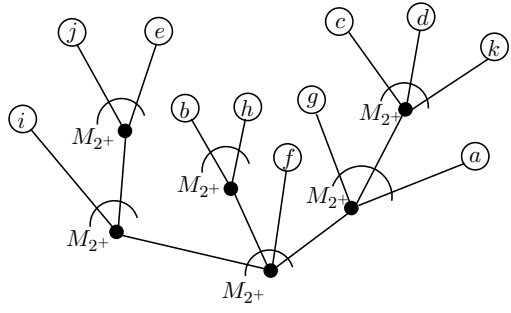
\begin{figure}\begin{center}
\scalebox{0.8} 
{
        \begin{pspicture}(0,-2.5464618)(8.379531,2.5064619)
        \psdots[dotsize=0.24200001](4.3599997,-1.9535382)
        \psdots[dotsize=0.24200001](5.6999993,-0.9135382)
        \psdots[dotsize=0.24200001](3.6999998,-0.59353817)
        \psdots[dotsize=0.24200001](1.7999996,-1.3135382)
        \psdots[dotsize=0.24200001](1.9399993,0.3664618)
        \psdots[dotsize=0.24200001](6.5999994,0.7264618)
        \pscircle[linewidth=0.025999999,dimen=outer](5.4999995,2.186462){0.24}
        \psline[linewidth=0.025999999cm](5.8199997,-0.8735382)(7.5799994,-0.1735382)
        \psline[linewidth=0.025999999cm](4.4599996,-1.8935382)(5.6599994,-0.9935382)
        \psline[linewidth=0.025999999cm](4.3799996,-1.9535382)(4.6199994,-0.29353818)
        \psline[linewidth=0.025999999cm](4.3399997,-1.9535382)(3.7399995,-0.6335382)
        \psline[linewidth=0.025999999cm](4.2999997,-1.9335382)(1.8199995,-1.3535382)
        \psline[linewidth=0.025999999cm](1.8199995,-1.3535382)(1.9599994,0.3864618)
        \psline[linewidth=0.025999999cm](1.8199995,-1.3335382)(0.3799995,0.3864618)
        \psline[linewidth=0.025999999cm](5.74,-0.9335382)(6.5599995,0.6464618)
        \psline[linewidth=0.025999999cm](5.6799994,-0.8335382)(5.1199994,0.6464618)
        \psline[linewidth=0.025999999cm](6.68,0.7064618)(7.9799995,1.5664618)
        \psline[linewidth=0.025999999cm](6.6199994,0.78646183)(6.8399997,2.0464618)
        \psline[linewidth=0.025999999cm](6.5799994,0.7264618)(5.6599994,2.0064619)
        \psline[linewidth=0.025999999cm](3.6999998,-0.59353817)(3.9399993,0.5664618)
        \psline[linewidth=0.025999999cm](3.7199996,-0.6135382)(3.0399992,0.5664618)
        \psline[linewidth=0.025999999cm](1.9599994,0.3064618)(2.4599993,1.8064618)
        \psline[linewidth=0.025999999cm](1.9599994,0.32646182)(1.1399997,1.8064618)
        \pscircle[linewidth=0.025999999,dimen=outer](6.8599997,2.2664618){0.24}
        \pscircle[linewidth=0.025999999,dimen=outer](8.08,1.7664618){0.24}
        \pscircle[linewidth=0.025999999,dimen=outer](7.7599993,-0.053538192){0.24}
        \pscircle[linewidth=0.025999999,dimen=outer](0.24,0.5664618){0.24}
        \pscircle[linewidth=0.025999999,dimen=outer](5.0399995,0.8864618){0.24}
        \pscircle[linewidth=0.025999999,dimen=outer](4.6399994,-0.07353819){0.24}
        \pscircle[linewidth=0.025999999,dimen=outer](3.9799993,0.7664618){0.24}
        \pscircle[linewidth=0.025999999,dimen=outer](2.9399993,0.7464618){0.24}
        \pscircle[linewidth=0.025999999,dimen=outer](2.4999993,2.0264618){0.24}
        \pscircle[linewidth=0.025999999,dimen=outer](1.0599993,2.0064619){0.24}
        \usefont{T1}{ptm}{m}{n}
        \rput(7.759062,-0.059006944){$a$}
        \usefont{T1}{ptm}{m}{n}
        \rput(2.9290624,0.7255243){$b$}
        \usefont{T1}{ptm}{m}{n}
        \rput(5.519062,2.1655242){$c$}
        \usefont{T1}{ptm}{m}{n}
        \rput(6.889062,2.2655244){$d$}
        \usefont{T1}{ptm}{m}{n}
        \rput(4.659062,-0.094475694){$f$}
        \usefont{T1}{ptm}{m}{n}
        \rput(2.5190618,2.0055244){$e$}
        \usefont{T1}{ptm}{m}{n}
        \rput(5.059062,0.8758368){$g$}
        \usefont{T1}{ptm}{m}{n}
        \rput(3.969062,0.7455243){$h$}
        \usefont{T1}{ptm}{m}{n}
        \rput(0.21906175,0.5455243){$i$}
        \usefont{T1}{ptm}{m}{n}
        \rput(1.069062,1.9958369){$j$}
        \psarc[linewidth=0.025999999](4.3399997,-2.0135381){0.4}{16.69925}{184.76364}
        \psarc[linewidth=0.025999999](5.8683586,-0.725179){0.5083592}{352.1909}{160.34618}
        \psarc[linewidth=0.025999999](1.7599996,-1.1935382){0.4}{16.69925}{184.76364}
        \psarc[linewidth=0.025999999](1.8799995,0.5064618){0.4}{16.69925}{184.76364}
        \psarc[linewidth=0.025999999](3.6199996,-0.29353818){0.4}{16.69925}{184.76364}
        \psarc[linewidth=0.025999999](6.5799994,0.7664618){0.4}{356.42365}{165.96375}
        \usefont{T1}{ptm}{m}{n}
        \rput(1.4190623,0.2358368){$M_{2^+}$}
        \usefont{T1}{ptm}{m}{n}
        \rput(8.069062,1.7455243){$k$}
        \usefont{T1}{ptm}{m}{n}
        \rput(6.199062,0.6413507){$M_{2^+}$}
        \usefont{T1}{ptm}{m}{n}
        \rput(5.3190618,-0.7386493){$M_{2^+}$}
        \usefont{T1}{ptm}{m}{n}
        \rput(1.3390617,-1.4586493){$M_{2^+}$}
        \usefont{T1}{ptm}{m}{n}
        \rput(3.1790617,-0.5586493){$M_{2^+}$}
        \usefont{T1}{ptm}{m}{n}
        \rput(3.9190617,-2.3186493){$M_{2^+}$}
        \end{pspicture}}\end{center}\caption{$M$-enriched Schr\"oder tree}\label{Schro}\end{figure}

As an example of this formula, we compute, in the Hopf algebra $\CN_{\Gr}$ the antipode of the almost complete graph $K_4-e$ (see Fig. \ref{antipodegr}) 
\be\label{casikcuatro}S(t_{\casikcuatro})=-t_{\casikcuatro}+3t_{\ltres}t_{\ldos}.\eeq

\begin{figure}[h]
\begin{center}
\scalebox{0.7} 
{
\begin{pspicture}(0,-9.41)(15.46,9.43)
\psdots[dotsize=0.32](4.279375,4.8453126)
\psdots[dotsize=0.32](3.879375,6.4453125)
\psline[linewidth=0.04cm](3.779375,6.5053124)(3.08,7.69)
\psline[linewidth=0.04cm](3.839375,6.5653124)(3.74,8.03)
\psline[linewidth=0.04cm](3.859375,6.465312)(4.58,7.81)
\usefont{T1}{ptm}{m}{n}
\rput(2.9251564,7.9653125){\Large c}
\usefont{T1}{ptm}{m}{n}
\rput(3.7073438,8.325313){\Large d}
\usefont{T1}{ptm}{m}{n}
\rput(4.6670313,8.105312){\Large b}
\psarc[linewidth=0.03](3.849375,6.7753124){0.71}{27.349876}{148.49573}
\psline[linewidth=0.04cm](4.199375,4.9653125)(3.84,6.5099998)
\psline[linewidth=0.04cm](4.319375,4.9653125)(4.82,6.0499997)
\usefont{T1}{ptm}{m}{n}
\rput(4.8321877,6.305312){\Large a}
\psarc[linewidth=0.03](4.3093753,5.2953124){0.71}{35.537678}{141.58194}
\psellipse[linewidth=0.04,dimen=outer](5.0793753,5.3253126)(0.24,0.26)
\psellipse[linewidth=0.04,dimen=outer](6.659375,5.3053126)(0.78,0.8)
\usefont{T1}{ptm}{m}{n}
\rput(6.2851562,5.5453124){\Large c}
\usefont{T1}{ptm}{m}{n}
\rput(5.0521874,5.3453126){\Large a}
\psline[linewidth=0.04cm](5.299375,5.3053126)(5.899375,5.3053126)
\psellipse[linewidth=0.04,dimen=outer](2.859375,6.9453125)(0.26,0.26)
\psellipse[linewidth=0.04,dimen=outer](2.879375,6.0453124)(0.26,0.26)
\psellipse[linewidth=0.04,dimen=outer](1.929375,6.0453124)(0.25,0.26)
\usefont{T1}{ptm}{m}{n}
\rput(2.8070312,6.8853126){\Large b}
\usefont{T1}{ptm}{m}{n}
\rput(2.8073437,6.0453124){\Large d}
\usefont{T1}{ptm}{m}{n}
\rput(1.9051563,6.0653124){\Large c}
\psline[linewidth=0.04cm](2.879375,6.7053123)(2.899375,6.325312)
\psline[linewidth=0.04cm](2.139375,6.0653124)(2.619375,6.0453124)
\usefont{T1}{ptm}{m}{n}
\rput(6.6673436,5.3085938){\Large d}
\usefont{T1}{ptm}{m}{n}
\rput(7.0270314,5.228594){\Large b}
\usefont{T1}{ptm}{m}{n}
\rput(3.1375,3.2203124){\LARGE $\widehat{\eta}(T)=\eta$}
\psellipse[linewidth=0.04,dimen=outer](4.679375,3.2203124)(0.24,0.26)
\psellipse[linewidth=0.04,dimen=outer](6.5493746,2.6303124)(1.13,1.11)
\psline[linewidth=0.04cm](4.9393754,3.2003124)(5.5793753,3.1003125)
\psellipse[linewidth=0.04,dimen=outer](5.889375,2.4203124)(0.25,0.26)
\psellipse[linewidth=0.04,dimen=outer](3.7593749,2.1403124)(0.24,0.26)
\psellipse[linewidth=0.04,dimen=outer](4.779375,2.1203125)(0.26,0.26)
\psellipse[linewidth=0.04,dimen=outer](4.7993746,1.2203125)(0.26,0.26)
\psellipse[linewidth=0.04,dimen=outer](3.849375,1.2203125)(0.25,0.26)
\psline[linewidth=0.04cm](4.7993746,1.8803126)(4.819375,1.5003123)
\psline[linewidth=0.04cm](4.0593753,1.2403123)(4.5393753,1.2203125)
\psline[linewidth=0.04cm](3.999375,2.2003124)(4.5593753,2.2003124)
\psline[linewidth=0.04cm](3.7593749,1.9003124)(3.7793748,1.4803125)
\psline[linewidth=0.04cm](3.9593751,2.0203123)(4.659375,1.4203124)
\usefont{T1}{ptm}{m}{n}
\rput(3.1284373,1.7653124){\huge =}
\psellipse[linewidth=0.04,dimen=outer](6.759375,3.2803123)(0.26,0.26)
\psellipse[linewidth=0.04,dimen=outer](6.779375,2.3803124)(0.26,0.26)
\psline[linewidth=0.04cm](6.779375,3.0403125)(6.7993746,2.6603124)
\psline[linewidth=0.04cm](6.099375,2.4203124)(6.5793753,2.4003124)
\usefont{T1}{ptm}{m}{n}
\rput(3.8051562,1.2403123){\Large c}
\usefont{T1}{ptm}{m}{n}
\rput(4.767031,2.1203125){\Large b}
\usefont{T1}{ptm}{m}{n}
\rput(4.7473435,1.1803124){\Large d}
\usefont{T1}{ptm}{m}{n}
\rput(3.7521875,2.1603124){\Large a}
\usefont{T1}{ptm}{m}{n}
\rput(5.825156,2.4403124){\Large c}
\usefont{T1}{ptm}{m}{n}
\rput(6.767031,3.2603123){\Large b}
\usefont{T1}{ptm}{m}{n}
\rput(6.7473435,2.3803124){\Large d}
\usefont{T1}{ptm}{m}{n}
\rput(4.6521873,3.2403126){\Large a}
\psdots[dotsize=0.32](11.019376,4.8853126)
\psdots[dotsize=0.32](10.599376,6.5653124)
\psline[linewidth=0.04cm](10.499375,6.625313)(9.76,7.9900002)
\psline[linewidth=0.04cm](10.559376,6.6853123)(10.54,8.09)
\psline[linewidth=0.04cm](10.579374,6.5853124)(11.4,7.87)
\psarc[linewidth=0.03](10.569374,6.8953123){0.71}{27.349876}{148.49573}
\usefont{T1}{ptm}{m}{n}
\rput(9.632188,8.265312){\Large a}
\usefont{T1}{ptm}{m}{n}
\rput(11.405156,8.025312){\Large c}
\usefont{T1}{ptm}{m}{n}
\rput(10.527031,8.285313){\Large b}
\psline[linewidth=0.04cm](10.939376,5.0053124)(10.6,6.65)
\psline[linewidth=0.04cm](11.059376,5.0053124)(11.66,6.35)
\psarc[linewidth=0.03](11.049376,5.3353124){0.71}{35.537678}{141.58194}
\psellipse[linewidth=0.04,dimen=outer](11.819374,5.3653126)(0.24,0.26)
\psellipse[linewidth=0.04,dimen=outer](13.399376,5.3453126)(0.78,0.8)
\psline[linewidth=0.04cm](12.039374,5.3453126)(12.639376,5.3453126)
\psellipse[linewidth=0.04,dimen=outer](9.399376,7.3453126)(0.26,0.26)
\psellipse[linewidth=0.04,dimen=outer](8.319375,7.2853127)(0.26,0.26)
\psellipse[linewidth=0.04,dimen=outer](8.489375,6.4453125)(0.25,0.26)
\psline[linewidth=0.04cm](8.579375,7.2653127)(9.179375,7.3253126)
\psline[linewidth=0.04cm](8.419375,6.645313)(8.379375,7.0853124)
\usefont{T1}{ptm}{m}{n}
\rput(8.292188,7.2853127){\Large a}
\usefont{T1}{ptm}{m}{n}
\rput(8.445156,6.4253125){\Large c}
\usefont{T1}{ptm}{m}{n}
\rput(9.367031,7.3053126){\Large b}
\usefont{T1}{ptm}{m}{n}
\rput(11.727344,6.6053123){\Large d}
\usefont{T1}{ptm}{m}{n}
\rput(13.412188,5.0053124){\Large a}
\usefont{T1}{ptm}{m}{n}
\rput(12.985156,5.5053124){\Large c}
\usefont{T1}{ptm}{m}{n}
\rput(13.527031,5.5053124){\Large b}
\usefont{T1}{ptm}{m}{n}
\rput(11.787344,5.3653126){\Large d}
\usefont{T1}{ptm}{m}{n}
\rput(10.6375,3.2303123){\LARGE $\widehat{\eta}(T)=\eta$}
\psellipse[linewidth=0.04,dimen=outer](15.159375,2.8103125)(0.24,0.26)
\psellipse[linewidth=0.04,dimen=outer](13.209375,3.1203125)(1.13,0.95)
\psline[linewidth=0.04cm](14.299376,2.9503124)(14.939375,2.8503125)
\usefont{T1}{ptm}{m}{n}
\rput(11.988438,1.2353123){\huge =}
\psellipse[linewidth=0.04,dimen=outer](13.759376,3.5103126)(0.26,0.26)
\psellipse[linewidth=0.04,dimen=outer](12.679375,3.4503124)(0.26,0.26)
\psellipse[linewidth=0.04,dimen=outer](12.849375,2.6103125)(0.25,0.26)
\psline[linewidth=0.04cm](12.939375,3.4303124)(13.539375,3.4903123)
\psline[linewidth=0.04cm](12.779374,2.8103125)(12.739374,3.2503123)
\psellipse[linewidth=0.04,dimen=outer](12.659374,1.8103125)(0.24,0.26)
\psellipse[linewidth=0.04,dimen=outer](13.679376,1.7903124)(0.26,0.26)
\psellipse[linewidth=0.04,dimen=outer](13.699375,0.8903124)(0.26,0.26)
\psellipse[linewidth=0.04,dimen=outer](12.749375,0.8903124)(0.25,0.26)
\psline[linewidth=0.04cm](13.699375,1.5503124)(13.719375,1.1703125)
\psline[linewidth=0.04cm](12.959374,0.9103125)(13.439376,0.8903124)
\psline[linewidth=0.04cm](12.899375,1.8703123)(13.459374,1.8703123)
\psline[linewidth=0.04cm](12.659374,1.5703125)(12.679376,1.1503124)
\psline[linewidth=0.04cm](12.859376,1.6903125)(13.559375,1.0903124)
\usefont{T1}{ptm}{m}{n}
\rput(12.685156,0.9303124){\Large c}
\usefont{T1}{ptm}{m}{n}
\rput(13.667031,1.7703124){\Large b}
\usefont{T1}{ptm}{m}{n}
\rput(12.632188,1.8503125){\Large a}
\usefont{T1}{ptm}{m}{n}
\rput(13.687344,0.9103125){\Large d}
\usefont{T1}{ptm}{m}{n}
\rput(12.785156,2.6503124){\Large c}
\usefont{T1}{ptm}{m}{n}
\rput(13.727032,3.4903123){\Large b}
\usefont{T1}{ptm}{m}{n}
\rput(15.107344,2.8103125){\Large d}
\usefont{T1}{ptm}{m}{n}
\rput(12.612187,3.4903126){\Large a}
\psdots[dotsize=0.32](10.659374,-2.8846874)
\psline[linewidth=0.04cm](10.579374,-2.7646875)(9.8,-1.52)
\psline[linewidth=0.04cm](10.699374,-2.7646875)(11.0,-1.3599999)
\psline[linewidth=0.04cm](10.659374,-2.7446873)(10.44,-1.3399999)
\psarc[linewidth=0.03](10.689376,-2.7946877){0.71}{35.537678}{141.58194}
\psline[linewidth=0.04cm](10.759375,-2.7846875)(11.52,-1.5799999)
\psellipse[linewidth=0.04,dimen=outer](11.799375,-2.2846875)(0.24,0.26)
\psellipse[linewidth=0.04,dimen=outer](12.819375,-2.3046875)(0.26,0.26)
\psellipse[linewidth=0.04,dimen=outer](12.839375,-3.2046876)(0.26,0.26)
\psellipse[linewidth=0.04,dimen=outer](11.889376,-3.2046876)(0.25,0.26)
\psline[linewidth=0.04cm](12.099376,-3.1846876)(12.579374,-3.2046876)
\psline[linewidth=0.04cm](12.039374,-2.2246873)(12.599376,-2.2246873)
\psline[linewidth=0.04cm](11.799375,-2.5246878)(11.819375,-2.9446871)
\psline[linewidth=0.04cm](11.999374,-2.4046874)(12.699374,-3.0046878)
\psline[linewidth=0.04cm](12.78,-2.5)(12.819375,-2.964687)
\usefont{T1}{ptm}{m}{n}
\rput(11.065156,-1.1046876){\Large c}
\usefont{T1}{ptm}{m}{n}
\rput(10.4270315,-1.1846875){\Large b}
\usefont{T1}{ptm}{m}{n}
\rput(12.8073435,-3.1846874){\Large d}
\usefont{T1}{ptm}{m}{n}
\rput(11.825156,-3.1846876){\Large c}
\usefont{T1}{ptm}{m}{n}
\rput(12.807032,-2.3446875){\Large b}
\usefont{T1}{ptm}{m}{n}
\rput(11.772187,-2.3046875){\Large a}
\usefont{T1}{ptm}{m}{n}
\rput(9.672188,-1.3446875){\Large a}
\usefont{T1}{ptm}{m}{n}
\rput(11.587344,-1.2446874){\Large d}
\psdots[dotsize=0.32](5.2793756,-4.6784377)
\psdots[dotsize=0.32](4.8393755,-2.9784374)
\psline[linewidth=0.04cm](4.7993746,-3.0184376)(4.24,-1.85375)
\psline[linewidth=0.04cm](4.8793745,-3.0584376)(5.16,-1.67375)
\psarc[linewidth=0.03](4.8693743,-2.748438){0.71}{47.385944}{148.49573}
\usefont{T1}{ptm}{m}{n}
\rput(4.1673436,-1.5984375){\Large d}
\usefont{T1}{ptm}{m}{n}
\rput(5.1721873,-1.4784375){\Large a}
\psline[linewidth=0.04cm](5.1993756,-4.5584373)(4.82,-2.87375)
\psline[linewidth=0.04cm](5.3193746,-4.5584373)(5.54,-3.11375)
\psarc[linewidth=0.03](5.3093753,-4.5884376){0.71}{35.537678}{141.58194}
\psellipse[linewidth=0.04,dimen=outer](5.9393754,-4.5384374)(0.24,0.26)
\psellipse[linewidth=0.04,dimen=outer](7.2493753,-4.6084375)(0.51,0.57)
\psline[linewidth=0.04cm](6.1593757,-4.5584373)(6.7593756,-4.5584373)
\psline[linewidth=0.04cm](7.7393756,-4.6184373)(8.3393755,-4.6184373)
\psellipse[linewidth=0.04,dimen=outer](8.559375,-4.6384373)(0.26,0.26)
\usefont{T1}{ptm}{m}{n}
\rput(7.2273436,-4.3984375){\Large d}
\usefont{T1}{ptm}{m}{n}
\rput(5.885156,-4.5184374){\Large c}
\usefont{T1}{ptm}{m}{n}
\rput(7.2121873,-4.7784376){\Large a}
\usefont{T1}{ptm}{m}{n}
\rput(8.547031,-4.6784377){\Large b}
\psellipse[linewidth=0.04,dimen=outer](3.9393754,-2.458438)(0.26,0.26)
\psellipse[linewidth=0.04,dimen=outer](2.8593755,-2.5184376)(0.26,0.26)
\psline[linewidth=0.04cm](3.1193745,-2.5384376)(3.7193744,-2.478438)
\usefont{T1}{ptm}{m}{n}
\rput(3.9073439,-2.4584374){\Large d}
\usefont{T1}{ptm}{m}{n}
\rput(2.8121874,-2.4984374){\Large a}
\psline[linewidth=0.04cm](5.3793755,-4.578438)(6.1,-3.39375)
\usefont{T1}{ptm}{m}{n}
\rput(5.6051564,-2.8584375){\Large c}
\usefont{T1}{ptm}{m}{n}
\rput(6.1870313,-3.0784376){\Large b}
\usefont{T1}{ptm}{m}{n}
\rput(6.0375,-6.6196876){\LARGE $\widehat{\eta}(T)=\eta$}
\psellipse[linewidth=0.04,dimen=outer](7.1793747,-7.7396874)(0.24,0.26)
\psellipse[linewidth=0.04,dimen=outer](8.199375,-7.7596874)(0.26,0.26)
\psellipse[linewidth=0.04,dimen=outer](8.219376,-8.659688)(0.26,0.26)
\psellipse[linewidth=0.04,dimen=outer](7.2693753,-8.659688)(0.25,0.26)
\psline[linewidth=0.04cm](8.219376,-7.9996877)(8.239374,-8.379687)
\psline[linewidth=0.04cm](7.4193745,-7.6796875)(7.9793744,-7.6796875)
\psline[linewidth=0.04cm](7.1793747,-7.9796877)(7.1993756,-8.399688)
\psline[linewidth=0.04cm](7.3793755,-7.8596873)(8.079374,-8.459687)
\usefont{T1}{ptm}{m}{n}
\rput(6.5484366,-8.114688){\huge =}
\psline[linewidth=0.04cm](9.119375,-7.0196877)(9.119375,-6.6196876)
\psellipse[linewidth=0.04,dimen=outer](7.8193746,-6.5996876)(0.24,0.26)
\psellipse[linewidth=0.04,dimen=outer](9.129375,-6.7896876)(0.51,0.69)
\psline[linewidth=0.04cm](8.039375,-6.6196876)(8.639376,-6.6196876)
\psline[linewidth=0.04cm](9.619375,-6.6796875)(10.219376,-6.6796875)
\psellipse[linewidth=0.04,dimen=outer](10.439376,-6.6996875)(0.26,0.26)
\usefont{T1}{ptm}{m}{n}
\rput(7.7651563,-6.5996876){\Large c}
\usefont{T1}{ptm}{m}{n}
\rput(9.107344,-6.4196877){\Large d}
\usefont{T1}{ptm}{m}{n}
\rput(9.092188,-7.1796875){\Large a}
\usefont{T1}{ptm}{m}{n}
\rput(7.2051563,-8.619688){\Large c}
\usefont{T1}{ptm}{m}{n}
\rput(8.187031,-7.7796874){\Large b}
\usefont{T1}{ptm}{m}{n}
\rput(7.1521873,-7.6996875){\Large a}
\usefont{T1}{ptm}{m}{n}
\rput(8.207344,-8.639688){\Large d}
\psline[linewidth=0.04cm](7.4993744,-8.639688)(7.9793754,-8.659688)
\usefont{T1}{ptm}{m}{n}
\rput(10.407031,-6.7196875){\Large b}
\psframe[linewidth=0.04,dimen=outer](13.32,0.03)(2.36,-9.41)
\psframe[linewidth=0.04,dimen=outer](15.44,9.41)(0.0,-0.03)
\psline[linewidth=0.04cm](7.76,9.37)(7.78,-0.01)
\psline[linewidth=0.04cm](15.4,9.41)(15.44,0.01)
\psframe[linewidth=0.04,dimen=outer](13.36,0.01)(8.76,-3.67)
\end{pspicture}
}
\end{center}\caption{Schr\"oder trees in the antipode formula for $K_4-e$ in $\CN_{\Gr}$ }\label{antipodegr}
\end{figure}
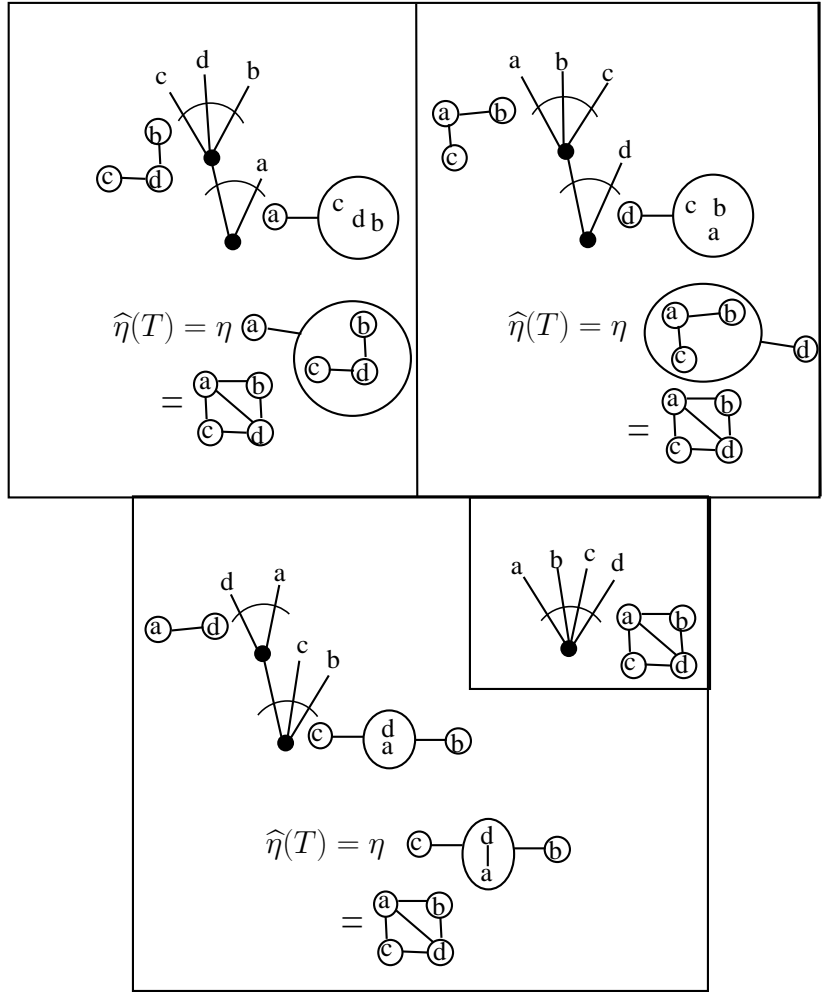

\subsection {Weights and algebra maps}

Let $M$ be an arbitrary connected with $|M[1]|=1$ and   $\mathcal{A}$ a $\KK$-algebra. An
algebra map
$$\omega:\NA\rightarrow \mathcal{A}$$ is called a weight on $M$ and the pair $(M,\omega),$
denoted by $M^{\omega}$, a weighted species. For a finite set $U$, the inventory of
$M^{\omega}[U]$, $|M[U]|_{\omega}$, is defined to be \be |M[U]|_{\omega}=\sum_{m\in
M[U]}\omega(t_{\tau(m)})\eeq

 This definition of weight and weighted species agrees with that given in \cite{B-L-L} if
we add the extra requirement that the structure of $M[1]$ has weight $1$.

The generating function $M^{\omega}(x)$, a formal power series with coefficients in
$\mathcal{A}$, is defined by \be
M^{\omega}(x)=x+\sum_{n=2}^{\infty}|M[n]|_{\omega}\frac{x^n}{n!}.\eeq \noindent Since the
number of structures of type $\alpha$ in $M[n]$ is equal to $\frac{n!}{\aut(\alpha)}$,
\be \label{Mseries}M^{\omega}(x)=x+\sum_{n=2}^{\infty}\left(\sum_{\alpha\in\Tt(M)[n]}
\frac{\omega(t_{\alpha})}{\aut(\alpha)}\right)x^n.\end{equation}

Let $M^{\omega_1}$ and $N^{\omega_2}$ be two $\mathcal{A}$-weighted positive species. The
substitution $M^{\omega_1}(N^{\omega_2})$ is defined as the weighted species
$(M(N),\omega_3)$, $\omega_3$ defined by \be\label{compos} \omega_3
(t_{\tau(a,m)})=\omega_2(t_{\tau(a)})\,.\,\omega_1(m)=\prod_{B\in\pi}
\omega_2(t_{\tau(n_B)})\,.\,\omega_1(t_{\tau(m)}),\eeq \noindent
$(a,m)=(\{n_B\}_{B\in\pi},m)$ being an arbitrary structure of $M(N)$.

The generating series as in Eq. (\ref{Mseries}) will be called {\em $M$-series} with
coefficients in $\mathcal{A}$.

Generating series transform the set theoretical operation of substitution into
substitution of formal power series (see \cite{B-L-L}): \be
M^{\omega_1}(N^{\omega_2})(x)=M^{\omega_1}(N^{\omega_2}(x)).\eeq \noindent

When $(M,\eta)$ is a set-operad, the structure of Hopf algebra on $\NA$ provides the set
of algebra morphisms $\mathrm{Al}(\NA,\mathcal{A})$ with the standard group structure \be
(\omega_1\ast\omega_2)(t_{\alpha})=(\mu\circ(\omega_1\otimes\omega_2)\circ\Delta)(t_{\alpha})=
\sum_{\eta(a,m')=m}\omega_1(t_{\tau(a)}).\omega_2(t_{\tau(m')}),\eeq  $\mu$ being the
product of $\mathcal{A}$. Recall that the identity of this group is
$e_{\mathcal{A}}\circ\epsilon$. The following proposition establishes an interesting link
between substitution of formal power series and the group structure of
$\mathrm{Al}(\NA,\mathcal{A})$.

\begin{prop}
For a set-operad $(M,\eta)$ and two $\mathcal{A}$-weights on $M$, $\omega_1,\ \omega_2$,
we have \be M^{\omega_1}(M^{\omega_2})(x)=M^{\omega_2\ast\omega_1}(x).\eeq
\end{prop}
\begin{proof}We only have to prove that for every $n$,
$|M[n]|_{\omega_2\ast\omega_1}=|M(M)[n]|_{\omega_3},$  $\omega_3$ the weight defined as
in Eq. (\ref{compos}). \begin{eqnarray} |M[n]|_{\omega_2\ast\omega_1}&=&\sum_{m\in
M[n]}(\omega_2\ast\omega_1)(m)=\sum_{m\in
M[n]}\sum_{\eta(a,m')=m}\omega_2(t_{\tau(a)})\,.\,\omega_1(t_{\tau(m')})\\
&=&\sum_{(a,m')\in
M(M)[n]}\omega_2(t_{\tau(a)})\,.\,\omega_1(t_{\tau(m')})=|M(M)[n]|_{\omega_3}\end{eqnarray}
\end{proof}

The previous proposition  can be reformulated as follows:
\begin{prop} For a set-operad $(M,\eta)$, the $M$-generating series with coefficients in $\mathcal{A}$ form a group
with respect to the operation of substitution. This group is anti-isomorphic to the group
of algebra maps $\mathrm{Al}(\NA,\mathcal{A})$.\end{prop}

It is a generalization of proposition in \cite{Joni-Rota}, that relates the Fa\'a di
Bruno Hopf algebra with the substitution of exponential formal power series.

Let us denote by  $(f(x))^{\langle -1\rangle}$ the substitutional inverse of the
generating series $f(x)$. From the previous proposition we obtain the corollary
\begin{coro}\label{coro1}Let $I$ be de identity morphism of $\NA$, and
$S$ the antipode. Then \be M^S(x)=(M^I(x))^{\langle -1\rangle}.\eeq More generally, for a
weight $\omega$ on $M$, \be M^{\omega\circ S}(x)=(M^{\omega}(x))^{\langle -1\rangle}.\eeq

\end{coro}
From the corollary, since $M^{I}(x)=x+M_{2^+}^I(x)$,
$M^I(M^S(x))=M^S(x)+M_{2^+}^I(M^S(x))=x$, and then $M^S(x)$ satisfies the implicit
equation: \be M^S(x)=x-M_{2^+}^I(M^S(x))=x+M_{2^+}^{-I}(M^S(x)). \eeq From this implicit
equation we can recover our antipode formula of Theorem \ref{antipode}.

\section{Antipode Formulas}

\subsection{Fa\'a di Bruno}
For the Fa\'a Di Bruno Hopf algebra $\CN_{E_+}$ (Example \ref{Faa}), $\Delta_{E_+}$ is the map
\begin{eqnarray} \Delta_{E_+}(t_n)&=&\sum_k\left(\sum_{\substack{j_1+2j_2+3j_3+\dots=n\\
j_1+j_2+\dots=k}}\frac{n!}{1!^{j_1}j_1!2!^{j_2}j_2!3!^{j_3}j_3!\dots}t_2^{j_2}t_3^{j_3}\dots\right
)\otimes
t_k\\
&=&\sum_kB_{n,k}(1,t_2,t_3,\dots,t_n)\otimes t_k
\end{eqnarray}

\noindent where $$B_{n,k}(t_1,t_2,\dots,t_n)=\sum_{\substack{\pi\in\Pi[n]\\
|\pi|=k}}\prod_{B\in\pi}t_{|B|}=\sum_{\substack{j_1+2j_2++\dots nj_n=n\\ j_1+j_2+\dots
j_n=k}}\frac{n!}{1!^{j_1}j_1!2!^{j_2}j_2!\dots
n!^{j_n}j_n!}t_1^{j_1}t_2^{j_2}t_3^{j_3}\dots t_n^{j_n}$$ is the partial Bell polynomial.
$\CN_{E_+}$ is the Fa\'a Di Bruno Hopf algebra. The series
$$E_+^I(x)=x+\sum_{n=2}^{\infty}t_n\frac{x^n}{n!}$$ represents a generic exponential delta power
series, and $$E_+^{S}=x+\sum_{n=2}^{\infty}S(t_n)\frac{x^n}{n!}$$ its substitutional
inverse. Then, having a formula for the antipode is equivalent to have a Lagrange
inversion formula (see \cite{Joni-Rota}).
In \cite{Ehrenborg-Mendez} the following proposition was proved:
\begin{prop}\label{biyeccion} Let $N$ be a delta species. Consider the set $\mathcal{F}^{(k)}_{N}[n]$
of $N$-enriched Schr\"oder trees with leaves
$\{1,2,...,n\}$ and  $k$ internal vertices. There is a bijection $\psi$ between
$\mathcal{F}^{(k)}_{N}[n]$ and the set $\gamma_k(N_{2^{+}})[n+k-1]$ of $k$-assemblies of
$N_{2^+}$-structures over the set $[n+k-1]$.
\end{prop}
\begin{remark} The bijection of above generalizes that obtained by Haiman and Schmitt in
\cite{Haimann-Schmitt}. From the construction in \cite{Ehrenborg-Mendez} it is easy to see that
$\psi$ preserves the natural weights on $\mathcal{F}^{(k)}_{N}[n]$ and
$\gamma_k(N_{2^{+}})[n+k-1]$ respectively. More precisely, let $\omega_1$ be the weight
on $\mathcal{F}^{(k)}_{N}[n]$ assigning to each tree $\mathrm{T}$ the monomial
$\prod_{v\in\Iv(\mathrm{T})} t_{\tau(m_v)}$, and $\omega_2$ on
$\gamma_k(N_{2^{+}})[n+k-1]$ assigning the monomial $t_{\tau(a)}$ to each assembly $a$.
Then,
$$\omega_2(\psi(\mathrm{T}))=\omega_1(\mathrm{T})),$$ for every tree $\mathrm{T}\in \mathcal{F}^{(k)}_{N}[n]$.
\end{remark}
From the previous proposition and remark we get:
\begin{eqnarray}
S(t_n)&=&\sum_{\substack{\mathrm{T}\in \mathcal{F}_{E_+}[n]\\
\hat{\eta}(\mathrm{T})=\{[n]\}}}\,(-1)^{\mid \mathrm{Iv(\mathrm{T})}\mid}\,\prod_{v\in
\mathrm{Iv(T)}}t_{|\pi_v|}\nonumber\\
&=&\sum^{n-1}_{k=1}\sum_{\mathrm{T}\in \mathcal{F}^{(k)}_{E_+}[n]}\,(-1)^k\,\prod_{v\in
\mathrm{Iv(T)}}t_{|\pi_v|}\nonumber\\
&=&\sum^{n-1}_{k=1}(-1)^k \sum_{\pi\in
\gamma_k(E_{2+})[n+k-1]}\prod_{B\in\pi}t_{|B|}\\\label{antipofa} &=&\sum^{n-1}_{k=1}
(-1)^k B_{n+k-1, k}(0,t_2,t_3,...).\label{ec2}
\end{eqnarray}

This is the antipode formula obtained by Haiman and Schmitt in \cite{Haimann-Schmitt}.

We now study more readable forms of the antipode for the natural Hopf algebras corresponding to  the operads  $\Ep$, $\Gr$, $\Arb$, $\Arb_{L}$, $\Grp$. The latter three examples being cases of the family of operads $\Arb_{N}$, $N$ being a monoid. All these antipodes are obtained from our general formula by using essentially the same technique. The depth of an internal vertex of a Schr\"{o}der tree is defined to be the number of vertices in the path to the root (including the vertex itself). The technique is this:  color the internal nodes of the corresponding trees in formula (\ref{antipode}) with its depth, beginning with the root which is colored $1$. Each internal node instructs to carry out a product of the operad, until a final result is obtained, after doing all the products and finishing with the root. Since all those examples of operads are graphical (graphs with some extra added structure), the final result is some kind of graph, and each product in the construction of this graph is obtained by connecting some vertices with new edges. The Schr\"oder tree is then codified in the resulting graph by coloring their edges, according to the color (the depth) of the internal node in which they were created. In this way more readable antipode formulas are obtained in all the above mentioned cases. Before doing that, we introduce the notion of prime structure on a set operad.

\begin{defi}\label{prime}
A  element $m\in M[U]$ of an operad $(M,\eta)$ will be called a {\em prime structure} if it only can be factored trivially. More  precisely, if $m=\eta(a,m'_{\pi})$ then either $a$ is the assembly of singletons, $a=\{\bullet_u\}_{u\in U}$, and $m'_{\pi}$ is isomorphic to $m$, or $a=\{m\}$ and $m'_{\{U\}}=\{\bullet_{U}\}$ is the singleton structure with label $U$. 
Clearly, the types of primes of a set-operad are the primitive elements of the corresponding Hopf algebra
\be \Delta(t_{\tau(m)})=1\otimes t_{\tau(m)}+t_{\tau(m)}\otimes 1.\eeq
\end{defi}

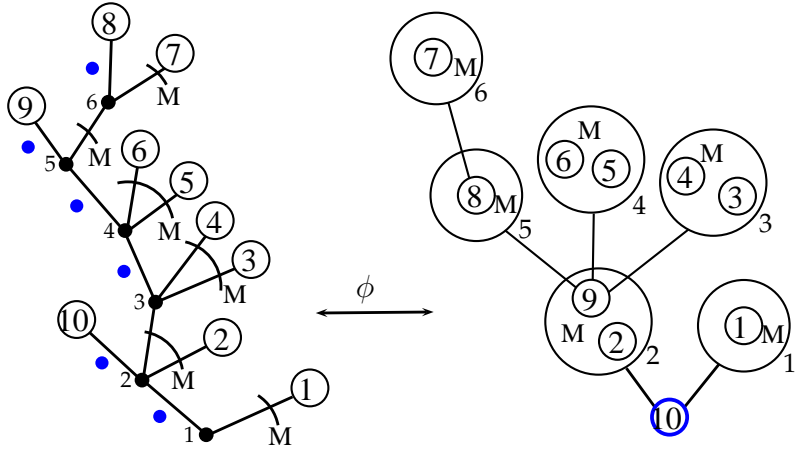
\begin{figure}
\begin{center}
\scalebox{0.7} 
{
\begin{pspicture}(0,-4.2165)(15.019688,4.2165)
\psline[linewidth=0.04cm](8.78,0.85825)(8.36,2.37825)
\psline[linewidth=0.04cm](11.36,-1.34175)(12.94,-0.10175)
\pscircle[linewidth=0.04,dimen=outer,fillstyle=solid](10.59,1.23175){0.37}
\usefont{T1}{ptm}{m}{n}
\rput(10.566875,1.2089375){\LARGE 6}
\pscircle[linewidth=0.04,dimen=outer,fillstyle=solid](11.47,1.05175){0.37}
\usefont{T1}{ptm}{m}{n}
\rput(11.437187,1.0289375){\LARGE 5}
\usefont{T1}{ppl}{m}{n}
\rput(12.012032,0.36675){\Large 4}
\usefont{T1}{ptm}{m}{n}
\rput(11.064375,1.78175){\Large M}
\pscircle[linewidth=0.04,dimen=outer](11.09,1.22825){1.03}
\psline[linewidth=0.054cm](12.78,-3.58175)(13.52,-2.66175)
\psline[linewidth=0.054cm](12.36,-3.55825)(11.76,-2.73825)
\pscircle[linewidth=0.04,dimen=outer,fillstyle=solid](11.09,-1.40825){0.37}
\usefont{T1}{ptm}{m}{n}
\rput(11.066719,-1.4310625){\LARGE 9}
\pscircle[linewidth=0.04,dimen=outer,fillstyle=solid](11.59,-2.22825){0.37}
\usefont{T1}{ptm}{m}{n}
\rput(11.5725,-2.2510624){\LARGE 2}
\usefont{T1}{ptm}{m}{n}
\rput(10.744375,-2.07825){\Large M}
\pscircle[linewidth=0.04,dimen=outer](11.23,-1.85175){1.03}
\usefont{T1}{ppl}{m}{n}
\rput(12.267813,-2.49325){\Large 2}
\pscircle[linewidth=0.04,dimen=outer,fillstyle=solid](8.91,0.54825){0.37}
\usefont{T1}{ptm}{m}{n}
\rput(8.875937,0.5254375){\LARGE 8}
\usefont{T1}{ppl}{m}{n}
\rput(9.846562,-0.09325){\Large 5}
\usefont{T1}{ptm}{m}{n}
\rput(9.464375,0.38175){\Large M}
\pscircle[linewidth=0.04,dimen=outer](8.91,0.54825){0.89}
\pscircle[linewidth=0.04,dimen=outer,fillstyle=solid](13.87,0.53175){0.37}
\usefont{T1}{ptm}{m}{n}
\rput(13.834375,0.5089375){\LARGE 3}
\pscircle[linewidth=0.04,dimen=outer,fillstyle=solid](12.89,0.91175){0.37}
\usefont{T1}{ptm}{m}{n}
\rput(12.87625,0.8889375){\LARGE 4}
\usefont{T1}{ppl}{m}{n}
\rput(14.406406,0.02675){\Large 3}
\usefont{T1}{ptm}{m}{n}
\rput(13.384375,1.34175){\Large M}
\pscircle[linewidth=0.04,dimen=outer](13.41,0.78825){1.03}
\usefont{T1}{ppl}{m}{n}
\rput(9.003593,2.50675){\Large 6}
\pscircle[linewidth=0.04,dimen=outer,fillstyle=solid](8.15,3.14825){0.89}
\pscircle[linewidth=0.04,dimen=outer,fillstyle=solid](13.99,-1.93175){0.37}
\usefont{T1}{ptm}{m}{n}
\rput(13.919687,-1.9545625){\LARGE 1}
\usefont{T1}{ptm}{m}{n}
\rput(14.544375,-2.09825){\Large M}
\usefont{T1}{ppl}{m}{n}
\rput(14.865156,-2.63325){\Large 1}
\pscircle[linewidth=0.04,dimen=outer](13.99,-1.93175){0.89}
\pscircle[linewidth=0.072,linecolor=blue,dimen=outer,fillstyle=solid](12.580078,-3.66825){0.37}
\usefont{T1}{ptm}{m}{n}
\rput(12.509766,-3.6874688){\LARGE 10}
\psline[linewidth=0.054cm,arrowsize=0.05291667cm 2.0,arrowlength=1.4,arrowinset=0.4]{<->}(5.86,-1.6900625)(8.1,-1.6700625)
\usefont{T1}{ptm}{m}{n}
\rput(6.7990627,-1.2300625){\LARGE $\phi$}
\psline[linewidth=0.04cm](11.12,-1.06175)(11.14,0.19825)
\psline[linewidth=0.04cm](10.8,-1.20175)(9.46,-0.12175)
\pscircle[linewidth=0.04,dimen=outer,fillstyle=solid](8.09,3.17175){0.37}
\usefont{T1}{ptm}{m}{n}
\rput(8.067031,3.1489375){\LARGE 7}
\usefont{T1}{ptm}{m}{n}
\rput(8.704375,2.98175){\Large M}
\psdots[dotsize=0.28](1.92,2.3165)
\usefont{T1}{ppl}{m}{n}
\rput(1.6151563,2.3265){\large 6}
\psdots[dotsize=0.28](1.12,1.1365)
\usefont{T1}{ppl}{m}{n}
\rput(0.8176563,1.1465){\large 5}
\psdots[dotsize=0.28](2.82,-1.4835)
\usefont{T1}{ppl}{m}{n}
\rput(2.5176563,-1.4735){\large 3}
\psdots[dotsize=0.28](3.78,-4.0035)
\usefont{T1}{ppl}{m}{n}
\rput(3.458125,-3.9935){\large 1}
\psdots[dotsize=0.28](2.24,-0.1235)
\usefont{T1}{ppl}{m}{n}
\rput(1.9428124,-0.1135){\large 4}
\psdots[dotsize=0.28](2.56,-2.9635)
\usefont{T1}{ppl}{m}{n}
\rput(2.2590625,-2.9535){\large 2}
\psline[linewidth=0.054cm](3.74,-4.0435)(5.42,-3.2835)
\psline[linewidth=0.054cm](3.74,-3.9635)(2.62,-3.0035)
\pscircle[linewidth=0.04,dimen=outer,fillstyle=solid](5.75,-3.1135){0.37}
\usefont{T1}{ptm}{m}{n}
\rput(5.6796875,-3.1363125){\LARGE 1}
\psline[linewidth=0.054cm](2.5,-2.9235)(1.46,-1.9635)
\psline[linewidth=0.054cm](2.56,-2.9635)(3.9,-2.2635)
\pscircle[linewidth=0.04,dimen=outer,fillstyle=solid](4.09,-2.1535){0.37}
\usefont{T1}{ptm}{m}{n}
\rput(4.0725,-2.1763124){\LARGE 2}
\pscircle[linewidth=0.04,dimen=outer,fillstyle=solid](1.3200781,-1.8135){0.37}
\usefont{T1}{ptm}{m}{n}
\rput(1.2497656,-1.8327187){\LARGE 10}
\psline[linewidth=0.054cm](1.12,1.1165)(0.46,2.0365)
\psline[linewidth=0.054cm](1.14,1.1165)(2.18,-0.0835)
\psline[linewidth=0.054cm](2.22,-0.1235)(2.8,-1.4435)
\pscircle[linewidth=0.04,dimen=outer,fillstyle=solid](0.37,2.2465){0.37}
\usefont{T1}{ptm}{m}{n}
\rput(0.34671876,2.2236874){\LARGE 9}
\psline[linewidth=0.054cm](2.28,-0.1035)(2.48,1.1165)
\psline[linewidth=0.054cm](2.3,-0.1235)(3.3,0.6165)
\psline[linewidth=0.054cm](2.86,-1.4235)(3.82,-0.1635)
\psline[linewidth=0.054cm](2.88,-1.4635)(4.38,-0.8235)
\psline[linewidth=0.054cm](2.58,-2.9635)(2.82,-1.4035)
\psline[linewidth=0.054cm](1.14,1.1365)(1.86,2.2365)
\psline[linewidth=0.054cm](1.94,2.3365)(1.98,3.5565)
\psline[linewidth=0.054cm](1.98,2.3165)(3.08,3.0165)
\pscircle[linewidth=0.04,dimen=outer,fillstyle=solid](3.19,3.2465){0.37}
\usefont{T1}{ptm}{m}{n}
\rput(3.1670313,3.2236874){\LARGE 7}
\pscircle[linewidth=0.04,dimen=outer,fillstyle=solid](1.97,3.8465){0.37}
\usefont{T1}{ptm}{m}{n}
\rput(1.9359375,3.8236876){\LARGE 8}
\pscircle[linewidth=0.04,dimen=outer,fillstyle=solid](2.53,1.4265){0.37}
\usefont{T1}{ptm}{m}{n}
\rput(2.506875,1.4036875){\LARGE 6}
\pscircle[linewidth=0.04,dimen=outer,fillstyle=solid](3.45,0.8265){0.37}
\usefont{T1}{ptm}{m}{n}
\rput(3.4171875,0.8036875){\LARGE 5}
\pscircle[linewidth=0.04,dimen=outer,fillstyle=solid](3.95,0.0665){0.37}
\usefont{T1}{ptm}{m}{n}
\rput(3.93625,0.0436875){\LARGE 4}
\pscircle[linewidth=0.04,dimen=outer,fillstyle=solid](4.63,-0.6735){0.37}
\usefont{T1}{ptm}{m}{n}
\rput(4.594375,-0.6963125){\LARGE 3}
\psarc[linewidth=0.054](3.19,-1.1865){0.87}{-5.013114}{77.19573}
\usefont{T1}{ptm}{m}{n}
\rput(4.324375,-1.3365){\Large M}
\psarc[linewidth=0.054](2.29,-0.0665){0.87}{11.821488}{102.99461}
\usefont{T1}{ptm}{m}{n}
\rput(3.084375,-0.1365){\Large M}
\psarc[linewidth=0.054](2.25,2.2535){0.87}{27.07208}{62.878696}
\psarc[linewidth=0.054](0.93,1.1335){0.87}{27.07208}{62.878696}
\psarc[linewidth=0.054](4.29,-3.9865){0.87}{10.840305}{57.01148}
\usefont{T1}{ptm}{m}{n}
\rput(5.184375,-4.0365){\Large M}
\usefont{T1}{ptm}{m}{n}
\rput(3.064375,2.4235){\Large M}
\usefont{T1}{ptm}{m}{n}
\rput(1.764375,1.2635){\Large M}
\usefont{T1}{ptm}{m}{n}
\rput(3.344375,-2.9765){\Large M}
\psarc[linewidth=0.054](2.47,-2.9065){0.87}{9.6887865}{80.75389}
\psdots[dotsize=0.25,linecolor=blue](2.9,-3.6565)
\psdots[dotsize=0.25,linecolor=blue](1.8,-2.6365)
\psdots[dotsize=0.25,linecolor=blue](2.22,-0.8965)
\psdots[dotsize=0.25,linecolor=blue](1.32,0.3435)
\psdots[dotsize=0.25,linecolor=blue](0.4,1.4635)
\psdots[dotsize=0.25,linecolor=blue](1.62,2.9635)
\end{pspicture}
}\end{center}\caption{Isomorphism $\phi:\CF_{X.M}\leftrightarrow \Arb_{L(M)}$}
\label{biyectionsroder}\end{figure}

We recall the bijection between Schr\"oder trees enriched with a species of the form $XM$, $M$ a positive species, and the rooted trees enriched with the species $L(M)$ of linear orders of $M$-structures (see \cite{Ehrenborg-Mendez}). This bijection  $\phi$ establishes an isomorphism between the species $\CF_{XM}$ and $\Arb_{L(M)}$. It goes as follows. An  $XM$-enriched tree $\TT\in\CF_{XM}[U]$ is a rooted tree where each internal node $v \in \Iv(\TT)$ is
decorated with an element of $XM[\pi_v]$, $\pi_v$ standing for the set of sons of $v$. So, at each internal vertex $v$, a preferred son $p(v)$ is chosen and an element $m_v$ is placed on the rest of them $(m_v \in M[\pi_v-\{p(v)\}])$. By starting at the root of $\TT$,
we can construct a unique path from the root to a leaf (the distinguished leaf of $\TT$), by letting the
successor of an internal node be its preferred son. We shall call this path the main spine of $\TT$. Contract all the vertices in the main spine placing the leaf in to the root. We get that the fiber of the root becomes enriched with a tuple of structures in $M(\CF_{XM})$, one for each internal vertex in the main spine. Then, apply the same algorithm recursively over  all the remaining internal Schr\"{o}der trees (see Fig. \ref{biyectionsroder}). This correspondence is clearly reversible. 

To make evident the linear order on $M$-structures in each fiber of the rooted tree, we color them with the depth of the corresponding internal vertex of the Schr\"oder tree. When $M$ is a graphic species, the colors can be transported into the edges of the graph.

\subsection{Pointed Fa\'a di Bruno}
We call $\CN_{\Eb}$ the pointed Fa\'a di Bruno Hopf algebra (Example \ref{pointed}). 
From the formula for the coproduct
\begin{equation}\label{coproductpoint}
\Delta(nt_n)=\sum_{k=1}^{n}B_{n,k}(0,2t_2,3t_3,\dots)\otimes k t_k,\end{equation}
\noindent $\CN_{\Eb}$ is isomorphic to the Fa\'a Di Bruno Hopf algebra by identifying $nt_n$ with the generator $t_n$ of $\CN_{E_+}$.

We have the following identity between Bell polynomials: \be
\label{belles}
B_{n,k}(0,2t_2,3t_3,4t_4\dots)=\binom{n}{k}k!B_{n-k,k}(t_2,t_3,t_4\dots),\eeq \noindent
This identity is easy to prove: the left hand side is the inventory of pointed partitions
on $[n]$ with $k$ blocks, none of cardinal $1$. The right hand side is the same inventory, because there are $\binom{n}{k}$ ways of choosing the distinguished elements of the
blocks in each pointed partition, $B_{n-k,k}(t_2,t_3,\dots)$ is the inventory of ordinary
partitions with $k$ blocks, on the remaining $[n-k]$ elements (giving each block $B$ the
weight $t_{|B|+1}$). Finally, $k!$ is the number of ways of placing the distinguished
elements, one in each block of the ordinary partition.

Then, from Eq. (\ref{coproductpoint}),
\begin{equation*}
n\Delta_{\Eb}(t_n)=\sum_{k=1}^{n}\binom{n}{k}k!B_{n-k,k}(t_2,t_3,t_4\dots)\otimes k t_k.
\end{equation*}
The generating function $(\Eb)^I(x)$ is equal to \be
(\Eb)^I(x)=x+\sum_{n=2}^{\infty}nt_n\frac{x^n}{n!}=x\left(1+\sum_{n=1}^{\infty}t_{n+1}\frac{x^n}{n!}\right)\eeq

By corollary \ref{coro1}, $nS(t_n)$ is the coefficient of $\frac{x^n}{n!}$ in the inverse
of $(\Eb)^I(x)$. 
Since $\Ep_{2+}=XE_+,$  Schr\"oder trees in the antipode formula are elements of $\CF_{XE_+}$ which is isomorphic via  $\phi$ to $\Arb_{L(E_+)}$.  
 
 An element of $\Arb_{L(E_+)}[n]$ is a tree $T$ on $n$ vertices, whose fibers are enriched with tuples of sets. The sets form a partition $\pi$ of $[n]-\{r\}$, $r$ being the root of $T$. These kind of trees are called ordered trees on partitions and were enumerated in \cite[Corollary 3.2]{Ehrenborg-Mendez}. The number of such trees is $(n)^{(k)}=n(n+1)\dots(n+k-1)=(n+k-1)_k$, $k$ being the number of blocks of $\pi$. Each corolla decorating the original Shr\"oder tree corresponds uniquely to a block of the partition plus the vertex where it is attached in $T$. Hence, $B_{n-1,k}(t_2,t_3,\dots,t_{n-1})$
is the inventory of the corollas associated to the blocks of the partition and we have
 
\begin{eqnarray}
n S(t_n)&=&\sum_{k=1}^{n-1}(-1)^k
(n+k-1)_{k} B_{n-1,k}(t_2,t_3,t_4,\dots)\\&=&\sum_{k=1}^{n-1}(-1)^k
\binom{n+k-1}{k}k!B_{n-1,k}(t_2,t_3,t_4,\dots)\\&=&\sum_{k=1}^{n-1}(-1)^k
\binom{n+k-1}{k}\left[\frac{x^{n-1}}{(n-1)!}\right]\left(\sum_{j=1}^{\infty}t_{j+1}\frac{x^j}{j!}\right)^k
\\&=&\left[\frac{x^{n-1}}{(n-1)!}\right]\sum_{k=1}^{n-1}
\binom{-n\,}{\,k}\left(\sum_{j=1}^{\infty}t_{j+1}\frac{x^j}{j!}\right)^k\\&=&\left[\frac{x^{n-1}}{(n-1)!}\right]
\left(1+\sum_{j=1}^{\infty}t_{j+1}\frac{x^j}{j!}\right)^{-n}\\&=&\left[\frac{x^{n-1}}{(n-1)!}\right]
\left(\frac{(\Eb)^I(x)}{x}\right)^{-n}.
\end{eqnarray}
Which is the classical Lagrange inversion formula. The relationship of the trees $\Arb_{L(E_+)}$ with the classical Lagrange inversion formula was first pointed out in \cite{ArbolesChen}.

\subsection{The family  of Hopf algebras \texorpdfstring{$\CN_{\Arb_{N}}$}{NAN}}.
We study here the antipode for the Hopf algebras corresponding to the operad of $N$-enriched rooted trees, $N$ being a monoid. This gives us a family of Hopf algebras, the simplest of them being the cases $N=E$ and $N=L$, giving rise respectively to the Hopf algebra of rooted trees and to the Hopf algebra of planar rooted trees respectively. Because of the isomorphism $$\Grp=\Arb_{E(\Bc')},$$ $\Bc$  being the species of biconnected graphs, other distinguished (but not evident) member of the family is the natural Hopf algebra of pointed connected graphs which we study in a separate subsection. 

\subsubsection{The antipode of \texorpdfstring{$\CN_{\Arb}$}{CA}.}
Now we consider the  case of the operad $\Arb_{E}=\Arb$. Denote by $\mathfrak{a}$ the species of (nonrooted) trees. Since $\Arb$ is obtained by pointing $\mathfrak{a},$ $\Arb=\mathfrak{a}^{\bullet}=X.\mathfrak{a}'$, the natural transformation $\phi$ goes from $\CF_{\Arb_{2^+}}$ to $\Arb_{L(\mathfrak{a}'_+)}$, the species of trees enriched with tuples of structures of $\mathfrak{a}'_+$. Each component of the tuple in the fiber of a vertex $v$ is a tree in $\mathfrak{a}'_+$, which together with $v$ becomes a rooted tree.  Coloring the edges of each of this rooted trees which are components of the fibers  with the depth of the corresponding internal vertex we obtain what we call an admissible coloration.

\begin{defi}\label{admisibleNAP}
Let $T$ be a rooted tree. An {\em admissible edge coloration} of $T$ is a function from the set of edges of $T$, $\mathscr{E}(T)$,  to the set $\PP$ of positive integers, that satisfies
\begin{enumerate}
\item There is at least on edge of color $1$.
\item $c$ is weakly increasing in any path from the root to the leaves.
\item The root of each connected component of the subgraph  induced by edges colored $i$, $i\geq 2$, has at least one incident edge of color $i-1$. 
\end{enumerate}

\end{defi}
The elements of $\Arb_{L(\mathfrak{a}')}$ are then equivalently described as  trees colored with admissible colorations.
The following diagram commutes
\begin{center}
\begin{equation}\label{diagrama1}
\begin{tikzpicture}
  \node (C) {$\CF_{\Arb_{2^+}}$};
  \node (P) [below of=C] {$\Arb$};
  \node (Ai) [right of=C] {$\Arb_{L(\mathfrak{a}'_+)}$};
  \draw[<->] (C) to node {$\phi$} (Ai);
  \draw[->] (C) to node [swap] {$\widehat{\eta}$} (P);
  \draw[<-] (P) to node [swap] {$\upsilon$} (Ai);
\end{tikzpicture}
\end{equation} 
\end{center}

\begin{figure}
\begin{center}

\scalebox{0.6} 
{
\begin{pspicture}(0,-7.828922)(15.879687,7.828922)
\psline[linewidth=0.04cm](12.086062,-3.8266406)(11.806625,-2.692422)
\psline[linewidth=0.054cm](13.046625,-7.188922)(12.366625,-6.232422)
\psline[linewidth=0.04cm](13.266625,-7.112422)(13.261664,-6.212422)
\psline[linewidth=0.04cm](12.046625,-5.692422)(10.946625,-4.292422)
\psline[linewidth=0.04cm](12.186625,-5.652422)(12.146063,-4.126641)
\psline[linewidth=0.04cm](10.686625,-3.7324219)(10.266625,-2.212422)
\psline[linewidth=0.04cm](12.346625,-5.732422)(13.16,-4.2266407)
\pscircle[linewidth=0.04,dimen=outer,fillstyle=solid](11.716625,-2.418922){0.37}
\usefont{T1}{ptm}{m}{n}
\rput(11.6935,-2.4417343){\LARGE 6}
\pscircle[linewidth=0.04,dimen=outer,fillstyle=solid](12.176625,-4.038922){0.37}
\usefont{T1}{ptm}{m}{n}
\rput(12.143812,-4.061734){\LARGE 5}
\usefont{T1}{ppl}{m}{n}
\rput(12.438656,-4.7439218){\Large 4}
\pscircle[linewidth=0.04,dimen=outer,fillstyle=solid](12.216625,-5.9689217){0.37}
\usefont{T1}{ptm}{m}{n}
\rput(12.193344,-5.9917345){\LARGE 9}
\pscircle[linewidth=0.04,dimen=outer,fillstyle=solid](10.816625,-4.042422){0.37}
\usefont{T1}{ptm}{m}{n}
\rput(10.782562,-4.065234){\LARGE 8}
\usefont{T1}{ppl}{m}{n}
\rput(11.573188,-4.703922){\Large 5}
\pscircle[linewidth=0.04,dimen=outer,fillstyle=solid](14.436625,-4.628922){0.37}
\usefont{T1}{ptm}{m}{n}
\rput(14.401,-4.6517344){\LARGE 3}
\pscircle[linewidth=0.04,dimen=outer,fillstyle=solid](13.236625,-4.088922){0.37}
\usefont{T1}{ptm}{m}{n}
\rput(13.222875,-4.1117344){\LARGE 4}
\usefont{T1}{ppl}{m}{n}
\rput(10.730219,-3.063922){\Large 6}
\pscircle[linewidth=0.04,dimen=outer,fillstyle=solid](10.296625,-2.3189218){0.37}
\usefont{T1}{ptm}{m}{n}
\rput(10.273656,-2.3417344){\LARGE 7}
\usefont{T1}{ppl}{m}{n}
\rput(12.178657,-3.023922){\Large 4}
\psline[linewidth=0.054cm](13.466625,-7.212422)(14.206625,-6.292422)
\pscircle[linewidth=0.04,dimen=outer,fillstyle=solid](13.251664,-5.848922){0.37}
\usefont{T1}{ptm}{m}{n}
\rput(13.234164,-5.871734){\LARGE 2}
\usefont{T1}{ppl}{m}{n}
\rput(12.894438,-6.563922){\Large 2}
\usefont{T1}{ppl}{m}{n}
\rput(13.173031,-4.723922){\Large 3}
\pscircle[linewidth=0.04,dimen=outer,fillstyle=solid](14.356625,-5.962422){0.37}
\usefont{T1}{ptm}{m}{n}
\rput(14.286312,-5.9852343){\LARGE 1}
\usefont{T1}{ppl}{m}{n}
\rput(14.111781,-6.7639217){\Large 1}
\pscircle[linewidth=0.072,linecolor=blue,dimen=outer,fillstyle=solid](13.251664,-7.298922){0.37}
\usefont{T1}{ptm}{m}{n}
\rput(13.181352,-7.3181405){\LARGE 10}
\usefont{T1}{ppl}{m}{n}
\rput(13.833032,-5.163922){\Large 3}
\psline[linewidth=0.04cm](12.526625,-5.792422)(14.086625,-4.732422)
\usefont{T1}{ppl}{m}{n}
\rput(13.474438,-6.543922){\Large 2}
\psline[linewidth=0.04cm](15.0060625,-4.2066407)(14.986062,-4.1866407)
\psline[linewidth=0.054cm,arrowsize=0.05291667cm 2.0,arrowlength=1.4,arrowinset=0.4]{<-}(6.8405623,-5.3007345)(9.080563,-5.2807345)
\usefont{T1}{ptm}{m}{n}
\rput(7.839625,-4.9007344){\LARGE $\upsilon$}
\usefont{T1}{ptm}{m}{n}
\rput{-270.0}(11.346422,-13.646265){\rput(12.464781,-1.1624218){\LARGE $=$}}
\psline[linewidth=0.04cm](14.78,1.8889219)(14.78,1.5089219)
\psline[linewidth=0.04cm](13.82,4.608922)(14.04,4.288922)
\psline[linewidth=0.04cm](11.96,4.4089217)(11.96,4.7533593)
\psline[linewidth=0.04cm](11.94,5.188922)(11.94,5.5533595)
\pscircle[linewidth=0.04,dimen=outer,fillstyle=solid](11.95,5.862422){0.37}
\usefont{T1}{ptm}{m}{n}
\rput(11.926875,5.839609){\LARGE 6}
\pscircle[linewidth=0.04,dimen=outer,fillstyle=solid](11.95,5.002422){0.37}
\usefont{T1}{ptm}{m}{n}
\rput(11.917188,4.9796095){\LARGE 5}
\psline[linewidth=0.04cm](9.64,4.648922)(9.22,6.168922)
\psline[linewidth=0.04cm](11.84,2.2489219)(13.8,3.688922)
\usefont{T1}{ppl}{m}{n}
\rput(12.472032,3.797422){\Large 4}
\pscircle[linewidth=0.04,dimen=outer](11.96,5.128922){1.16}
\psline[linewidth=0.054cm](13.64,0.20892188)(14.38,1.1289219)
\psline[linewidth=0.054cm](13.22,0.23242188)(12.62,1.0524219)
\pscircle[linewidth=0.04,dimen=outer,fillstyle=solid](11.57,2.0724218){0.37}
\usefont{T1}{ptm}{m}{n}
\rput(11.546719,2.0496094){\LARGE 9}
\pscircle[linewidth=0.04,dimen=outer,fillstyle=solid](12.53,2.0724218){0.37}
\usefont{T1}{ptm}{m}{n}
\rput(12.5125,2.0496094){\LARGE 2}
\pscircle[linewidth=0.04,dimen=outer](12.09,1.9389219){1.03}
\usefont{T1}{ppl}{m}{n}
\rput(13.127812,1.2974219){\Large 2}
\pscircle[linewidth=0.04,dimen=outer,fillstyle=solid](9.77,4.338922){0.37}
\usefont{T1}{ptm}{m}{n}
\rput(9.735937,4.316109){\LARGE 8}
\usefont{T1}{ppl}{m}{n}
\rput(10.706562,3.6974218){\Large 5}
\pscircle[linewidth=0.04,dimen=outer](9.77,4.338922){0.89}
\pscircle[linewidth=0.04,dimen=outer,fillstyle=solid](14.73,4.8924217){0.37}
\usefont{T1}{ptm}{m}{n}
\rput(14.694375,4.8696094){\LARGE 3}
\pscircle[linewidth=0.04,dimen=outer,fillstyle=solid](13.83,4.8924217){0.37}
\usefont{T1}{ptm}{m}{n}
\rput(13.81625,4.8696094){\LARGE 4}
\usefont{T1}{ppl}{m}{n}
\rput(15.266406,3.817422){\Large 3}
\pscircle[linewidth=0.04,dimen=outer](14.27,4.578922){1.03}
\usefont{T1}{ppl}{m}{n}
\rput(9.843594,6.237422){\Large 6}
\pscircle[linewidth=0.04,dimen=outer,fillstyle=solid](9.01,6.938922){0.89}
\pscircle[linewidth=0.04,dimen=outer,fillstyle=solid](14.77,2.178922){0.37}
\usefont{T1}{ptm}{m}{n}
\rput(14.699688,2.1561093){\LARGE 1}
\usefont{T1}{ppl}{m}{n}
\rput(15.725156,1.1574218){\Large 1}
\pscircle[linewidth=0.04,dimen=outer](14.85,1.8589219){0.89}
\pscircle[linewidth=0.072,linecolor=blue,dimen=outer,fillstyle=solid](13.440078,0.122421876){0.37}
\usefont{T1}{ptm}{m}{n}
\rput(13.369765,0.103203125){\LARGE 10}
\psline[linewidth=0.054cm,arrowsize=0.05291667cm 2.0,arrowlength=1.4,arrowinset=0.4]{<->}(6.42,3.2606094)(8.66,3.2806094)
\usefont{T1}{ptm}{m}{n}
\rput(7.3590627,3.6806095){\LARGE $\phi$}
\psline[linewidth=0.04cm](11.66,2.408922)(12.0,3.9889219)
\psline[linewidth=0.04cm](11.36,2.388922)(10.32,3.668922)
\psline[linewidth=0.04cm](8.95,7.088922)(9.08,6.548922)
\psdots[dotsize=0.4,dotstyle=asterisk](9.15,6.3689218)
\psdots[dotsize=0.4,dotstyle=asterisk](12.06,1.3489219)
\pscircle[linewidth=0.04,dimen=outer,fillstyle=solid](8.95,7.362422){0.37}
\usefont{T1}{ptm}{m}{n}
\rput(8.9270315,7.339609){\LARGE 7}
\psdots[dotsize=0.4,dotstyle=asterisk](14.24,4.168922)
\psdots[dotsize=0.4,dotstyle=asterisk](11.96,4.208922)
\psline[linewidth=0.04cm](14.66,4.568922)(14.4,4.328922)
\psline[linewidth=0.04cm](11.76,1.7889218)(11.94,1.5089219)
\psline[linewidth=0.04cm](12.4,1.7489219)(12.22,1.5289218)
\psdots[dotsize=0.4,dotstyle=asterisk](14.77,1.2889218)
\psdots[dotsize=0.28](1.92,5.851609)
\usefont{T1}{ppl}{m}{n}
\rput(1.6151563,5.8616095){\large 6}
\psdots[dotsize=0.28](1.12,4.6716094)
\usefont{T1}{ppl}{m}{n}
\rput(0.8176563,4.681609){\large 5}
\psdots[dotsize=0.28](2.82,2.0516093)
\usefont{T1}{ppl}{m}{n}
\rput(2.5176563,2.0616093){\large 3}
\psdots[dotsize=0.28](3.78,-0.4683906)
\usefont{T1}{ppl}{m}{n}
\rput(3.458125,-0.45839062){\large 1}
\psdots[dotsize=0.28](2.24,3.4116094)
\usefont{T1}{ppl}{m}{n}
\rput(1.9428124,3.4216094){\large 4}
\psdots[dotsize=0.28](2.56,0.5716094)
\usefont{T1}{ppl}{m}{n}
\rput(2.2590625,0.58160937){\large 2}
\psline[linewidth=0.054cm](3.74,-0.5083906)(5.42,0.25160939)
\psline[linewidth=0.054cm](3.74,-0.42839062)(2.62,0.53160936)
\pscircle[linewidth=0.04,dimen=outer,fillstyle=solid](5.75,0.42160937){0.37}
\usefont{T1}{ptm}{m}{n}
\rput(5.6796875,0.3987969){\LARGE 1}
\psline[linewidth=0.054cm](2.5,0.6116094)(1.46,1.5716094)
\psline[linewidth=0.054cm](2.56,0.5716094)(3.9,1.2716094)
\pscircle[linewidth=0.04,dimen=outer,fillstyle=solid](4.09,1.3816093){0.37}
\usefont{T1}{ptm}{m}{n}
\rput(4.0725,1.3587968){\LARGE 2}
\pscircle[linewidth=0.04,dimen=outer,fillstyle=solid](1.3200781,1.7216094){0.37}
\usefont{T1}{ptm}{m}{n}
\rput(1.2497656,1.7023907){\LARGE 10}
\psline[linewidth=0.054cm](1.12,4.6516094)(0.46,5.5716095)
\psline[linewidth=0.054cm](1.14,4.6516094)(2.18,3.4516094)
\psline[linewidth=0.054cm](2.22,3.4116094)(2.8,2.0916095)
\pscircle[linewidth=0.04,dimen=outer,fillstyle=solid](0.37,5.7816095){0.37}
\usefont{T1}{ptm}{m}{n}
\rput(0.34671876,5.7587967){\LARGE 9}
\psline[linewidth=0.054cm](2.28,3.4316094)(2.48,4.6516094)
\psline[linewidth=0.054cm](2.3,3.4116094)(3.3,4.1516094)
\psline[linewidth=0.054cm](2.86,2.1116095)(3.82,3.3716094)
\psline[linewidth=0.054cm](2.88,2.0716093)(4.38,2.7116094)
\psline[linewidth=0.054cm](2.58,0.5716094)(2.82,2.1316094)
\psline[linewidth=0.054cm](1.14,4.6716094)(1.86,5.7716093)
\psline[linewidth=0.054cm](1.94,5.871609)(1.98,7.0916095)
\psline[linewidth=0.054cm](1.98,5.851609)(3.08,6.5516095)
\pscircle[linewidth=0.04,dimen=outer,fillstyle=solid](3.19,6.7816095){0.37}
\usefont{T1}{ptm}{m}{n}
\rput(3.1670313,6.7587967){\LARGE 7}
\pscircle[linewidth=0.04,dimen=outer,fillstyle=solid](1.97,7.3816094){0.37}
\usefont{T1}{ptm}{m}{n}
\rput(1.9359375,7.358797){\LARGE 8}
\pscircle[linewidth=0.04,dimen=outer,fillstyle=solid](2.53,4.9616094){0.37}
\usefont{T1}{ptm}{m}{n}
\rput(2.506875,4.938797){\LARGE 6}
\pscircle[linewidth=0.04,dimen=outer,fillstyle=solid](3.45,4.3616095){0.37}
\usefont{T1}{ptm}{m}{n}
\rput(3.4171875,4.338797){\LARGE 5}
\pscircle[linewidth=0.04,dimen=outer,fillstyle=solid](3.95,3.6016095){0.37}
\usefont{T1}{ptm}{m}{n}
\rput(3.93625,3.5787969){\LARGE 4}
\pscircle[linewidth=0.04,dimen=outer,fillstyle=solid](4.63,2.8616095){0.37}
\usefont{T1}{ptm}{m}{n}
\rput(4.594375,2.8387969){\LARGE 3}
\psarc[linewidth=0.054](3.39,2.4086094){0.87}{-5.013114}{77.19573}
\psarc[linewidth=0.054](2.23,3.4086094){0.87}{25.347755}{99.812035}
\psarc[linewidth=0.054](2.25,5.7886095){0.87}{27.07208}{62.878696}
\psarc[linewidth=0.054](0.93,4.668609){0.87}{27.07208}{62.878696}
\psarc[linewidth=0.054](4.29,-0.45139062){0.87}{10.840305}{57.01148}
\psarc[linewidth=0.054](2.47,0.62860936){0.87}{9.6887865}{80.75389}
\psdots[dotsize=0.25,linecolor=blue](2.9,-0.121390626)
\psdots[dotsize=0.25,linecolor=blue](1.8,0.8986094)
\psdots[dotsize=0.25,linecolor=blue](2.22,2.6386094)
\psdots[dotsize=0.25,linecolor=blue](1.32,3.8786094)
\psdots[dotsize=0.25,linecolor=blue](0.4,4.9986095)
\psdots[dotsize=0.25,linecolor=blue](1.62,6.4986095)
\pscircle[linewidth=0.04,dimen=outer](3.3602207,6.103359){0.15}
\psdots[dotsize=0.36,dotstyle=asterisk](3.3602207,5.4733596)
\psline[linewidth=0.04cm](3.36,5.9733596)(3.36,5.6733594)
\pscircle[linewidth=0.04,dimen=outer](1.9402206,5.4033594){0.15}
\psdots[dotsize=0.36,dotstyle=asterisk](1.9402206,4.7733593)
\psline[linewidth=0.04cm](1.94,5.2733593)(1.94,4.9733596)
\psdots[dotsize=0.36,dotstyle=asterisk](2.9454408,2.5733595)
\psline[linewidth=0.04cm](2.9454412,3.5133593)(2.9454412,3.1933594)
\psline[linewidth=0.04cm](2.9454412,3.1133595)(2.9454412,2.7933593)
\pscircle[linewidth=0.04,dimen=outer,fillstyle=solid](2.9454412,3.6033595){0.15}
\pscircle[linewidth=0.04,dimen=outer,fillstyle=solid](2.9454412,3.1633594){0.15}
\pscircle[linewidth=0.04,dimen=outer](4.2602205,2.2233593){0.15}
\psdots[dotsize=0.36,dotstyle=asterisk](4.5802207,1.6733594)
\psline[linewidth=0.04cm](4.32,2.1133595)(4.54,1.7733594)
\pscircle[linewidth=0.04,dimen=outer](4.960221,2.2433593){0.15}
\psline[linewidth=0.04cm](4.86,2.1333594)(4.64,1.7933594)
\pscircle[linewidth=0.04,dimen=outer](5.420221,-0.33664063){0.15}
\psdots[dotsize=0.36,dotstyle=asterisk](5.420221,-0.96664065)
\psline[linewidth=0.04cm](5.42,-0.46664062)(5.42,-0.7666406)
\psline[linewidth=0.04cm,arrowsize=0.05291667cm 2.0,arrowlength=1.4,arrowinset=0.4]{->}(1.98,-0.5866406)(1.98,-1.8666406)
\usefont{T1}{ptm}{m}{n}
\rput(2.2090626,-1.1866406){\LARGE $\widehat{\eta}$}
\pscircle[linewidth=0.04,dimen=outer](3.3802207,0.58335936){0.15}
\psdots[dotsize=0.36,dotstyle=asterisk](3.7002206,0.033359375)
\psline[linewidth=0.04cm](3.44,0.47335938)(3.66,0.13335937)
\pscircle[linewidth=0.04,dimen=outer](4.0802207,0.6033594){0.15}
\psline[linewidth=0.04cm](3.98,0.4933594)(3.76,0.15335937)
\psline[linewidth=0.04cm](2.7794375,-3.9866407)(2.5,-2.8524218)
\psline[linewidth=0.054cm](3.74,-7.348922)(3.06,-6.3924217)
\psline[linewidth=0.04cm](3.96,-7.272422)(3.955039,-6.3724217)
\psline[linewidth=0.04cm](2.74,-5.8524218)(1.64,-4.4524217)
\psline[linewidth=0.04cm](2.88,-5.812422)(2.8394375,-4.2866406)
\psline[linewidth=0.04cm](1.38,-3.892422)(0.96,-2.372422)
\psline[linewidth=0.04cm](2.96,-5.9524217)(3.9,-4.4266405)
\pscircle[linewidth=0.04,dimen=outer,fillstyle=solid](2.41,-2.5789218){0.37}
\usefont{T1}{ptm}{m}{n}
\rput(2.386875,-2.6017344){\LARGE 6}
\pscircle[linewidth=0.04,dimen=outer,fillstyle=solid](2.87,-4.1989217){0.37}
\usefont{T1}{ptm}{m}{n}
\rput(2.8371875,-4.2217345){\LARGE 5}
\pscircle[linewidth=0.04,dimen=outer,fillstyle=solid](2.91,-6.128922){0.37}
\usefont{T1}{ptm}{m}{n}
\rput(2.8867188,-6.1517344){\LARGE 9}
\pscircle[linewidth=0.04,dimen=outer,fillstyle=solid](1.51,-4.2024217){0.37}
\usefont{T1}{ptm}{m}{n}
\rput(1.4759375,-4.2252345){\LARGE 8}
\pscircle[linewidth=0.04,dimen=outer,fillstyle=solid](5.13,-4.788922){0.37}
\usefont{T1}{ptm}{m}{n}
\rput(5.094375,-4.811734){\LARGE 3}
\pscircle[linewidth=0.04,dimen=outer,fillstyle=solid](4.01,-4.268922){0.37}
\usefont{T1}{ptm}{m}{n}
\rput(3.99625,-4.291734){\LARGE 4}
\pscircle[linewidth=0.04,dimen=outer,fillstyle=solid](0.99,-2.478922){0.37}
\usefont{T1}{ptm}{m}{n}
\rput(0.96703124,-2.5017343){\LARGE 7}
\psline[linewidth=0.054cm](4.16,-7.3724217)(4.9,-6.4524217)
\pscircle[linewidth=0.04,dimen=outer,fillstyle=solid](3.945039,-6.008922){0.37}
\usefont{T1}{ptm}{m}{n}
\rput(3.927539,-6.0317345){\LARGE 2}
\pscircle[linewidth=0.04,dimen=outer,fillstyle=solid](5.05,-6.1224217){0.37}
\usefont{T1}{ptm}{m}{n}
\rput(4.9796877,-6.1452346){\LARGE 1}
\pscircle[linewidth=0.072,linecolor=blue,dimen=outer,fillstyle=solid](3.945039,-7.458922){0.37}
\usefont{T1}{ptm}{m}{n}
\rput(3.8747265,-7.478141){\LARGE 10}
\psline[linewidth=0.04cm](3.22,-5.9524217)(4.78,-4.8924217)
\psline[linewidth=0.04cm](5.68,-4.978922)(5.66,-4.958922)
\end{pspicture}
}\end{center}\caption{A Schr\"oder tree on $\Arb$ and isomorphic admissible colorations}\label{Schroder-colored}
\end{figure}
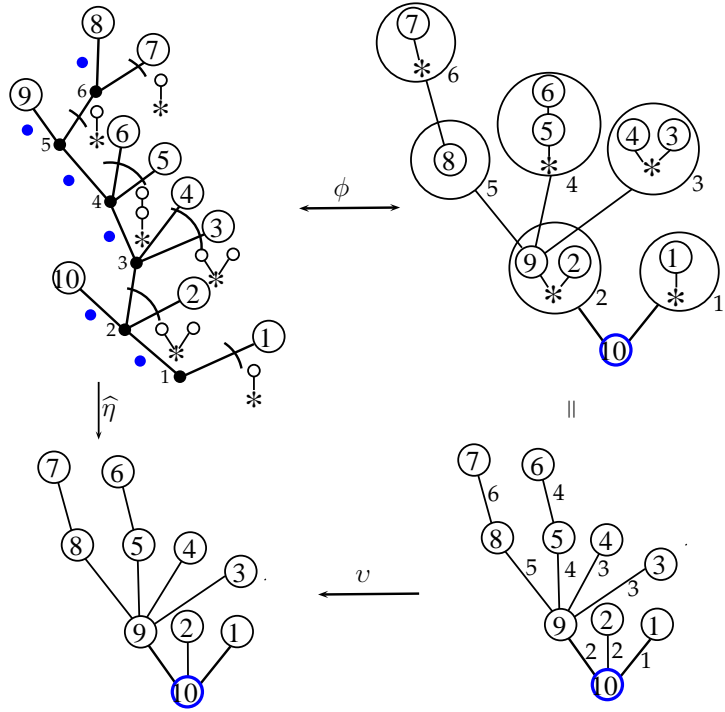

\noindent Where $\upsilon$ is the forgetful transformation which erases the colors of the trees in $\Arb_{L(\mathfrak{a}')}$. It can be easily checked by induction, verifying that contracting over the main spine without keeping track of the depth of internal vertices, is the same as applying the product $\eta$ iteratively all over it (see Fig. \ref{Schroder-colored}). Then we obtain the following antipode formula for the Hopf algebra $\CN_{\Arb}$.
\begin{equation}\label{antipodeA}
S(t_{\tau(T)})=\sum_{c}\prod_{i\geq 1}(-1)^{|a_c(i)|}t_{\tau(a_c(i))}.
\end{equation}
\noindent Here $c$ runs over all the admissible colorations of the tree $T$, $a_c(i)$ is the assembly of trees colored with color $i$.
\subsubsection{The general case \texorpdfstring{$\Arb_N$}{AN}.}
We now extend this result to the general case $\Arb_N$, which is a little more involved.\\
Since $\Arb_N=XN(\Arb_N)$, the bijection $\phi$ is as follows
 $$\CF_{(\Arb_N)_{2^+}}=\CF_{X.N_+(\Arb_N)}\stackrel{\phi}{\leftrightarrow}\Arb_{L(N_+(\Arb_N))}.$$
 The bijection $\phi$ acts as the product of the operad $\Arb$ but keeping the information of the depth of internal vertices in the original Scr\"oder tree, and without performing the products $\nu:N.N\rightarrow N$ of the monoid.  Hence, the elements of $\Arb_{L(N_+(\Arb_N))}$ are what we call {\em factored colored trees}. A factored colored tree $T_{N_f}$, is a structure with data:
 \begin{enumerate} 
 \item A rooted tree $T$,
 \item  An admissible coloration $c$ of $T$,
 \item For every color $i$ in the image of $c$, the connected structures in $a_c(i)$ are $N$-enriched trees.
 \end{enumerate}
 
 Denote by $\overline{\nu}$ the extension of the product $\nu:N.N\rightarrow N$ to $L(N_+)$ ,
$$\overline{\nu}:L(N_+)\rightarrow N,$$
\noindent defined recursively as $\overline{\nu}(n)=n$, for $n$ an element of $N\subset L(N_+)$, and for $k\geq 2$,
$$\overline{\nu}(n_1,n_2,\dots,n_k)=\nu(n_1,\overline{\nu}(n_2,n_3,\dots,n_k)).$$
We further extend $\overline{\nu}$ over $\Arb_{L(N_+(\Arb_N))}$ by applying $\overline{\nu}$ over the tuples of $M$-structures that enrich the fibers of the vertices of the factored colored tree $T_{N_f}$. This transformation is denoted by $\widehat{\nu}$. The following diagram commutes
\begin{center}
\begin{equation}\label{diagrama2}
\begin{tikzpicture}
  \node (C) {$\CF_{(\Arb_{N})_{2^+}}$};
  \node (P) [below of=C] {$\Arb_N$};
  \node (Ai) [right of=C] {$\Arb_{L(N_+(\Arb_N))}$};
  \draw[<->] (C) to node {$\phi$} (Ai);
  \draw[->] (C) to node [swap] {$\widehat{\eta}$} (P);
  \draw[<-] (P) to node [swap] {$\widehat{\nu}.$} (Ai);
\end{tikzpicture}
\end{equation} 
\end{center} 

Then, we obtain the antipode formula
\begin{theo} The antipode for the Hopf algebra $\CN_{\Arb_N}$  is given by the formula
\be\label{antipodaA_N} S(t_{\tau(T_N)})=\sum_{\widehat{\nu}(T_{N_f})=T_N}\prod_{i\geq 1}(-1)^{|a_c(i)|}t_{\tau(a_c(i))}.
\eeq
\end{theo}
The antipode formula for the Hopf algebra of planar trees $\CN_{\Arb_L}$, the case $N=L$, is as in Eq. (\ref{antipodeA}), with the only difference that the admissible colorations have to be also weakly increasing from right to left. For example we have

\begin{equation} S(t_{\arbolitoizq})=-t_{\ldos}^3+t_{\ltres}t_{\ldos}+t_{\ldos}t_{\cdos}-t_{\arbolitoizq},
\end{equation}

\noindent the terms corresponding to the $4$ admissible colorations:
\[
\scalebox{1} 
{
\begin{pspicture}(0,-0.3615)(3.51375,0.3615)
\psdots[dotsize=0.12](0.2653125,0.2785)
\psdots[dotsize=0.12](0.2653125,0.0385)
\psdots[dotsize=0.12](0.4253125,-0.2015)
\psdots[dotsize=0.12](0.5653125,0.0385)
\psline[linewidth=0.046cm](0.2653125,0.2385)(0.2653125,0.0585)
\psline[linewidth=0.046cm](0.2853125,0.0185)(0.4053125,-0.2015)
\psline[linewidth=0.046cm](0.4253125,-0.2215)(0.5653125,0.0385)
\usefont{T1}{ptm}{m}{n}
\rput(0.58953124,-0.2215){\small 1}
\usefont{T1}{ptm}{m}{n}
\rput(0.2196875,-0.2015){\small 2}
\usefont{T1}{ptm}{m}{n}
\rput(0.09296875,0.1485){3}
\psdots[dotsize=0.12](1.1986458,0.2415)
\psdots[dotsize=0.12](1.1986458,0.0014999998)
\psdots[dotsize=0.12](1.3586458,-0.2385)
\psdots[dotsize=0.12](1.4986458,0.0014999998)
\psline[linewidth=0.046cm](1.1986458,0.2015)(1.1986458,0.021499999)
\psline[linewidth=0.046cm](1.2186458,-0.0185)(1.3386458,-0.2385)
\psline[linewidth=0.046cm](1.3586458,-0.2585)(1.4986458,0.0014999998)
\usefont{T1}{ptm}{m}{n}
\rput(1.0372396,0.1115){2}
\usefont{T1}{ptm}{m}{n}
\rput(1.1530209,-0.2215){\small 2}
\usefont{T1}{ptm}{m}{n}
\rput(1.5228646,-0.2215){\small 1}
\psdots[dotsize=0.12](2.1319792,0.2785)
\psdots[dotsize=0.12](2.1319792,0.0385)
\psdots[dotsize=0.12](2.291979,-0.2015)
\psdots[dotsize=0.12](2.4319792,0.0385)
\psline[linewidth=0.046cm](2.1319792,0.2385)(2.1319792,0.0585)
\psline[linewidth=0.046cm](2.1519792,0.0185)(2.271979,-0.2015)
\psline[linewidth=0.046cm](2.291979,-0.2215)(2.4319792,0.0385)
\usefont{T1}{ptm}{m}{n}
\rput(2.456198,-0.2215){\small 1}
\usefont{T1}{ptm}{m}{n}
\rput(2.056198,-0.2015){\small 1}
\usefont{T1}{ptm}{m}{n}
\rput(1.970573,0.1485){2}
\psdots[dotsize=0.12](3.0653124,0.2785)
\psdots[dotsize=0.12](3.0653124,0.0385)
\psdots[dotsize=0.12](3.2253125,-0.2015)
\psdots[dotsize=0.12](3.3653126,0.0385)
\psline[linewidth=0.046cm](3.0653124,0.2385)(3.0653124,0.0585)
\psline[linewidth=0.046cm](3.0853126,0.0185)(3.2053125,-0.2015)
\psline[linewidth=0.046cm](3.2253125,-0.2215)(3.3653126,0.0385)
\usefont{T1}{ptm}{m}{n}
\rput(3.3895311,-0.2215){\small 1}
\usefont{T1}{ptm}{m}{n}
\rput(2.9895313,-0.2015){\small 1}
\usefont{T1}{ptm}{m}{n}
\rput(2.8721876,0.1485){1}
\end{pspicture}
}
\]

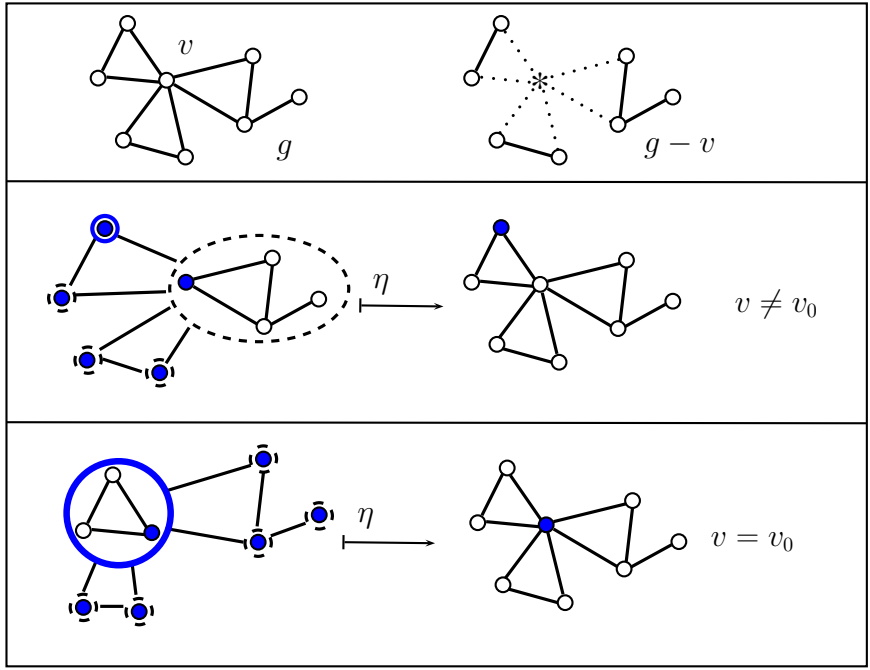
\begin{figure}
\begin{center}
\scalebox{0.7} 
{
\begin{pspicture}(0,-6.32)(16.38,6.32)
\pscircle[linewidth=0.04,dimen=outer,fillstyle=solid](5.58,4.52){0.16}
\pscircle[linewidth=0.04,dimen=outer,fillstyle=solid](2.24,3.7){0.16}
\pscircle[linewidth=0.04,dimen=outer,fillstyle=solid](3.42,3.38){0.16}
\pscircle[linewidth=0.04,dimen=outer,fillstyle=solid](4.7,5.3){0.16}
\pscircle[linewidth=0.04,dimen=outer,fillstyle=solid](1.76,4.88){0.16}
\pscircle[linewidth=0.04,dimen=outer,fillstyle=solid](2.32,5.92){0.16}
\psline[linewidth=0.06cm](2.38,5.78)(2.96,4.94)
\psline[linewidth=0.06cm](1.92,4.9)(2.96,4.8)
\psline[linewidth=0.06cm](2.24,5.82)(1.84,5.02)
\psline[linewidth=0.06cm](4.58,5.24)(3.2,4.9)
\psline[linewidth=0.06cm](3.2,4.76)(4.42,4.04)
\psline[linewidth=0.06cm](4.7,5.16)(4.58,4.14)
\psline[linewidth=0.06cm](4.64,4.02)(5.44,4.46)
\psline[linewidth=0.06cm](2.28,3.82)(3.0,4.78)
\psline[linewidth=0.06cm](3.1,4.72)(3.38,3.52)
\pscircle[linewidth=0.04,dimen=outer,fillstyle=solid](3.06,4.84){0.16}
\pscircle[linewidth=0.04,dimen=outer,fillstyle=solid](4.54,4.0){0.16}
\psline[linewidth=0.06cm](2.34,3.64)(3.28,3.38)
\pscircle[linewidth=0.04,dimen=outer,fillstyle=solid](12.675,4.52){0.16}
\pscircle[linewidth=0.04,dimen=outer,fillstyle=solid](9.335,3.7){0.16}
\pscircle[linewidth=0.04,dimen=outer,fillstyle=solid](10.515,3.38){0.16}
\pscircle[linewidth=0.04,dimen=outer,fillstyle=solid](11.795,5.3){0.16}
\pscircle[linewidth=0.04,dimen=outer,fillstyle=solid](8.855,4.88){0.16}
\pscircle[linewidth=0.04,dimen=outer,fillstyle=solid](9.415,5.92){0.16}
\psline[linewidth=0.06cm,linestyle=dotted,dotsep=0.16cm](9.475,5.78)(10.015,5.02)
\psline[linewidth=0.06cm,linestyle=dotted,dotsep=0.16cm](9.015,4.9)(9.975,4.8)
\psline[linewidth=0.06cm](9.335,5.82)(8.935,5.02)
\psline[linewidth=0.06cm,linestyle=dotted,dotsep=0.16cm](11.675,5.24)(10.355,4.9)
\psline[linewidth=0.06cm,linestyle=dotted,dotsep=0.16cm](10.375,4.7)(11.515,4.04)
\psline[linewidth=0.06cm](11.795,5.16)(11.675,4.14)
\psline[linewidth=0.06cm](11.735,4.02)(12.535,4.46)
\psline[linewidth=0.06cm,linestyle=dotted,dotsep=0.16cm](9.375,3.82)(10.035,4.68)
\psline[linewidth=0.06cm,linestyle=dotted,dotsep=0.16cm](10.215,4.66)(10.495,3.5)
\pscircle[linewidth=0.04,dimen=outer,fillstyle=solid](11.635,4.0){0.16}
\psdots[dotsize=0.32,dotstyle=asterisk](10.155,4.84)
\psline[linewidth=0.06cm](9.455,3.66)(10.375,3.4)
\pscircle[linewidth=0.04,dimen=outer,fillstyle=solid](12.675,0.64){0.16}
\pscircle[linewidth=0.04,dimen=outer,fillstyle=solid](9.335,-0.18){0.16}
\pscircle[linewidth=0.04,dimen=outer,fillstyle=solid](10.515,-0.52){0.16}
\pscircle[linewidth=0.04,dimen=outer,fillstyle=solid](11.795,1.42){0.16}
\pscircle[linewidth=0.04,dimen=outer,fillstyle=solid](8.855,1.0){0.16}
\pscircle[linewidth=0.04,dimen=outer,fillstyle=solid,fillcolor=blue](9.415,2.04){0.16}
\psline[linewidth=0.06cm](9.475,1.9)(10.055,1.06)
\psline[linewidth=0.06cm](9.015,1.02)(10.055,0.92)
\psline[linewidth=0.06cm](9.335,1.94)(8.935,1.14)
\psline[linewidth=0.06cm](11.675,1.36)(10.295,1.02)
\psline[linewidth=0.06cm](10.295,0.88)(11.515,0.16)
\psline[linewidth=0.06cm](11.795,1.28)(11.675,0.26)
\psline[linewidth=0.06cm](11.735,0.14)(12.535,0.58)
\psline[linewidth=0.06cm](9.375,-0.06)(10.095,0.9)
\psline[linewidth=0.06cm](10.195,0.8)(10.475,-0.36)
\pscircle[linewidth=0.04,dimen=outer,fillstyle=solid](10.155,0.96){0.16}
\pscircle[linewidth=0.04,dimen=outer,fillstyle=solid](11.635,0.12){0.16}
\psline[linewidth=0.06cm](9.435,-0.26)(10.375,-0.52)
\pscircle[linewidth=0.04,dimen=outer,fillstyle=solid](5.95,0.68){0.16}
\pscircle[linewidth=0.04,dimen=outer,fillstyle=solid](5.07,1.46){0.16}
\psline[linewidth=0.06cm](4.95,1.4)(3.57,1.06)
\psline[linewidth=0.06cm](3.57,0.92)(4.79,0.2)
\psline[linewidth=0.06cm](5.07,1.32)(4.95,0.3)
\pscircle[linewidth=0.04,dimen=outer,fillstyle=solid,fillcolor=blue](3.43,1.0){0.16}
\psellipse[linewidth=0.06,linestyle=dashed,dash=0.16cm 0.16cm,dimen=outer](4.82,0.88)(1.73,1.04)
\psline[linewidth=0.06cm](5.01,0.18)(5.81,0.62)
\pscircle[linewidth=0.04,dimen=outer,fillstyle=solid](4.91,0.16){0.16}
\pscircle[linewidth=0.04,dimen=outer,fillstyle=solid,fillcolor=blue](1.55,-0.48){0.16}
\pscircle[linewidth=0.04,dimen=outer,fillstyle=solid,fillcolor=blue](2.93,-0.72){0.16}
\pscircle[linewidth=0.06,linestyle=dashed,dash=0.16cm 0.16cm,dimen=outer](2.93,-0.72){0.28}
\pscircle[linewidth=0.04,dimen=outer,fillstyle=solid,fillcolor=blue](1.07,0.7){0.16}
\pscircle[linewidth=0.04,dimen=outer,fillstyle=solid,fillcolor=blue](1.89,2.02){0.16}
\psline[linewidth=0.06cm](2.13,1.88)(3.25,1.4)
\psline[linewidth=0.06cm](1.33,0.76)(3.05,0.82)
\psline[linewidth=0.06cm](1.71,1.84)(1.23,0.94)
\psline[linewidth=0.06cm](1.79,-0.28)(3.13,0.52)
\psline[linewidth=0.06cm](3.49,0.14)(3.05,-0.54)
\psline[linewidth=0.06cm](1.85,-0.46)(2.59,-0.7)
\pscircle[linewidth=0.08,linecolor=blue,dimen=outer](1.89,2.02){0.28}
\pscircle[linewidth=0.06,linestyle=dashed,dash=0.16cm 0.16cm,dimen=outer](1.57,-0.48){0.28}
\pscircle[linewidth=0.06,linestyle=dashed,dash=0.16cm 0.16cm,dimen=outer](1.07,0.7){0.28}
\pscircle[linewidth=0.04,dimen=outer,fillstyle=solid,fillcolor=blue](5.96,-3.42){0.16}
\pscircle[linewidth=0.04,dimen=outer,fillstyle=solid,fillcolor=blue](1.48,-5.18){0.16}
\pscircle[linewidth=0.04,dimen=outer,fillstyle=solid,fillcolor=blue](2.54,-5.26){0.16}
\pscircle[linewidth=0.04,dimen=outer,fillstyle=solid,fillcolor=blue](4.9,-2.36){0.16}
\pscircle[linewidth=0.04,dimen=outer,fillstyle=solid](1.48,-3.72){0.16}
\pscircle[linewidth=0.04,dimen=outer,fillstyle=solid](2.04,-2.66){0.16}
\psline[linewidth=0.06cm](2.12,-2.78)(2.68,-3.66)
\psline[linewidth=0.06cm](1.64,-3.7)(2.68,-3.8)
\psline[linewidth=0.06cm](1.96,-2.78)(1.56,-3.58)
\psline[linewidth=0.06cm](4.66,-2.48)(3.0,-2.98)
\psline[linewidth=0.06cm](3.04,-3.68)(4.52,-3.94)
\psline[linewidth=0.06cm](4.9,-2.62)(4.78,-3.64)
\psline[linewidth=0.06cm](5.04,-3.82)(5.7,-3.58)
\psline[linewidth=0.06cm](1.48,-4.88)(1.74,-4.26)
\psline[linewidth=0.06cm](2.4,-4.3)(2.48,-4.94)
\pscircle[linewidth=0.04,dimen=outer,fillstyle=solid,fillcolor=blue](2.78,-3.76){0.16}
\pscircle[linewidth=0.04,dimen=outer,fillstyle=solid,fillcolor=blue](4.8,-3.94){0.16}
\psline[linewidth=0.06cm](1.8,-5.16)(2.24,-5.16)
\pscircle[linewidth=0.06,linestyle=dashed,dash=0.16cm 0.16cm,dimen=outer](4.9,-2.36){0.28}
\pscircle[linewidth=0.124,linecolor=blue,dimen=outer](2.15,-3.39){1.05}
\pscircle[linewidth=0.06,linestyle=dashed,dash=0.16cm 0.16cm,dimen=outer](4.8,-3.96){0.28}
\pscircle[linewidth=0.06,linestyle=dashed,dash=0.16cm 0.16cm,dimen=outer](5.96,-3.44){0.28}
\pscircle[linewidth=0.06,linestyle=dashed,dash=0.16cm 0.16cm,dimen=outer](2.54,-5.26){0.28}
\pscircle[linewidth=0.06,linestyle=dashed,dash=0.16cm 0.16cm,dimen=outer](1.48,-5.18){0.28}
\pscircle[linewidth=0.04,dimen=outer,fillstyle=solid](12.79,-3.93){0.16}
\pscircle[linewidth=0.04,dimen=outer,fillstyle=solid](9.45,-4.75){0.16}
\pscircle[linewidth=0.04,dimen=outer,fillstyle=solid](10.63,-5.09){0.16}
\pscircle[linewidth=0.04,dimen=outer,fillstyle=solid](11.91,-3.15){0.16}
\pscircle[linewidth=0.04,dimen=outer,fillstyle=solid](8.97,-3.57){0.16}
\pscircle[linewidth=0.04,dimen=outer,fillstyle=solid](9.53,-2.53){0.16}
\psline[linewidth=0.06cm](9.59,-2.67)(10.17,-3.51)
\psline[linewidth=0.06cm](9.13,-3.55)(10.17,-3.65)
\psline[linewidth=0.06cm](9.45,-2.63)(9.05,-3.43)
\psline[linewidth=0.06cm](11.79,-3.21)(10.41,-3.55)
\psline[linewidth=0.06cm](10.41,-3.69)(11.63,-4.41)
\psline[linewidth=0.06cm](11.91,-3.29)(11.79,-4.31)
\psline[linewidth=0.06cm](11.85,-4.43)(12.65,-3.99)
\psline[linewidth=0.06cm](9.49,-4.63)(10.21,-3.67)
\psline[linewidth=0.06cm](10.29,-3.71)(10.59,-4.93)
\pscircle[linewidth=0.04,dimen=outer,fillstyle=solid,fillcolor=blue](10.27,-3.61){0.16}
\pscircle[linewidth=0.04,dimen=outer,fillstyle=solid](11.75,-4.45){0.16}
\psline[linewidth=0.06cm](9.55,-4.83)(10.49,-5.09)
\psline[linewidth=0.04cm,tbarsize=0.07055555cm 5.0,arrowsize=0.05291667cm 2.0,arrowlength=1.4,arrowinset=0.4]{|->}(6.4,-3.94)(8.16,-3.96)
\psline[linewidth=0.04cm,tbarsize=0.07055555cm 5.0,arrowsize=0.05291667cm 2.0,arrowlength=1.4,arrowinset=0.4]{|->}(6.72,0.56)(8.32,0.54)
\fontsize{16pt}{2pt}
\usefont{T1}{ptm}{m}{n}
\rput(7.137344,0.955){$\eta$}
\usefont{T1}{ptm}{m}{n}
\rput(6.8414063,-3.49){$\eta$}
\usefont{T1}{ptm}{m}{n}
\rput(14.641406,0.59){$v\neq v_0$}
\usefont{T1}{ptm}{m}{n}
\rput(14.191406,-3.93){$v= v_0$}
\psframe[linewidth=0.04,dimen=outer](16.38,6.32)(0.0,-6.32)
\psline[linewidth=0.04cm](0.02,2.92)(16.36,2.9)
\psline[linewidth=0.04cm](0.02,-1.66)(16.36,-1.68)
\usefont{T1}{ptm}{m}{n}
\rput(3.4314063,5.47){$v$}
\usefont{T1}{ptm}{m}{n}
\rput(12.811406,3.59){$g-v$}
\usefont{T1}{ptm}{m}{n}
\rput(5.3014064,3.51){$g$}
\end{pspicture}
}
\caption{A graph with a cutpoint $v$ and factorizations according with the selection of the distinguished vertex, $v\neq v_0$, and $v=v_0$.}\label{noprimos} 
\end{center}
\end{figure}
\subsection {The Hopf algebra \texorpdfstring{$\CN_{\Grp}$}{NGrp}.}
Recall that for a connected graph $g$, a vertex $v$ is called a {\em cutpoint} if the graph $g-v$, obtained by deleting $v$, is disconnected. A {\em nonseparable} (or {\em biconnected}) graph is a connected graph that does not have cutpoints. 

The prime  structures of the operad $\Grp$ are the pointed graphs $(g,v)$  such that $g$ is nonseparable. We prove that in the following proposition.
\begin{prop} A pointed graph $(g,v_0)\in \Grp[U]$ is a prime structure of $\Grp$ if and only if $g$ is nonseparable. 
\end{prop}
\begin{proof}
Assuming that $(g,v_0)$ is not prime, let $(a,(g'_{\pi},B_0))\in\Grp[U]$ be a non trivial factorization of $(g,v_0)$. Then, by the definition of $\eta$,  every distinguished element $v$ of a non singleton graph $(g_B,v)$ in $a$ is a cutpoint of $g$, because if deleted, it disconnects from $g$ the rest of the vertices of $g_B$.
Conversely, we are going to prove that if $g$ contains a cutpoint, then there exists a non trivial factorization of $(g,v_0)$. Let $v$ a cutpoint of $g$. Then $g-v$ splits in some $k\geq 2$ conected components $g_1,g_2,\dots,g_k$. If $v\neq v_0$, chose one of them such that $v_0$ is not there. We choose $g_1$ assuming without loss of generality that  $v_0$ is not in $g_1$. Consider the graph $g_1+v$ generated in $g$ by $g_1$ and $v$. Denote by $U_1$ the set of vertices of $g_1+v$. Let $\pi$ be the partition formed by singletons in $U_2=U-U_1$ and only one `big block' $U_1$, $\pi=\{U_1\}\cup \{\{u\}\mid u\in U_2\}$. Let $a$ be the assembly with subjacent partition $\pi$ and whose structures are the singleton graphs on the elements of $U_2$ and the pointed graph $(g_1+v,v)$ on $U_1$. Note that $v_0$ is in $U_2$. Define $g'_{\pi}$ to be the graph obtained by contracting to a point the subgraph $g_1+v$ in $g$. Let $B_0=\{v_0\}$ be the block of singletons containing $v_0$. $\eta(a,(g'_\pi,B_0))$ is clearly equal to $(g,v_0)$. When $v=v_0$ the construction is similar, except that $B_0$ is chosen to be $U_1$ (see Figure (\ref{noprimos})).    
\end{proof}

Recall that a {\em biconnected component} of a graph is a maximal biconnected subgraph. Two different biconnected components have at most one vertex in common (a cutpoint), and they partition the edges of the graph. 
We denote by $\Bc$ the species of biconnected graphs.
The species of pointed graphs satisfies the identity 

\begin{equation}
\Grp=XE(\Bc'(\Grp)).
\end{equation}
The reader can check details in \cite [sect. 4.2.]{B-L-L}.
Hence, $\Grp$ is isomorphic to the species $\Arb_{E(\mathscr{B}'_c)}$ of trees enriched with assemblies of $\mathscr{B}'_c$-structures. Since $E(\mathscr{B}'_c)$ is a monoid, $\Arb_{E(\mathscr{B}'_c)}$ is an operad. 
The reader may also check that  $\Grp$ and $\Arb_{E(\mathscr{B}'_c)}$ are also isomorphic as operads. Then, we can deduce the antipode formula for $\CN_{\Grp}$ from the general one for $\CN_{\Arb_N}$(\ref{antipodaA_N}).

\begin{defi}Let $(g,v_0)$ be a pointed graph and let $c:\mathscr{E}(g)\rightarrow \PP$ be a coloration of the edges of $g$ with colors in the positive integers $\PP$ such that edges in the same biconnected component have the same color. Denote by $a(i)=\{g_{B}^{(i)}\}_B$ the assembly of subgraphs of $g$ induced by the edges colored $i$.
 We say that  $c$ is an {\em admissible coloring} relative to $(g,v_0)$ if it satisfies the following conditions
\begin{enumerate}
\item There is at least one edge colored $1$.
\item $c$ is weakly increasing in any sequence $B_0, B_1,\dots,B_n$ of biconnected components such that $B_{i-1}$ and $B_i$ share a cutpoint $i=1,\dots,n$, and  $v_0\in B_0$.
\item For $i\ge 2$ in the image of $c$, in every connected component $g_B$ of $a(i)$ there is always one vertex $v_B$ which is also in a biconnected component of color $i-1$. 
\end{enumerate}
\end{defi}

The vertex $v_B$ in the third condition is necessarily a cutpoint, since it is in two different biconnected components. Observe that as a consequence of the definition of admissibility, $a(1)$ is connected graph that contains $v_0$.
Denote by $a^{\bullet}(i)$ the assembly of pointed graphs obtained by choosing as distinguished element the cutpoint $v_B$ on each graph $g_B^{(i)}$,   $a^{\bullet}(i)=\{(g_{B}^{(i)},v_{B})\}_B$, $i\geq 2$. For the graph $a(1)=g_1$, choose $v_0$ to be the distinguish element, $a(1)=(g_1,v_0)$.
By condition $3$, the image of an admissible coloration $c$ is a segment $[k]=\{1,2,\dots,k\}$ of $\PP$.

 Now, we give our combinatorial recipe for the antipode of $\CN_{\Grp}$.

\begin{theo}Let $(g,v_0)$ be a pointed graph. The antipode $S_{\Grp}(g,v_o)$ is given by:
$$ S_{\Grp}(g,v_0)=\sum_{c\in\admc(g,v_0)}\prod_{i=1}^k(-1)^{|a^{\bullet}(i)|}t_{\tau(a^{\bullet}(i))}.$$
\noindent where $\admc(g,v_0)$ is the set of admissible colorations on $(g,v_0)$.
\end{theo} 

\begin{figure}
\begin{center}
\scalebox{0.7} 
{
\begin{pspicture}(0,-6.41)(15.26,6.41)
\pscircle[linewidth=0.04,dimen=outer](1.1726563,5.101719){0.33}
\usefont{T1}{ptm}{m}{n}
\rput(1.1473438,5.087656){\Large $1$}
\pscircle[linewidth=0.04,dimen=outer](2.8526564,5.641719){0.33}
\usefont{T1}{ptm}{m}{n}
\rput(2.8273437,5.6276565){\Large $3$}
\pscircle[linewidth=0.04,dimen=outer](1.9526563,5.641719){0.33}
\usefont{T1}{ptm}{m}{n}
\rput(1.9273437,5.6276565){\Large $2$}
\pscircle[linewidth=0.04,dimen=outer](3.5726562,4.681719){0.33}
\usefont{T1}{ptm}{m}{n}
\rput(3.5473437,4.6676564){\Large $4$}
\pscircle[linewidth=0.046,dimen=outer](8.372656,3.6817188){0.33}
\usefont{T1}{ptm}{m}{n}
\rput(8.347343,3.6676562){\Large $9$}
\pscircle[linewidth=0.04,dimen=outer](7.3726563,5.081719){0.33}
\usefont{T1}{ptm}{m}{n}
\rput(7.347344,5.067656){\Large $7$}
\pscircle[linewidth=0.04,dimen=outer](4.152656,4.081719){0.33}
\usefont{T1}{ptm}{m}{n}
\rput(4.1273437,4.067656){\Large $5$}
\pscircle[linewidth=0.04,dimen=outer](8.132656,4.521719){0.33}
\usefont{T1}{ptm}{m}{n}
\rput(8.107344,4.507656){\Large $8$}
\pscircle[linewidth=0.04,dimen=outer](5.8926563,4.201719){0.33}
\usefont{T1}{ptm}{m}{n}
\rput(5.867344,4.1876564){\Large $6$}
\psline[linewidth=0.04cm](1.34,4.871719)(2.2,4.271719)
\psline[linewidth=0.04cm](2.0,5.351719)(2.22,4.271719)
\psline[linewidth=0.04cm](2.78,5.351719)(2.26,4.271719)
\psdots[dotsize=0.16](2.24,4.2117186)
\psdots[dotsize=0.16](2.8,3.7117188)
\psline[linewidth=0.04cm](3.38,4.431719)(2.82,3.7317188)
\psline[linewidth=0.04cm](3.94,3.8517187)(2.84,3.7117188)
\psline[linewidth=0.04cm](2.24,4.2117186)(2.8,3.7117188)
\psdots[dotsize=0.16](4.86,2.2317188)
\psdots[dotsize=0.16](5.98,2.9317188)
\psdots[dotsize=0.16](6.98,3.6917188)
\psline[linewidth=0.04cm](4.86,2.2717187)(5.96,2.9117188)
\psline[linewidth=0.04cm](5.96,2.9317188)(6.94,3.6717188)
\psline[linewidth=0.04cm](5.98,2.9517188)(5.9,3.9117188)
\psline[linewidth=0.04cm](6.98,3.7117188)(7.26,4.791719)
\psline[linewidth=0.04cm](6.98,3.7117188)(7.88,4.351719)
\psline[linewidth=0.04cm](7.0,3.6917188)(8.08,3.6717188)
\psline[linewidth=0.04cm](2.78,3.7117188)(4.82,2.2517188)
\usefont{T1}{ppl}{m}{n}
\rput(5.158125,2.0617187){\large 1}
\usefont{T1}{ppl}{m}{n}
\rput(2.5190625,3.4817188){\large 2}
\usefont{T1}{ppl}{m}{n}
\rput(6.2790623,2.8017187){\large 2}
\usefont{T1}{ppl}{m}{n}
\rput(7.117656,3.3217187){\large 3}
\usefont{T1}{ppl}{m}{n}
\rput(1.9576563,3.9817188){\large 3}
\psdots[dotsize=0.2,linecolor=blue](1.64,4.351719)
\psdots[dotsize=0.2,linecolor=blue](2.42,3.7917187)
\psdots[dotsize=0.2,linecolor=blue](7.72,3.4917188)
\psdots[dotsize=0.2,linecolor=blue](6.56,3.1717188)
\psdots[dotsize=0.2,linecolor=blue](5.56,2.4517188)
\psarc[linewidth=0.04](2.83,2.7617188){1.29}{47.663002}{141.70984}
\psarc[linewidth=0.04](2.31,3.5417187){1.29}{70.51387}{150.52411}
\psarc[linewidth=0.04](6.6,3.4117188){1.3}{-12.60016}{71.56505}
\psarc[linewidth=0.04](4.95,1.6217188){1.29}{47.663002}{141.48308}
\psarc[linewidth=0.04](5.72,2.3717186){1.3}{23.962488}{92.7927}
\pscircle[linewidth=0.04,dimen=outer,fillstyle=solid,fillcolor=blue](0.64,3.8317187){0.14}
\pscircle[linewidth=0.04,dimen=outer](0.94,4.291719){0.14}
\pscircle[linewidth=0.04,dimen=outer](1.26,3.8317187){0.14}
\psline[linewidth=0.04cm](0.88,4.1917186)(0.68,3.9317188)
\psline[linewidth=0.04cm](1.02,4.2117186)(1.24,3.9317188)
\psline[linewidth=0.04cm](0.74,3.8317187)(1.14,3.8317187)
\pscircle[linewidth=0.04,dimen=outer,fillstyle=solid](7.6,2.5117188){0.14}
\pscircle[linewidth=0.04,dimen=outer](7.9,2.9717188){0.14}
\pscircle[linewidth=0.04,dimen=outer,fillstyle=solid,fillcolor=blue](8.22,2.5117188){0.14}
\psline[linewidth=0.04cm](7.84,2.8717186)(7.64,2.6117187)
\psline[linewidth=0.04cm](7.98,2.8917189)(8.2,2.6117187)
\psline[linewidth=0.04cm](7.7,2.5117188)(8.1,2.5117188)
\pscircle[linewidth=0.04,dimen=outer,fillstyle=solid,fillcolor=blue](1.06,3.2617188){0.14}
\pscircle[linewidth=0.04,dimen=outer](1.94,3.2617188){0.14}
\pscircle[linewidth=0.04,dimen=outer](1.5,3.2617188){0.14}
\psline[linewidth=0.04cm](1.62,3.2717187)(1.82,3.2717187)
\psline[linewidth=0.04cm](1.16,3.2517188)(1.38,3.2517188)
\pscircle[linewidth=0.04,dimen=outer,fillstyle=solid,fillcolor=blue](4.12,2.2217188){0.14}
\pscircle[linewidth=0.04,dimen=outer](3.66,2.2417188){0.14}
\psline[linewidth=0.04cm](3.76,2.2117188)(3.98,2.2117188)
\pscircle[linewidth=0.04,dimen=outer,fillstyle=solid](6.76,2.7217188){0.14}
\pscircle[linewidth=0.04,dimen=outer,fillstyle=solid,fillcolor=blue](7.2,2.7217188){0.14}
\psline[linewidth=0.04cm](6.86,2.7117188)(7.08,2.7117188)
\rput{8.787891}(0.71680516,-1.6753714){\pscircle[linewidth=0.04,dimen=outer](11.26016,3.8266149){0.33}}
\usefont{T1}{ptm}{m}{n}
\rput{8.787891}(0.71680516,-1.6753715){\rput(11.234848,3.8125522){\Large $1$}}
\rput{8.787891}(0.983463,-1.7388158){\pscircle[linewidth=0.04,dimen=outer](11.806326,5.5300536){0.33}}
\usefont{T1}{ptm}{m}{n}
\rput{8.787891}(0.983463,-1.7388158){\rput(11.781013,5.5159907){\Large $3$}}
\rput{8.787891}(0.9391137,-1.5225844){\pscircle[linewidth=0.04,dimen=outer](10.377119,5.349585){0.33}}
\usefont{T1}{ptm}{m}{n}
\rput{8.787891}(0.9391137,-1.5225844){\rput(10.351807,5.3355227){\Large $2$}}
\rput{8.787891}(0.8915544,-1.8705038){\pscircle[linewidth=0.04,dimen=outer](12.617274,4.866154){0.33}}
\usefont{T1}{ptm}{m}{n}
\rput{8.787891}(0.8915544,-1.8705038){\rput(12.591962,4.8520913){\Large $4$}}
\rput{8.787891}(1.0330262,-2.0041854){\pscircle[linewidth=0.04,dimen=outer](13.557886,5.71988){0.33}}
\usefont{T1}{ptm}{m}{n}
\rput{8.787891}(1.0330262,-2.0041854){\rput(13.532574,5.7058177){\Large $5$}}
\psline[linewidth=0.04cm](11.048281,4.067067)(10.507743,5.0763335)
\psline[linewidth=0.04cm](11.341703,4.1326656)(11.762117,5.209537)
\psline[linewidth=0.04cm](11.553202,3.942748)(12.416377,4.64284)
\psline[linewidth=0.04cm](12.881384,5.0385275)(13.379809,5.4798565)
\psline[linewidth=0.04cm](10.696036,5.4292436)(11.492756,5.5119348)
\usefont{T1}{ppl}{m}{n}
\rput{8.787891}(0.7762346,-1.8083756){\rput(12.148154,4.147137){\large 2}}
\usefont{T1}{ppl}{m}{n}
\rput{8.787891}(0.9269799,-1.9682126){\rput(13.263598,5.048128){\large 2}}
\usefont{T1}{ppl}{m}{n}
\rput{8.787891}(1.0128564,-1.623786){\rput(11.064703,5.7808647){\large 3}}
\usefont{T1}{ppl}{m}{n}
\rput{8.787891}(0.81762135,-1.5562034){\rput(10.527321,4.544248){\large 3}}
\usefont{T1}{ppl}{m}{n}
\rput{8.787891}(0.8589081,-1.6787276){\rput(11.345238,4.7516418){\large 3}}
\pscircle[linewidth=0.102,linecolor=blue,dimen=outer](12.432656,2.2828126){0.33}
\usefont{T1}{ptm}{m}{n}
\rput(12.407344,2.26875){\Large $9$}
\rput{-1.1497928}(-0.060150906,0.2804228){\pscircle[linewidth=0.04,dimen=outer](13.943315,3.1375144){0.33}}
\usefont{T1}{ptm}{m}{n}
\rput{-1.1497928}(-0.060150906,0.2804228){\rput(13.918002,3.123452){\Large $7$}}
\psline[linewidth=0.04cm](12.694653,2.3524125)(13.706893,2.9522212)
\psline[linewidth=0.04cm](12.7,2.15)(13.971471,2.083273)
\usefont{T1}{ppl}{m}{n}
\rput{-1.1497928}(-0.035556346,0.2690982){\rput(13.383497,1.9083432){\large 3}}
\usefont{T1}{ppl}{m}{n}
\rput{-1.1497928}(-0.055358376,0.2648761){\rput(13.163211,2.8929617){\large 3}}
\pscircle[linewidth=0.04,dimen=outer](12.792656,3.5028124){0.33}
\usefont{T1}{ptm}{m}{n}
\rput(12.7673435,3.48875){\Large $6$}
\psline[linewidth=0.04cm](12.56,2.59)(12.74,3.21)
\psline[linewidth=0.04cm](12.2,2.51)(11.36,3.55)
\usefont{T1}{ppl}{m}{n}
\rput(12.439062,3.0028124){\large 2}
\psline[linewidth=0.074cm,arrowsize=0.05291667cm 2.0,arrowlength=1.4,arrowinset=0.4]{<->}(8.92,3.67)(10.54,3.67)
\rput{8.787891}(-0.3490522,-1.0841572){\pscircle[linewidth=0.04,dimen=outer](6.88016,-2.8133852){0.33}}
\usefont{T1}{ptm}{m}{n}
\rput{8.787891}(-0.3490522,-1.0841572){\rput(6.8548474,-2.8274477){\Large $1$}}
\rput{8.787891}(-0.08919983,-1.0190341){\pscircle[linewidth=0.04,dimen=outer](6.5863256,-1.0899465){0.33}}
\usefont{T1}{ptm}{m}{n}
\rput{8.787891}(-0.089199826,-1.0190341){\rput(6.5610127,-1.104009){\Large $3$}}
\rput{8.787891}(-0.16011319,-0.8570949){\pscircle[linewidth=0.04,dimen=outer](5.497119,-1.470415){0.33}}
\usefont{T1}{ptm}{m}{n}
\rput{8.787891}(-0.16011319,-0.8570949){\rput(5.4718065,-1.4844775){\Large $2$}}
\rput{8.787891}(-0.087802045,-1.2050241){\pscircle[linewidth=0.04,dimen=outer](7.797274,-1.173846){0.33}}
\usefont{T1}{ptm}{m}{n}
\rput{8.787891}(-0.087802045,-1.2050241){\rput(7.7719617,-1.1879085){\Large $4$}}
\rput{8.787891}(-0.13976502,-1.3013184){\pscircle[linewidth=0.04,dimen=outer](8.397886,-1.5601199){0.33}}
\usefont{T1}{ptm}{m}{n}
\rput{8.787891}(-0.13976502,-1.3013184){\rput(8.372574,-1.5741824){\Large $5$}}
\psline[linewidth=0.04cm](6.668281,-2.5729327)(5.72,-1.71)
\psline[linewidth=0.04cm](6.8817034,-2.4873345)(6.7,-1.39)
\psline[linewidth=0.04cm](7.0932016,-2.597252)(7.7,-1.43)
\psline[linewidth=0.04cm](7.1813836,-2.7014725)(8.16,-1.73)
\usefont{T1}{ppl}{m}{n}
\rput{-0.20341204}(0.0070531643,0.027023572){\rput(7.6081543,-1.972863){\large 2}}
\usefont{T1}{ppl}{m}{n}
\rput{-0.35776374}(0.013429367,0.0406267){\rput(6.505238,-2.1283584){\large 3}}
\pscircle[linewidth=0.102,linecolor=blue,dimen=outer](8.052656,-4.3571873){0.33}
\usefont{T1}{ptm}{m}{n}
\rput(8.027344,-4.37125){\Large $9$}
\rput{-1.1497928}(0.07220757,0.19119534){\pscircle[linewidth=0.04,dimen=outer](9.563314,-3.5024855){0.33}}
\usefont{T1}{ptm}{m}{n}
\rput{-1.1497928}(0.07220757,0.19119534){\rput(9.538002,-3.5165482){\Large $7$}}
\rput{-1.1497928}(0.09207,0.19742326){\pscircle[linewidth=0.04,dimen=outer](9.883581,-4.489111){0.33}}
\usefont{T1}{ptm}{m}{n}
\rput{-1.1497928}(0.09207,0.19742326){\rput(9.858269,-4.5031734){\Large $8$}}
\psline[linewidth=0.04cm](8.314652,-4.2875876)(9.326893,-3.6877787)
\psline[linewidth=0.04cm](8.32,-4.49)(9.591471,-4.556727)
\usefont{T1}{ppl}{m}{n}
\rput{-1.1497928}(0.08678105,0.1811754){\rput(9.063498,-4.2316566){\large 3}}
\pscircle[linewidth=0.04,dimen=outer](8.172656,-3.0171876){0.33}
\usefont{T1}{ptm}{m}{n}
\rput(8.147344,-3.03125){\Large $6$}
\psline[linewidth=0.04cm](8.12,-4.03)(8.16,-3.29)
\psline[linewidth=0.04cm](7.82,-4.13)(6.98,-3.09)
\usefont{T1}{ppl}{m}{n}
\rput(7.9990625,-3.6971874){\large 2}
\psarc[linewidth=0.04](8.1,-4.4882812){1.3}{341.97757}{43.640797}
\psarc[linewidth=0.04](8.03,-4.7382812){1.29}{73.03264}{94.55013}
\pscircle[linewidth=0.04,dimen=outer,fillstyle=solid](9.579559,-5.548281){0.14}
\pscircle[linewidth=0.04,dimen=outer](9.279559,-5.088281){0.14}
\psline[linewidth=0.04cm](9.33956,-5.188281)(9.539559,-5.4482813)
\psline[linewidth=0.04cm](9.199559,-5.168281)(8.979559,-5.4482813)
\psline[linewidth=0.04cm](9.479559,-5.548281)(9.079559,-5.548281)
\psdots[dotsize=0.36,dotstyle=asterisk](8.8795595,-5.57)
\rput{-90.0}(11.878282,5.241719){\pscircle[linewidth=0.04,dimen=outer,fillstyle=solid](8.56,-3.3182812){0.14}}
\psline[linewidth=0.04cm](8.55,-3.4182813)(8.55,-3.6382813)
\psdots[dotsize=0.36,dotstyle=asterisk](8.54,-3.87)
\rput{-45.727913}(4.9257364,3.6650093){\pscircle[linewidth=0.04,dimen=outer,fillstyle=solid](6.808639,-4.0081687){0.14}}
\psline[linewidth=0.04cm](6.8712854,-4.086753)(7.02486,-4.24428)
\rput(0.0, 0.0){\psdots[dotsize=0.36,dotangle=44.272087,dotstyle=asterisk](7.179455,-4.417179)}
\psarc[linewidth=0.04](8.05,-4.418281){1.29}{113.7458}{146.16125}
\usefont{T1}{ppl}{m}{n}
\rput(7.365156,-3.905){\Large 1}
\psarc[linewidth=0.04](7.16,-2.8882813){1.3}{41.390053}{78.27685}
\psarc[linewidth=0.04](6.99,-2.9982812){1.29}{90.88374}{144.92729}
\rput{28.04805}(0.049261354,-4.2946925){\pscircle[linewidth=0.04,dimen=outer](8.621806,-2.048734){0.14}}
\rput{28.04805}(-0.09362951,-4.136397){\pscircle[linewidth=0.04,dimen=outer](8.233483,-2.2556274){0.14}}
\psline[linewidth=0.04cm](8.334687,-2.1903763)(8.511198,-2.096334)
\psline[linewidth=0.04cm](7.938117,-2.4243248)(8.132278,-2.3208783)
\rput(0.0, 0.0){\psdots[dotsize=0.4,dotangle=28.04805,dotstyle=asterisk](7.717708,-2.5210414)}
\rput{105.88233}(5.0850697,-8.134398){\pscircle[linewidth=0.04,dimen=outer,fillstyle=solid](5.613942,-2.1471655){0.14}}
\rput{105.88233}(4.2275047,-8.31566){\pscircle[linewidth=0.04,dimen=outer](5.253601,-2.561598){0.14}}
\psline[linewidth=0.04cm](5.333364,-2.4765222)(5.528706,-2.2130048)
\psline[linewidth=0.04cm](5.3524404,-2.616651)(5.6819572,-2.7516272)
\psline[linewidth=0.04cm](5.641309,-2.2433481)(5.7507734,-2.6280785)
\rput(0.0, 0.0){\psdots[dotsize=0.36,dotangle=-254.11766,dotstyle=asterisk](5.826396,-2.8144999)}
\psline[linewidth=0.074cm,arrowsize=0.05291667cm 2.0,arrowlength=1.4,arrowinset=0.4]{<->}(10.78,0.71)(9.82,-0.65)
\psline[linewidth=0.074cm,arrowsize=0.05291667cm 2.0,arrowlength=1.4,arrowinset=0.4]{<->}(5.12,0.79)(5.66,-0.49)
\usefont{T1}{ppl}{m}{n}
\rput(11.585156,2.775){\Large 1}
\rput{-1.1497928}(-0.04028848,0.28665072){\pscircle[linewidth=0.04,dimen=outer](14.263581,2.1508892){0.33}}
\usefont{T1}{ptm}{m}{n}
\rput{-1.1497928}(-0.04028848,0.28665072){\rput(14.238269,2.1368268){\Large $8$}}
\psline[linewidth=0.04cm](14.19716,2.4622848)(14.064814,2.8450174)
\usefont{T1}{ppl}{m}{n}
\rput{-1.1497928}(-0.051011056,0.28923845){\rput(14.379354,2.688517){\large 3}}
\psframe[linewidth=0.04,dimen=outer](15.26,6.41)(0.0,-6.41)
\end{pspicture}
}\end{center}\label{biyectionschrocol}\caption{Biyection between Shr\"oder trees and admissible colorations in a pointed graph.}
\end{figure}
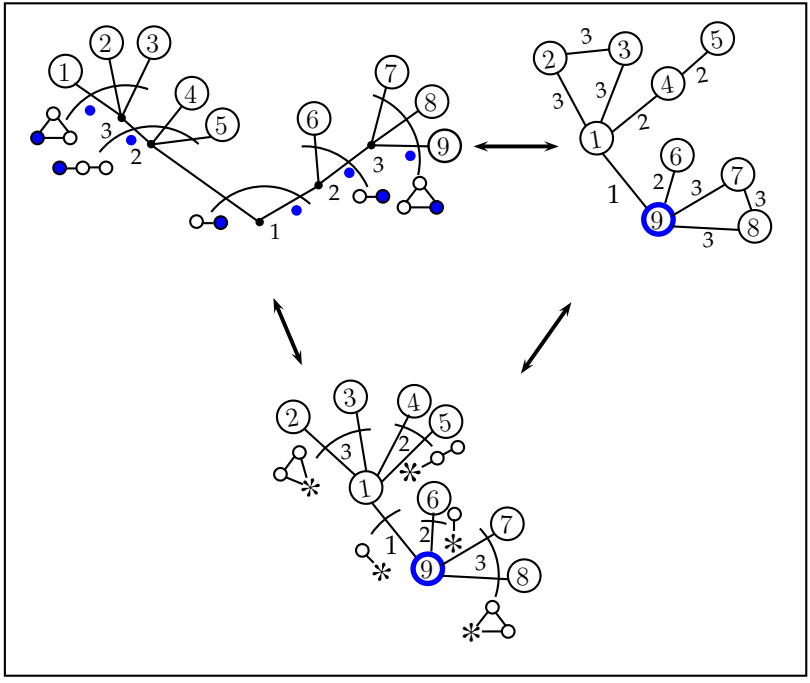

Before studying the antipode for $\CN_{\Gr}$, we introduce the general notion of divisibility on assemblies of $M$-structures, $M$ being a set operad.
 A set-operad $(M,\eta)$ is called a c-operad (cancellative operad) if it satisfies the left cancellation law \cite{PhdMendez, MendezandYang,MendezKoszulduality},
$$\eta(a,m)=\eta(a,m')\Rightarrow m=m'.$$

All the examples of operads treated in this article are c-operads.
The natural transformation $\eta:M(M)\rightarrow M$ can be extended by substitution with the identity of $E$, to $$\overline{\eta}:E(N(N))\rightarrow E(N).$$  The elements of $E(M(M))$ are pairs of assemblies of $M$-structures $(a_2,a'_2)$, where $a_2$ is an assembly over the partition subjacent  to $a_1$.
\begin{defi}
\label{divisibilidad} Let $a_1$ and $a_2$ be two assemblies of $M$-structures over the same set $U$. We say that $a_1$ {\em divides} $a_2$ if there exists a third one $a'_2$ such that
$$\overline{\eta}(a_1,a'_2)=a_2.$$ 
\noindent We denote this divisibility relation by $a_1\preceq_{\eta} a_2$, and by $a_2/a_1$ the assembly $a'_2$. It is unique because of the left cancellation law. The relation $\preceq_{\eta}$ is a partial order over the set $E(M)[U]$. We denote this poset by $P_M[U]$  \cite{MendezandYang, PhdMendez}.
\end{defi} 

\subsection{The Hopf algebra \texorpdfstring{$\Gr$}{Gr}}
The notion of module of a graph \cite{MODULE} plays an important role in the Hopf algebra structure of $\CN_{\Gr}$.
\begin{defi} Let $g$ be a connected graph, and $B$ a subset  the set $U$ of vertices of $g$. The graph $g_B$ generated in $g$ by $B$ is said to be a {\em module} of $g$ if every vertex $v$ in $U-B$, if connected to some vertex in $B$, it is connected to every vertex in $B$. 
\end{defi}
Let $g_1$ be the assembly of connected graphs formed by the union of $g_B$ and the singleton graphs corresponding to the vertices in $U-B$.
It is easy to see that $g_B$ is a module of $g$ if and only if $g_1$ divides $g$. The quotient $g/g_1$ is the graph obtained by contracting in $g$ the block $B$ to a point, 
\begin{itemize}\item The vertices  of $g/g_1$ are the blocks of the partition whose only not singleton block is $B$. \item The edges are either pairs of the form $\{\{x\},\{y\}\}$  with $x,y\in U-B$ and $\{x,y\}\in \mathscr{E}(g)$, or of the form $\{B,\{y\}\}$ with $\{x,y\}\in \mathscr{E}(g)$ for some (all) $x\in B$.
\end{itemize}
The prime elements of $\Gr$ are the connected graphs that do not have modules except the trivial ones. The primitive elements of $\CN_{\Gr}$ are hence completely classified.

\begin{defi}Let $c$ be a coloration of the edges of a connected graph $g$ is  a function from the edges of $g$ to $\PP$. For $i>1$ in the image of $c$, we denote by $g(i)$ the subgraph of $g$ obtained by deleting all the edges of color $j$, $1\leq j<i$. For $i=1$, we make $g(1)=g$.\\
The coloration $c$ is said to be {\em admissible} if it satisfies
\begin{enumerate}
\item \label{dosgr}For every $i\leq 2$ in the image of $c$, $g(i-1)$ divides $g(i)$.
\item \label{unogr}For every vertex $v$ of $g$, if there is an edge incident to $v$ colored $i$, $i\geq 2$, then there is at least one edge  of color $i-1$ incident with $v$.
\end{enumerate}
\end{defi}
\begin{prop}
The antipode of $\CN_{\Gr}$ is given by
\be S(g)=\sum_{c\in \mathrm{Ad}(g)}\prod_i(-1)^{|g(i-1)/g(i)|^{\#} }t_{\tau(g(i-1)/g(i))}.
\eeq
\noindent where $\mathrm{Ad}(g)$ is the set of edge admissible coloration of $g$, and $| g(i-1)/g(i)|^{\#}$ denotes the number of non-singleton connected components of $g(i-1)/g(i)$.
\end{prop}
\begin{proof}
We are going to codify the trees in the general formula (\ref{antipode}) into admissible colorations. For a Shr\"oder tree $\TT$ such that $\widehat{\eta}(\TT)=g$, we color $i$ the edges obtained when applying the products on $\TT$ corresponding to vertices of depth $i$. The coloration $c$ so obtained satisfies condition \ref{dosgr} because the non-singleton connected components of $g(i)$ are the graphs $\widehat{\eta}(\TT_v)$, where $v$ runs over the internal vertices of $\TT$ of depth $i$, and $\TT_v$ is the sub-tree of the descendants of $v$. By the recursive definition of $\widehat{\eta}$, $$g(i-1)=\overline{\eta}(g(i),\{g_v|d(v)=i-1\}),$$ \noindent $\{g_v|d(v)=i-1\}$ being the graphs enriching $\TT$ at depth $i-1$. Condition \ref{unogr} is satisfied because each internal vertex of depth $i$, $i\geq 2$, is the son of a vertex of depth $i-1$ which has at least two sons.  The reverse construction is as follows. Place at the root of the tree the graph $g(1)/g(2)$, then create $|g(2)|^{\#}$ vertices and edges connecting them with the root. Use the non-singleton connected graphs of $g(2)/g(3)$, to enrich the newly created $| g(2)|^{\#}=|g(2)/g(3)|^{\#}$ vertices with the non-singleton connected components of $g(2)/g(3)$. This procedure continues recursively on each of those vertices as roots until we reach $g(k)$, $k$ being the maximum of the image of $c$, and finish by placing the non-singleton components of $g(k)$. Condition \ref{unogr} assures that all the internal vertices of the tree obtained in this way have at least two descendants.
\end{proof}
For example, the antipode in formula (\ref{casikcuatro}) is readily computed from the admissible colorations in Fig. \ref{coloracionesgr}.
\begin{figure}
\begin{center}






\scalebox{1} 

{

\begin{pspicture}(0,-0.803125)(7.603802,0.803125)

\psdots[dotsize=0.12](2.343948,0.4748996)

\psdots[dotsize=0.12](3.1834223,0.4748996)

\psdots[dotsize=0.12](2.343948,-0.48622116)

\psdots[dotsize=0.12](3.1834223,-0.48622116)

\psline[linewidth=0.04cm](2.343948,0.41836306)(3.1414487,-0.42968467)

\psline[linewidth=0.04cm](3.1834223,0.36182654)(3.1834223,-0.3731481)

\psline[linewidth=0.04cm](2.3859217,0.4748996)(3.0994751,0.4748996)

\psline[linewidth=0.04cm](2.343948,0.36182654)(2.343948,-0.42968467)

\psline[linewidth=0.04cm](2.3859217,-0.48622116)(3.0994751,-0.48622116)

\usefont{T1}{ptm}{m}{n}

\rput(2.747934,0.631875){\footnotesize 2}

\usefont{T1}{ptm}{m}{n}

\rput(2.2279341,-0.048125){\footnotesize 2}

\usefont{T1}{ptm}{m}{n}

\rput(3.2596529,0.031875){\footnotesize 1}

\usefont{T1}{ptm}{m}{n}

\rput(2.7276735,-0.658125){\footnotesize 1}

\usefont{T1}{ptm}{m}{n}

\rput(2.8076737,0.121875){\footnotesize 1}

\usefont{T1}{ptm}{m}{n}

\rput(7.4845834,0.011875){\footnotesize 1}

\usefont{T1}{ptm}{m}{n}

\rput(6.9445834,-0.668125){\footnotesize 1}

\psdots[dotsize=0.12](6.5630255,0.46663132)

\psdots[dotsize=0.12](7.4025,0.46663132)

\psdots[dotsize=0.12](6.5630255,-0.4944894)

\psdots[dotsize=0.12](7.4025,-0.4944894)

\psline[linewidth=0.04cm](6.5630255,0.4100948)(7.360526,-0.43795294)

\psline[linewidth=0.04cm](7.4025,0.3535583)(7.4025,-0.38141635)

\psline[linewidth=0.04cm](6.6049995,0.46663132)(7.3185525,0.46663132)

\psline[linewidth=0.04cm](6.5630255,0.3535583)(6.5630255,-0.43795294)

\psline[linewidth=0.04cm](6.6049995,-0.4944894)(7.3185525,-0.4944894)

\usefont{T1}{ptm}{m}{n}

\rput(7.0525,0.121875){\footnotesize 2}

\usefont{T1}{ptm}{m}{n}

\rput(6.9042187,0.641875){\footnotesize 1}

\usefont{T1}{ptm}{m}{n}

\rput(6.4245834,-0.068125){\footnotesize 1}

\usefont{T1}{ptm}{m}{n}

\rput(5.3699565,-0.013125){\footnotesize 2}

\usefont{T1}{ptm}{m}{n}

\rput(4.8299565,-0.653125){\footnotesize 2}

\psdots[dotsize=0.12](4.42657,0.46989956)

\psdots[dotsize=0.12](5.2660446,0.46989956)

\psdots[dotsize=0.12](4.42657,-0.49122116)

\psdots[dotsize=0.12](5.2660446,-0.49122116)

\psline[linewidth=0.04cm](4.42657,0.41336307)(5.2240705,-0.43468466)

\psline[linewidth=0.04cm](5.2660446,0.35682654)(5.2660446,-0.3781481)

\psline[linewidth=0.04cm](4.4685435,0.46989956)(5.182097,0.46989956)

\psline[linewidth=0.04cm](4.42657,0.35682654)(4.42657,-0.43468466)

\psline[linewidth=0.04cm](4.4685435,-0.49122116)(5.182097,-0.49122116)

\usefont{T1}{ptm}{m}{n}

\rput(4.85329,0.126875){\footnotesize 1}

\usefont{T1}{ptm}{m}{n}

\rput(4.81329,0.626875){\footnotesize 1}

\usefont{T1}{ptm}{m}{n}

\rput(4.3013105,-0.023125){\footnotesize 1}

\psdots[dotsize=0.12](0.20103984,0.4816313)

\psdots[dotsize=0.12](1.0405142,0.4816313)

\psdots[dotsize=0.12](0.20103984,-0.47948942)

\psdots[dotsize=0.12](1.0405142,-0.47948942)

\psline[linewidth=0.04cm](0.20103984,0.4250948)(0.9985405,-0.42295292)

\psline[linewidth=0.04cm](1.0405142,0.3685583)(1.0405142,-0.36641636)

\psline[linewidth=0.04cm](0.2430136,0.4816313)(0.95656693,0.4816313)

\psline[linewidth=0.04cm](0.20103984,0.3685583)(0.20103984,-0.42295292)

\psline[linewidth=0.04cm](0.2430136,-0.47948942)(0.95656693,-0.47948942)

\usefont{T1}{ptm}{m}{n}

\rput(0.58421874,0.636875){\footnotesize 1}

\usefont{T1}{ptm}{m}{n}

\rput(0.52421874,-0.663125){\footnotesize 1}

\usefont{T1}{ptm}{m}{n}

\rput(1.1442188,-0.043125){\footnotesize 1}

\usefont{T1}{ptm}{m}{n}

\rput(0.62421876,0.156875){\footnotesize 1}

\usefont{T1}{ptm}{m}{n}

\rput(0.04421875,-0.023125){\footnotesize 1}

\end{pspicture} 
}\end{center} \caption{Example of admissible colorations on a graph.}\label{coloracionesgr}
\end{figure}
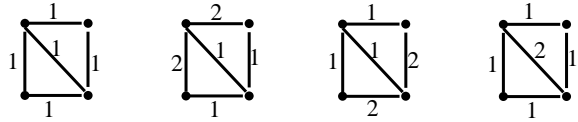

\section{The natural reduced incidence Hopf algebra}
We now make a brief sketch of the construction in \cite{PhdMendez} of the natural reduced incidence Hopf algebra.
When $M$ is a $c$-operad, the natural Hopf algebra can be obtained as a reduced incidence algebra from the family of intervals $\mathrm{Int}_M$ of the family of posets $P_M[U]$, $U$ a finite set. Those intervals are of the form $[a_1,a_2]$, $a_1$ and $a_2$ a pair of assemblies of $M$-structures over the same set. The family $ \mathrm{Int}_M$ is (module poset isomorphism) closed under poset product and by taking subintervals (in the terminology of \cite{Schmitt2}, it is a hereditary family).

We define the equivalence relation $\sim $ in $\mathrm{Int}_M$ by,
\be\label{natural}[a_1,a_2]\sim [a_1', a_2']\mbox{ iff } \tau(a_2/a_1)=\tau(a'_2/a'_1).\eeq 

\noindent Then, every interval in $\mathrm{Int}_M$ is equivalent to one of the form $[\hat{0},a]$. The types of the intervals can be identified with isomorphism types of assemblies of $M$-structures, and hence with monomials $$\widetilde{[\hat{0},a]}=\prod_{B\in \pi} t_{\tau(m_B)}. $$\\ The  equivalence relation $\sim$ is order and poset product compatible (a Hopf relation). According to Schmitt \cite{Schmitt2},  the algebra generated by the types ,  $\KK[t_{\alpha}|\alpha\in \Tt(M)] $ has a structure of bialgebra with coproduct and counity
\begin{eqnarray*}
\Delta (t_{\tau(m)})&=&\Delta \widetilde{[\hat{0},\{m\}]}= \sum_{\hat{0}\preceq_{\eta}a\preceq_{\eta}\{m\}}\widetilde{[\hat{0},a]}\otimes \widetilde{[a,\{m\}]}\\&=&\sum_{\hat{0}\preceq_{\eta}a\preceq_{\eta}\{m\}}\widetilde{[\hat{0},a]}\otimes \widetilde{[\hat{0},\{m\}/a]}=\sum_{\eta(a,m')=m}t_{\tau(a)}\otimes t_{\tau(m')}
\\ \epsilon(\widetilde{[a_1,a_2]})&=&\begin{cases}1&\mbox{ if $a_1=a_2$}\\0&\mbox{ otherwise.}\end{cases}\end{eqnarray*}

\noindent Identifying the type of the singleton intervals with $1$, this reduced incidence algebra becomes a Hopf algebra, which is clearly isomorphic to $\CN_M$. In a similar way, the poset isomorphism relation on $\mathrm{Int}_M$ gives rise to the standard reduced incidence Hopf algebra $\CH_N$ \cite{Chapoton-Livernet}. Since the relation $\sim$ implies poset isomorphism, we have the epimorphism of Hopf algebras
 \begin{eqnarray*}
P:\CN_M&\twoheadrightarrow &\CH_M,\\
P(\widetilde{[\hat{0},a]})&=&\widehat{[\hat{0},a]},
\end{eqnarray*}
  \noindent where  the hat in the right hand side denotes poset isomorphism type.

\subsection{\texorpdfstring{$\CN_{\Arb}$}{NA} and \texorpdfstring{$\HCK$}{HCK}.}

Denote by $\HCK$ the Connes and Kreimer Hopf algebra \cite{ConnesKreimer}.
$\HCK$, like $\CN_{\Arb}$, has as generators unlabelled rooted trees. Its identity $1$ is thought of as the empty tree. We describe its coproduct by:
\be \label{cock}\Delta_{{\mathrm CK}}(t_{\tau(T)})=\sum_{b_v}t_{\tau(a_{b_v}(2))}\otimes t_{\tau(T_{b_v}(1))},
\eeq

\noindent where the sum of above is over all the bicoloring $b_v:V(T)\rightarrow \{1,2\}$ of the vertices of $T$, such that the subgraph generated by the vertices of color $1$, $T_{b_v}(1)$, is a sub-rooted tree of $T$. $a_{b_v}(2)$ is the forest generated by the vertices of color $2$. This description of the coproduct is equivalent to the usual one in terms of admissible cuts.

Let $T_1,T_2,\dots, T_k$ be a forest of trees in $E(\Arb)[U]$ for some set of vertices $U$.
Denote by $T=[T_1,T_2,\dots,T_k]$ the labelled tree obtained by adding an edge from the root of each of the trees in the forest, to a vertex not in $U$, which becomes the root.
Define $B^-(t_{\tau(T)})=t_{\tau(T_1)}t_{\tau(T_2)}\dots t_{\tau(T_k)}$ and $B^+(t_{\tau(T_1)}t_{\tau(T_2)}\dots t_{\tau(T_k)})=t_{\tau(T)}$.
 
In \cite[Theorem 6.11]{Chapoton-Livernet} an elegant proof of the isomorphism between the Hopf algebras $\CH_{\Arb}$ and $\HCK$ was given, based upon the universal property of the pair $(\HCK, B^+)$. On the other hand we have the epimorphism $P:\CN_{\Arb}\rightarrow \CH_{\Arb}$, that composed with Chapoton-Livernet isomorphism gives us the Hopf algebra epimorphism $B^{-}=\hat{P}:\CN_{\Arb}\rightarrow \HCK$. Here we give a direct combinatorial proof of this fact. 

\begin{prop} The multiplicative extension of
$B^-$ is a Hopf algebra epimorphism,
$$B^-:\CN_{\Arb}\twoheadrightarrow \HCK.$$  
\end{prop}
 \begin{proof}
$B^-$ obviously preserves identities and coidentities. Denote by $\Delta$ the coproduct of $\CN_{\Arb}$.  We will prove now that $B^-$ is a comultiplicative map,
$$ B^-\otimes B^-\circ \Delta=\Delta_{{\mathrm CK}}\circ B^-. $$
By multiplicativity of the maps involved, it is enough to prove that for every tree $T$ 
\be\label{coalgebraT}  (B^-\otimes B^- \circ  \Delta)(t_{\tau(T)})
=\Delta_{{\mathrm CK}}\circ B^- (t_{\tau(T)}). \eeq
Since $B^+$ is the inverse of $B^-$ when restricted to the vector space generated by the monomials $t_{\tau(T)}$, Eq. (\ref{coalgebraT}) is equivalent to
\be\label{coalgebraT1}  (B^-\otimes B^-\circ \Delta\circ B^+) (t_{\tau(T)})=\Delta_{{\mathrm CK}}(t_{\tau(T)}).\eeq
 Observe that $[T]$ is the planted tree obtained by attaching to $T$ an extra vertex as root. Since $B^+ t_{\tau (T)}=t_{\tau([T])}$, the right hand side of Eq. (\ref{coalgebraT1}) is equal to 
$$(B^-\otimes B^-\circ\Delta)t_{\tau[T]}=\sum_{b_e:\mathscr{E}([T])\rightarrow \{1,2\}} B^- t_{\tau(a_{b_e}(2))}\otimes B^- t_{\tau([T]_{b_e}(1))}.$$
\noindent The above sum is over the edge bicolorings $b_e$
of $[T]$ as in the Eq. (\ref{coproductnap}). There is a bijection between edge bicolorings of $[T]$ and  vertex bicolorings of $T$, $b_e\leftrightarrow b_v$, $b_v$ obtained from $b_e$ by `moving up' the color of each edge to  its endpoint vertex furthest from the root. We can easily check
$$B^-t_{\tau (a_{b_e}(2))}=t_{\tau (a_{b_v}(2))}\mbox{, and }
B^-t_{\tau[T]_{b_e}(1)}=t_{\tau (T_{b_v}(1))},$$
\noindent which finishes the proof. 

\end{proof}
\bibliographystyle{amsplain}

\providecommand{\bysame}{\leavevmode\hbox to3em{\hrulefill}\thinspace}
\providecommand{\MR}{\relax\ifhmode\unskip\space\fi MR }
\providecommand{\MRhref}[2]{%
  \href{http://www.ams.org/mathscinet-getitem?mr=#1}{#2}
}
\providecommand{\href}[2]{#2}

\end{document}